\documentclass[a4paper,10pt]{smfart}
\usepackage[french,english]{babel}
\usepackage{smfthm}
\usepackage{bull}

\usepackage[utf8]{inputenc}
\usepackage{amsmath}
\usepackage{array}
\usepackage{amsthm}
\usepackage{amssymb}
\input xy
\xyoption{all}

\newcommand{\A}{{\mathcal{A}}}
\newcommand{\B}{{\mathcal{B}}}
\newcommand{\C}{{\mathcal{C}}}
\newcommand{\F}{{\mathcal{F}}}

\newcommand{\pol}{{\mathcal{P}ol}}
\newcommand{\E}{{\mathcal{E}}}
\newcommand{\I}{{\mathcal{I}}}
\newcommand{\Q}{{\operatorname{Q}}}
\newcommand{\Qse}{{\mathcal{Q}}}
\newcommand{\LL}{{\mathcal{L}}}
\newcommand{\M}{{\mathcal{M}}}
\newcommand{\N}{{\mathcal{N}}}
\newcommand{\col}{{\rm colim}\,}
\newcommand{\Md}{\text{-}\mathbf{Mod}}
\newcommand{\FF}{\mathbb{F}}
\newcommand{\GL}{\operatorname{GL}}
\newcommand{\cd}{\operatorname{Card}}
\newcommand{\op}{\mathrm{op}}
\newcommand{\rep}{\mathbf{Rep}}
\newcommand{\lgr}{\ell}
\newcommand{\db}{\bar{\delta}}
\newcommand{\tb}{\bar{\tau}}
\newcommand{\rg}{\operatorname{rg}}
\newcommand{\rad}{\operatorname{Rad}}
\newcommand{\soc}{\operatorname{Soc}}
\newcommand{\surj}{\operatorname{Surj}}
\newcommand{\red}{\mathrm{red}}
\newcommand{\R}{\operatorname{R}}
\newcommand{\irr}{\operatorname{Irr}}

\author[A. DJAMENT]{Aur\'elien DJAMENT}
\address{CNRS, laboratoire Analyse, Géométrie et Applications (UMR7539)\\ Institut Galilée\\ Université Sorbonne Paris Nord\\ 99 avenue Jean-Baptiste Clément\\ 93430 VILLETANEUSE\\ FRANCE}
\email{djament@math.cnrs.fr}
\urladdr{https://djament.perso.math.cnrs.fr/}

\author[T. GAUJAL]{Thomas GAUJAL}
\address{Université de Lille\\ laboratoire Paul Painlev\'e (UMR8524)\\ Cité scientifique, bât. M2\\ 59655 VILLENEUVE D'ASCQ CEDEX\\ FRANCE}
\email{thomas.gaujal@gmail.com}

\title[Représentations génériques en inégale caractéristique]{Représentations génériques des groupes linéaires finis en inégale caractéristique}

\alttitle{Generic representations of finite general linear groups in nondescribing characteristic}

\newtheorem{thi}{Th\'eor\`eme}
\newtheorem{dfii}{D\'efinition}

\begin{document}
\newenvironment{nota}{\begin{enonce}{Notation}}{\end{enonce}}
\newenvironment{hyp}{\begin{enonce}{Hypoth\`ese}}{\end{enonce}}

\frontmatter

\begin{abstract}
On étudie les \emph{représentations génériques} des groupes linéaires sur un anneau fini $R$ à coefficients dans un corps $k$ dans lequel le cardinal de $R$ est inversible, c'est-à-dire les foncteurs depuis les $R$-modules projectifs de type fini vers les $k$-espaces vectoriels. On obtient en particulier une classification de ces représentations \emph{simples} qui permet de démontrer une conjecture de Djament-Touzé-Vespa sur les dimensions prises par un tel foncteur.
\end{abstract}

\begin{altabstract}
We study \emph{generic representations} of general linear groups over a finite ring $R$ with coefficients in a field $k$ in which the cardinality of $R$ is invertible, that is functors from finitely-generated projective $R$-modules to $k$-vector spaces. We obtain especially a classification of such \emph{simple} representations, what allows to prove a conjecture of Djament-Touzé-Vespa on dimensions taken by such a functor.
\end{altabstract}

\subjclass{18A25, 18E05, 18E10, 18E35, 19A99, 20G99, 20J99.}

\keywords{catégories de foncteurs, catégories additives, foncteurs antipolynomiaux, foncteurs simples, recollements de catégories abéliennes.}

\altkeywords{functor categories, additive categories, antipolynomial functors, simple functors, recollements of abelian categories.}

\thanks{Les auteurs ont bénéficié du soutien partiel de l’Agence Nationale de la Recherche, via le projet ANR ChroK (ANR-16-CE40-0003), le Labex CEMPI (ANR-11-LABX-0007-01), et, pour le premier auteur, le projet ANR AlMaRe (ANR-19-CE40-0001-01).}

\maketitle

%\tableofcontents

\mainmatter

\section*{Introduction}

\paragraph*{Notations} Dans tout l'article, la lettre $\A$ désigne une catégorie additive essentiellement petite. On travaille sur un corps commutatif $k$ (en particulier, les produits tensoriels de base non spécifiée sont pris sur $k$).

On note $k[-]$ le foncteur de $k$-linéarisation des ensembles vers les $k$-espaces vectoriels. Ainsi, si $E$ est un ensemble, le $k$-espace vectoriel $k[E]$ a une base canonique qui s'identifie à $E$ ; pour $e\in E$, on notera $[e]\in k[E]$ l'élément correspondant de cette base.

\paragraph*{Prélude : représentations des groupes linéaires finis en inégale caractéristique} Soit $R$ un anneau  de cardinal fini $N$. Il est classique que la théorie des représentations $k$-linéaires des groupes linéaires $\GL_n(R)$ est totalement différente dans les deux situations suivantes.
\begin{enumerate}
    \item \textbf{Égale caractéristique :} \emph{$N$ est une puissance de la caractéristique $p$ de $k$}. Toutes ces représentations sont alors \emph{polynomiales} ; si $J$ désigne le radical de $R$, le morphisme de groupes $\GL_n(R)\to\GL_n(R/J)$ induit par la réduction modulo $J$ est surjectif et son noyau est un $p$-groupe, de sorte que les représentations $k$-linéaires irréductibles de $\GL_n(R)$ sont les mêmes que celles de $\GL_n(R/J)$. Ces représentations sont classifiées (au moins si $k$ est assez gros), via la théorie des représentations des groupes algébriques par exemple.
    \item \textbf{Inégale caractéristique :} \emph{$N$ est inversible dans $k$}. Les représentations $k$-linéaires irréductibles de $\GL_n(R)$ sont bien comprises si $R$ est semi-simple, mais ce n'est plus le cas sinon, du moins pour $n$ arbitraire. Ainsi, pour chaque nombre premier $l$, la classification des représentations complexes irréductibles de tous les groupes $\GL_n(\mathbb{Z}/l^2\mathbb{Z})$ constitue un problème sauvage.
    \end{enumerate}
Cette dichotomie possède un analogue pour les \emph{représentations génériques} des groupes linéaires sur $R$ à coefficients dans $k$, c'est-à-dire les objets de catégories de foncteurs dont nous allons rappeler ci-dessous la définition et l'origine. Le présent article est consacré à ces représentations génériques dans le cas d'inégale caractéristique.

\paragraph*{Représentations des petites catégories} Suivant l'article classique de B. Mitchell \cite{Mi72}, on peut considérer les petites catégories préadditives comme des anneaux à plusieurs objets et les foncteurs additifs de ces catégories vers la catégorie $\mathbf{Ab}$ des groupes abéliens comme des modules sur celles-ci. Si $\C$ est une petite catégorie (non nécessairement préadditive), un foncteur de $\C$ vers les $k$-espaces vectoriels s'identifie à un foncteur additif de la catégorie préadditive $k[\C]$, obtenue par linéarisation des morphismes de $\C$, vers $\mathbf{Ab}$, ou encore à une représentation $k$-linéaire de $\C$. L'étude de ces foncteurs, dont la catégorie sera notée $\F(\C;k)$, constitue de fait une généralisation naturelle des représentations linéaires des groupes, et possède des interactions fécondes avec ces dernières.

\paragraph*{Représentations des catégories additives} Un cas important est celui des foncteurs depuis une source \emph{additive} $\A$. On peut alors distinguer les foncteurs additifs (notamment étudiés par Auslander \cite{Aus74}, lorsque $\A$ possède suffisamment de propriétés de finitude), qui constituent une sous-catégorie remarquable de la catégorie de tous les foncteurs, en général beaucoup plus difficile à comprendre. Si $R$ est un anneau, on nomme depuis la série d'articles de N. Kuhn \cite{Ku1,Ku2,Ku3} (consacrée au cas où $R=k$ est un corps fini) \emph{représentations génériques des groupes linéaires} sur $R$ à coefficients dans $k$ les foncteurs (non nécessairement additifs) de la catégorie des $R$-modules à gauche projectifs de type fini vers les $k$-modules. Ces représentations génériques forment une catégorie abélienne aux bonnes propriétés notée $\F(R,k)$. Elles interviennent notamment dans des questions de topologie algébrique (notamment dans le cas susmentionné où $R=k$ est un corps fini --- cf. \cite{HLS}) et d'homologie des groupes --- cf. par exemple le survol \cite{fct-surv}.

\paragraph*{Foncteurs polynomiaux et antipolynomiaux} Fondamentale dans l'étude des foncteurs d'une catégorie additive vers une catégorie abélienne et dans leurs applications, la notion de \emph{foncteur polynomial} introduite par Eilenberg et Mac Lane \cite{EML} constitue une généralisation de la notion de foncteur additif. Il existe toutefois des situations non triviales où il n'existe pas de foncteur polynomial non constant (voir la proposition~\ref{pr-pfpnc}), en particulier la situation \emph{antipolynomiale} dont on rappelle maintenant la définition :
\begin{dfii}[Cf. définition~\ref{df-ktriv}] Soit $\A$ une catégorie additive.
\begin{enumerate}
    \item On dit que $\A$ est \emph{$k$-triviale} si, pour tous objets $x$ et $y$ de $\A$, le groupe abélien $\A(x,y)$ est d'ordre fini et inversible dans $k$.
    \item Un foncteur de $\F(\A;k)$ est dit \emph{antipolynomial} s'il se factorise à travers un foncteur additif de $\A$ vers une catégorie additive $k$-triviale.
\end{enumerate}
\end{dfii}
 De plus, on peut ramener l'étude des foncteurs de $\F(\A;k)$ possédant assez de propriétés de finitude à celles de foncteurs polynomiaux et de foncteurs antipolynomiaux en un sens précisé dans \cite[§\,4]{DTV} (qui motive l'introduction de la notion d'antipolynomialité) et brièvement rappelé dans la section~\ref{sfpap}.

Si $R$ est un anneau fini, la situation d'égale caractéristique est exactement celle où tous les foncteurs de longueur finie de $\F(R,k)$ sont polynomiaux, tandis que la situation d'inégale caractéristique (on dit aussi que $R$ est \emph{$k$-trivial}) est exactement celle où tous les foncteurs de $\F(R,k)$ sont antipolynomiaux.

Si les foncteurs polynomiaux ont fait l'objet de nombreux travaux et possèdent une structure qui est assez bien comprise (voir par exemple Pirashvili \cite{Pira88} ou, plus récemment, \cite[§\,5]{DTV}, parmi un grand nombre d'autres articles), les foncteurs antipolynomiaux, auxquels notre article est consacré, ont reçu fort peu d'attention jusqu'à présent.

\paragraph*{Précédents et inspirations}
En 2015, Kuhn \cite{Ku-adv} a démontré, à partir de résultats de Kov\'acs \cite{Kov} sur les représentations de monoïdes finis de matrices, le théorème suivant :
\begin{thi}[Kuhn]\label{thiK} Si $\FF$ est un corps fini de caractéristique n'ayant pas la même caractéristique que $k$, il existe une équivalence de catégories
$$\F(\FF,k)\simeq\prod_{n\in\mathbb{N}}k[\GL_n(\FF)]\Md\;.$$
\end{thi}

Ce résultat contraste fortement avec la situation d'égale caractéristique : il implique en particulier que la catégorie $\F(\FF,k)$ est \emph{localement finie} (c'est-à-dire engendrée par des objets de longueur finie), alors que $\F(k,k)$ n'est jamais localement finie.

Le présent article est issu en grande partie de la thèse de doctorat \cite{these} du second auteur, menée sous la direction du premier auteur. Cette thèse contient également une nouvelle démonstration, purement fonctorielle (sans recours aux résultats de Kov\'acs), du théorème~\ref{thiK} de Kuhn, qui n'apparaît pas ici, tandis que les résultats des sections~\ref{sfl} et~\ref{shom} de notre article ne figurent pas dans \cite{these}.

Les travaux récents de Nagpal \cite{Nag1,Nag2} nous ont inspirés : ils sont consacrés à une situation assez similaire à celle de Kuhn \cite{Ku-adv}, mais qui donne lieu à une structure plus subtile, à savoir l'étude des foncteurs depuis la catégorie des \emph{monomorphismes} des espaces vectoriels de dimension finie sur un corps fini $\FF$ vers les $k$-espaces vectoriels lorsque $k$ et $\FF$ sont de caractéristiques distinctes. Nagpal \cite[§\,4.2]{Nag1} a en particulier introduit des variantes des outils classiques que sont les foncteurs de \emph{décalage} et de \emph{différence} (dont la définition est rappelée au début de la section~\ref{sfpap}) qui constituent en quelque sorte des formes fonctorielles de la \emph{restriction parabolique} en théorie des représentations des groupes linéaires (les foncteurs de décalage usuels étant des analogues de la restriction classique). La définitions de ces foncteurs fait également sens pour les foncteurs de source additive, même si leurs propriétés formelles ne sont pas tout à fait les mêmes que celles des foncteurs de Nagpal. Nous étudions ces foncteurs, que nous notons $\tb_x$ et $\db_x$, dans le contexte d'une source additive $k$-triviale (dont $x$ est ici un objet), à la section~\ref{sdecp} ; ils constituent un ingrédient crucial de notre travail.

Une autre source d'inspiration provient de l'article \cite{GP-cow} de G. Powell, qui introduit, dans un cas particulier d'égale caractéristique (la catégorie $\F(k,k)$, où $k$ est un corps fini), des foncteurs fondamentaux, qu'il nomme foncteurs \emph{co-Weyl}, dont la définition se généralise sans difficulté à une source additive arbitraire, et qui possèdent des propriétés spécifiques (notamment cohomologiques) dans la situation d'inégale caractéristique que nous considérons (voir les sections~\ref{sQA} et~\ref{sQAM}).

\medskip
 Le théorème~\ref{thiK} montre en particulier que $\F(\FF,k)$ est une catégorie semi-simple lorsque $k$ est un corps de caractéristique nulle et $\FF$ un corps fini. On voit en revanche facilement que la catégorie $\F(R,k)$ ne peut pas être semi-simple si $R$ est un anneau non semi-simple (cf. l'exemple~\ref{ex-Qsom}). En particulier, on ne peut pas s'attendre à une structure aussi simple pour les foncteurs d'une catégorie additive $k$-triviale arbitraire vers les $k$-espaces vectoriels qu'une décomposition telle que celle du théorème~\ref{thiK}. Plutôt qu'une décomposition en produit direct, on obtient en général une stratification de la catégorie de foncteurs par des sous-catégories remarquables.

\paragraph*{Résultats principaux}
Dans la section~\ref{sFd}, nous définissons, lorsque la catégorie additive $\A$ est $k$-triviale, des sous-catégories bilocalisantes\,\footnote{La définition de sous-catégorie \emph{bilocalisante}, ainsi que d'autres notions classiques dans les catégories abéliennes utiles dans cet article, est rappelée dans l'appendice~\ref{ap-catab}.} $\F_d(\A;k)$ de $\F(\A;k)$ (où $d$ est un entier) telles que :
\begin{enumerate}
    \item $\F_d(\A;k)\subset\F_{d+1}(\A;k)$ ; $\F_d(\A;k)$ est réduite à $0$ (resp. aux foncteurs constants) pour $d<0$ (resp. $d=0$) ;
    \item si la catégorie des foncteurs \emph{additifs} de $\A$ vers $\mathbf{Ab}$ est localement finie, alors tout foncteur de type fini de $\F(\A;k)$ appartient à $\F_d(\A;k)$ pour un certain entier $d$ (cela s'applique en particulier à $\F(R,k)$, où $R$ est un anneau $k$-trivial) ;
    \item si $A : \A\to\mathbf{Ab}$ est un foncteur additif, alors le foncteur linéarisé $k[A]$ appartient à $\F_d(\A;k)$ si et seulement si $A$ est de longueur finie au plus égale à $d$.
\end{enumerate}
La définition explicite des sous-catégories $\F_d(\A;k)$ s'inspire d'une des définitions des foncteurs polynomiaux et fait intervenir les foncteurs de décalage parabolique susmentionnés.

Nous démontrons dans la section~\ref{sec-ThStruc} le théorème de structure suivant, qui constitue un analogue du théorème~\ref{thiK} pour une source beaucoup plus générale :
\begin{thi}\label{thif} Supposons que la catégorie additive $\A$ est $k$-triviale. Pour tout $d\in\mathbb{N}$, il existe une équivalence de catégories
$$\F_d(\A;k)/\F_{d-1}(\A;k)\simeq\prod k[\mathrm{Aut}(A)]\Md$$
où le produit est pris sur les classes d'isomorphisme de foncteurs additifs $A : \A\to\mathbf{Ab}$ de longueur finie $d$.
\end{thi}
On montre en fait au théorème~\ref{th-princ} un résultat plus précis, qui décrit explicitement l'équivalence de catégories précédente ainsi que les foncteurs section et co-section associés.

Une conséquence importante du théorème précédent est la description des foncteurs simples des catégories $\F_d(\A;k)$ à partir des représentations $k$-linéaires irréductibles des groupes $\mathrm{Aut}(A)$, où $A$ est un foncteur additif de longueur finie sur $\A$. Sous une légère hypothèse de finitude supplémentaire sur $\A$, cela décrit tous les foncteurs simples de $\F(\A;k)$. On peut ainsi démontrer, au théorème~\ref{th-fct_dim}, une conjecture émise dans \cite{DTV} sur les dimensions prises par les foncteurs de type fini de $\F(R,k)$, où $R$ est un anneau $k$-trivial.

Un autre corollaire important est le résultat de finitude suivant :
\begin{thi}[Proposition~\ref{pr-lin_fini}] Supposons que la catégorie additive $\A$ est $k$-triviale. Soit $A : \A\to\mathbf{Ab}$ un foncteur additif de longueur finie. Alors le foncteur $k[A]$ de $\F(\A;k)$ est de longueur finie.
\end{thi}

Nous explorons également, dans la section~\ref{shom}, des conséquences homologiques du théorème~\ref{thif} : nous montrons que les groupes d'extensions entre deux foncteurs de $\F_d(\A;k)$ sont les mêmes calculés dans cette catégorie et dans $\F(\A;k)$ (corollaire~\ref{cor-plg_hom}) et que, lorsque $k$ est de caractéristique nulle, les catégories $\F_d(\A;k)$ sont de dimension globale finie (proposition~\ref{pr-gldim}).

\section{Représentations des catégories additives}

\paragraph*{Représentations additives et non additives}

On note $\rep(\A)$ la catégorie des foncteurs additifs de $\A$ vers les groupes abéliens, les morphismes étant les transformations naturelles. Les objets de $\rep(\A)$ sont souvent appelés \emph{représentations additives} de la catégorie $\A$, ou $\A$-modules à gauche (cf. Mitchell \cite{Mi72}).

On note $\F(\A;k)$ la catégorie des foncteurs (non nécessairement additifs) de $\A$ vers la catégorie des $k$-espaces vectoriels. Si $R$ est un anneau, on note $\mathbf{P}(R)$ la catégorie des $R$-modules à gauche projectifs de type fini et $\F(R,k):=\F(\mathbf{P}(R);k)$ la catégorie des représentations génériques des groupes linéaires sur $R$ à coefficients dans $k$. On notera que la catégorie $\rep(\mathbf{P}(R))$ est équivalente à la catégorie des $R$-modules à \emph{droite} (l'équivalence étant donnée dans une direction par l'évaluation en $R$ et dans l'autre par la tensorisation au-dessus de $R$).

Les catégories $\rep(\A)$ et $\F(\A;k)$ sont des catégories abéliennes (où l'exactitude se teste au but) aux bonnes propriétés : ce sont des catégories de Grothendieck \cite{Pop}, et elles possèdent assez d'objets projectifs. Plus précisément, elles sont respectivement engendrées par les foncteurs représentables $\A(a,-)$ et $k[\A(a,-)]$ (où $a$ parcourt un squelette de $\A$), qui sont projectifs de type fini, car représentant l'évaluation en $a$ grâce au lemme de Yoneda.

Tout foncteur $F$ de $\F(\A;k)$ se scinde naturellement en somme directe de son terme constant $F(0)$ et d'un foncteur réduit, c'est-à-dire nul en $0$, noté $F^\red$.

\paragraph*{Foncteurs de type fini}

Il résulte de ce qui précède que les catégories abéliennes $\rep(\A)$ et $\F(\A;k)$ sont \emph{localement de type fini} (i.e. engendrées par des objets de type fini\,\footnote{La notion d'objet de type fini ainsi que les autres propriétés de finitude dans les catégories abéliennes qui apparaîtront dans l'article sont rappelées dans l'appendice~\ref{ap-catab}.}). Comme $\A$ est additive, un objet de $\rep(\A)$ est de type fini si et seulement s'il est quotient d'un foncteur représentable $\A(a,-)$. 

\begin{nota}\label{noteng}
Étant donné des objets $a$ et $A$ de $\A$ et $\rep(\A)$ respectivement et un élément $\xi$ de $A(a)$, on note $A_\xi$ le sous-foncteur de $A$ engendré par $\xi$, c'est-à-dire l'image du morphisme $\A(a,-)\to A$ associé à $\xi\in A(a)\simeq\mathrm{Hom}(\A(a,-),A)$.
\end{nota}

Ainsi, un foncteur $A$ de $\rep(\A)$ est de type fini si et seulement s'il existe $a\in\mathrm{Ob}\,\A$ et $\xi\in A(a)$ tels que $A_\xi=A$.

Un foncteur de $\F(\A;k)$ est de type fini si et seulement s'il est isomorphe à un quotient d'une somme directe \emph{finie} de foncteurs de la forme $k[\A(a_i,-)]$ pour des objets $a_i$ de $\A$.
\begin{prop}\label{pr-lin_tf} Si $A$ est un foncteur de type fini de $\rep(\A)$, alors $k[A]$ est un foncteur de type fini de $\F(\A;k)$.
\end{prop}

\begin{proof}
Cela résulte de la préservation des épimorphismes par le foncteur $k[-]$ et des observations précédentes.
\end{proof}

\paragraph*{Dualités}
On note $(-)^\sharp$ le foncteur $\mathrm{Hom}_\mathbb{Z}(-,\mathbb{Q}/\mathbb{Z}) : \mathbf{Ab}^\op\to\mathbf{Ab}$, qui induit une équivalence entre la sous-catégorie pleine des groupes abéliens finis et sa catégorie opposée. On note encore $(-)^\sharp : \rep(\A)^\op\to\rep(\A^\op)$ la post-composition par ce foncteur. Elle induit donc une équivalence involutive entre les sous-catégories pleines de foncteurs additifs à valeurs finies.

On note $(-)^\vee:=\mathrm{Hom}(-,k)$ le foncteur de dualité des $k$-espaces vectoriels. On note encore $(-)^\vee : \F(\A;k)^\op\to\F(\A^\op;k)$  la post-composition par le foncteur précédent. Elle induit une équivalence involutive entre les sous-catégories pleines de foncteurs à valeurs de dimensions finies.

Si $A$ (resp. $F$) est un foncteur de $\rep(\A)$ (resp. $\F(\A;k)$) prenant des valeurs finies (resp. de dimensions finies sur $k$), alors $A$ (resp. $F$) est de type fini si et seulement si $A^\sharp$ (resp. $F^\vee$) est de type cofini, noethérien si et seulement si $A^\sharp$ (resp. $F^\vee$) est artinien, fini si et seulement si $A^\sharp$ (resp. $F^\vee$) est fini.

\paragraph*{Condition de finitude sur les homomorphismes (FH)}
Considérons la condition suivante sur la catégorie additive $\A$ :

\noindent
(FH)\;\;\; Pour tous objets $a$ et $b$ de $\A$, le groupe abélien $\A(a,b)$ est fini.

Celle-ci garantit que tous les foncteurs de type fini de $\rep(\A)$ (resp. $\F(\A;k)$) sont à valeurs finies (resp. de dimensions finies). On peut donc alors utiliser les foncteurs de dualité précédents, et travailler indifféremment avec les foncteurs de type fini de $\F(\A;k)$ ou les foncteurs de type cofini de $\F(\A^\op;k)$, par exemple.

La condition (FH) implique aussi que les anneaux (resp. $k$-algèbres) d'endomorphismes des objets de type fini de $\rep(\A)$ (resp. $\F(\A;k)$) sont finis (resp. de dimension finie) ; en particulier, la proposition~\ref{sousobjets-nbfini} s'applique à tous les objets finis de $\rep(\A)$. La condition (FH) prémunit également contre un certain nombre de difficultés --- par exemple, $\F(\A;k)$ ne peut être localement noethérienne que si (FH) est vérifiée \cite[prop.~11.1]{DTV}.

\section{Foncteurs polynomiaux et antipolynomiaux}\label{sfpap}

\paragraph*{Foncteurs de décalage et de différence}

Si $a$ est un objet de $\A$, on note $\tau_a$ l'endofoncteur de $\F(\A;k)$ donné par la précomposition par $a\oplus - : \A\to\A$, qu'on appelle \emph{décalage par $a$}. Comme tout foncteur de précomposition il est bicontinu, et en particulier exact. On note de plus que $a\mapsto\tau_a$ définit un foncteur de $\A$ vers les endofoncteurs de $\F(\A;k)$. En particulier, les morphismes $0\to a\to 0$ induisent des transformations naturelles $\mathrm{Id}\simeq\tau_0\to\tau_a\to\mathrm{Id}$ dont la composée égale l'identité. Le noyau de $\tau_a\to\mathrm{Id}$, noté $\delta_a$ et appelé foncteur de \emph{différence associé à $a$}, donne donc lieu à un scindement naturel $\tau_a\simeq\mathrm{Id}\oplus\delta_a$. Ainsi, $\delta_a$
 est lui aussi un foncteur bicontinu. Comme précédemment, $a\mapsto\delta_a$ définit un foncteur de $\A$ vers les endofoncteurs de $\F(\A;k)$.
 
\paragraph*{Foncteurs polynomiaux}

Étant donné un entier $d\ge -1$, on note $\pol_d(\A;k)$ la sous-catégorie pleine de $\F(A;k)$ des foncteurs $F$ tels que $\delta_{a_0}\dots\delta_{a_d}(F)=0$ pour tout $(a_0,\dots,a_d)\in\mathrm{Ob}\,\A^{d+1}$. Comme les foncteurs de différence sont bicontinus, cette sous-catégorie est \emph{bilocalisante}. Ses objets sont appelés foncteurs polynomiaux de degré au plus $d$.

\begin{rema} La définition classique des foncteurs polynomiaux dans notre contexte remonte aux débuts des années 1950 ; due à Eilenberg et Mac Lane \cite{EML}, elle est donnée en termes d'\emph{effets croisés}. L'équivalence entre ce point de vue et celui ici utilisé en termes de foncteurs de différence, facile, est discutée en détail dans \cite[§\,3.2]{DV-pol}.
\end{rema}

Les foncteurs polynomiaux jouissent de nombreuses propriétés remarquables ; en particulier, on dispose de résultats de classification des foncteurs polynomiaux simples \cite{Pira88,DTV} beaucoup plus précis que pour des foncteurs simples généraux, ou de propriétés de finitude beaucoup plus favorables pour les foncteurs polynomiaux \cite{DT-schw,DT-noeth}.

\paragraph*{Catégories $k$-triviales}
Toutefois, il y a des situations où la théorie des foncteurs polynomiaux n'est d'aucun secours pour comprendre la catégorie $\F(\A;k)$, car il n'existe pas de foncteur polynomial non constant. Précisément, on a le résultat élémentaire suivant :

\begin{prop}\cite[prop.~2.13]{DTV}\label{pr-pfpnc} Les assertions suivantes sont équivalentes :
\begin{enumerate}
    \item la catégorie $\F(\A;k)$ ne possède pas de foncteur polynomial non constant ;
    \item il n'existe pas de foncteur additif non nul dans $\F(\A;k)$ ;
    \item pour tous objets $x$ et $y$ de $\A$, on a $\A(x,y)\otimes_\mathbb{Z}k=0$ ;
    \item pour tout objet $x$ de $\A$, on a $\A(x,x)\otimes_\mathbb{Z}k=0$.
\end{enumerate}
\end{prop}

\begin{defi}\cite[déf.~4.1]{DTV}\label{df-ktriv} La catégorie additive $\A$ est dite \emph{$k$-triviale} si elle vérifie la condition (FH) ainsi que les conditions équivalentes de la proposition précédente.

Un anneau $R$ est dit $k$-trivial si la catégorie $\mathbf{P}(R)$ est $k$-triviale, i.e. si $R$ est fini et de cardinal inversible dans $k$.
\end{defi}

Dans cet article, on supposera le plus souvent (et même toujours à partir du milieu de la section~\ref{sFd}) que la catégorie $\A$ est $k$-triviale.

\paragraph*{Foncteurs antipolynomiaux}
On dit \cite[déf.~4.2]{DTV} qu'un foncteur de $\F(\A;k)$ est \emph{antipolynomial} s'il se factorise, à isomorphisme près, à travers un foncteur additif $\A\to\B$ dont le but $\B$ est une catégorie additive $k$-triviale.

\paragraph*{Décompositions à la Steinberg} L'introduction des foncteurs antipolynomiaux est motivée par le théorème suivant \cite[cor.~4.11]{DTV}, qui constitue l'un des principaux résultats de \cite{DTV} et un analogue fonctoriel de théorèmes de décomposition dus à R. Steinberg :
\begin{theo}[Djament-Touzé-Vespa]
Tout foncteur de type fini et de type cofini de $\F(\A;k)$ dont les valeurs sont des espaces vectoriels de dimensions finies sur $k$ est isomorphe à la composée du foncteur diagonale $\A\to\A\times\A$ et d'un foncteur, unique à isomorphisme près, de $\F(\A\times\A;k)$ qui est polynomial par rapport à la première variable et antipolynomial par rapport à la deuxième variable.
\end{theo}

On en déduit en particulier \cite[cor.~4.13]{DTV} que, si le corps $k$ est algébriquement clos, les foncteurs simples à valeurs de dimensions finies de $\F(\A;k)$ sont exactement les produits tensoriels d'un foncteur simple polynomial à valeurs de dimensions finies et d'un foncteur antipolynomial simple. Si l'on sait dire beaucoup sur la structure des foncteurs polynomiaux simples à valeurs de dimensions finies \cite[§\,5]{DTV}, celle des foncteurs antipolynomiaux simples demeurait essentiellement inconnue, hormis dans le cas d'une source semi-simple \cite{Ku-adv}. 

\section{Linéarisation des foncteurs additifs}

\subsection{Morphismes}

Le principe de linéarisation suivant, qui figure dans \cite[lemme~C.1]{DTV}, est simple mais fondamental.

\begin{prop}\label{pr-morlin} Soient $A$ et $B$ des foncteurs de $\rep(\A)$. L'application $k$-linéaire canonique
$$\Theta_{A,B} : k[\mathrm{Hom}(A,B)]\to\mathrm{Hom}(k[A],k[B])$$
est injective ; elle est bijective si $A$ est de type fini.
\end{prop}

Le résultat suivant relie l'image d'un morphisme $\Theta_{A,B}(\sum_{i=1}^n\lambda_i[f_i])$ (où les $\lambda_i$ sont non nuls et les $f_i$ deux à deux distincts) aux images des $f_i$. Il interviendra dans la démonstration de l'importante proposition~\ref{pr-Fd-tftcf}.

\begin{prop}\label{pr-sqtau} Soient $(A_i)$ et $(B_j)$ des familles d'objets de $\rep(\A)$ et, pour tous $i, j$, $(\lambda_{i,j}(f))_{f\in\mathrm{Hom}(A_i,B_j)}$ une famille presque nulle d'éléments de $k$. On suppose que les $A_i$ sont de type fini. Notons $F$ l'image du morphisme
$$\bigoplus_i k[A_i]\xrightarrow{\big(\Theta_{A_i,B_j}(\sum_{f\in\mathrm{Hom}(A_i,B_j)}\lambda_{i,j}(f)[f])\big)_{i,j}}\prod_j k[B_j]\;.$$
Si $\lambda_{i,j}(f)\in k^\times$, alors il existe un objet $x$ de $\A$ tel que $k[\mathrm{Im}\,f]$ soit un sous-quotient de $\tau_x(F)$.
\end{prop}

\begin{proof} Comme les foncteurs $\tau_x$ sont bicontinus, utilisant les isomorphismes $\tau_x k[T]\simeq k[T]\otimes k[T(x)]\simeq\bigoplus_{\xi\in T(x)} k[T]$ naturels en le foncteur $T$ de $\rep(\A)$, on voit que $\tau_x(F)$ s'identifie à l'image du morphisme
$$\bigoplus_{i,\xi\in A_i(x)} k[A_i]\to\prod_{j,\zeta\in B_j(x)} k[B_j]$$
dont la composante correspondant aux facteurs $(i,\xi)$ à la source et $(j,\zeta)$ au but est l'image par $\Theta_{A_i,B_j}$ de
$$\underset{f_*\xi=\zeta}{\sum_{f\in\mathrm{Hom}(A_i,B_j)}}\lambda_{i,j}(f)[f].$$

Comme $A$ est de type fini, on peut choisir $x$ tel que $A_i$ soit engendré par un élément $\xi$ de $A_i(x)$.  Ainsi, la fonction
$$\mathrm{Hom}(A_i,B_j)\to B_j(x)\qquad f\mapsto f_*\xi$$
est \emph{injective}. Par conséquent, une composante du morphisme précédent est $\lambda_{i,j}(f)[f]$, dont l'image est $k[\mathrm{Im}\,f]$ puisque $\lambda_{i,j}\in k^\times$, d'où la conclusion.
\end{proof}

\subsection{Dualité ; applications}

Les deux énoncés  suivants, élémentaires mais très importants, figurent essentiellement dans \cite[lemme~11.5]{DTV}.

\begin{prop}\label{pr-dual_lin}
Supposons que $A$ est un foncteur de type fini de $\rep(\A)$ tel que, pour tout objet $a$ de $\A$, $A(a)$ soit un groupe abélien dont le cardinal est fini et inversible dans $k$. Alors il existe un isomorphisme $k[A]^\vee\simeq k[A^\sharp]$ dans $\F(\A^\op;k)$.
\end{prop}

\begin{proof}
Supposons d'abord que $k$ contient toutes les racines de l'unité. Alors, si $V$ est un groupe abélien fini tel que $V\otimes_\mathbb{Z}k=0$, la $k$-algèbre $k[V]$ est semi-simple \emph{déployée} ; autrement dit, le morphisme canonique de $k$-algèbres $k[V]\to k^{\mathrm{Hom}_\mathbb{Z}(V,k^\times)}$ (où $k^E$ désigne l'algèbre produit de copies de $k$ indexées par $E$, pour tout ensemble $E$) est un isomorphisme. De plus, $\mathrm{Hom}_\mathbb{Z}(V,k^\times)$ est naturellement isomorphe à $V^\sharp$. On en déduit un isomorphisme naturel $k[A]\simeq k^{A^\sharp}$ dans $\F(\A;k)$, puis la conclusion recherchée, en dualisant.

Dans le cas général, soit $K$ une extension du corps $k$ contenant les racines de l'unité. D'après ce qui précède, les foncteurs $k[A]^\vee\otimes K$ et $k[A^\sharp]\otimes K$ de $\F(\A;K)$ sont isomorphes. De plus, cet isomorphisme persiste en remplaçant $K$ par une sous-extension \emph{finie} du corps $k$ --- cela provient de ce que le morphisme canonique $\F(\A;k)\otimes K\to \F(\A;K)(F\otimes K,G\otimes K)$ est un isomorphisme pour tous foncteurs $F$ et $G$ de $\F(\A;k)$, avec $F$ de type fini (appliquer ensuite cet isomorphisme avec $F=k[A]$ et $G=k^{A^\sharp}$, puis avec $F=k^{A^\sharp}$ et $G=k[A]$ --- $k[A]$ est de type fini dans $\F(\A;k)$ car $A$ l'est dans $\rep(\A)$, et $k^{A^\sharp}$ est de type fini dans $\F(\A;k)$ car $k^{A^\sharp}\otimes K\simeq K[A]$ l'est dans $\F(\A;K)$ d'après ce qui précède). Autrement dit, on peut supposer que $K$ est de degré fini $d\in\mathbb{N}^*$ sur $k$. Post-composant par la restriction des scalaires de $K$ à $k$, on en déduit que $(k[A])^{\oplus d}\simeq (k^{A^\sharp})^{\oplus d}$ dans $\F(\A;k)$. Maintenant, $k[A]$ a un anneau d'endomorphismes de dimension finie sur $k$ grâce à la proposition~\ref{pr-dual_lin} (et au fait que $A$ est de type fini et à valeurs finies), de même que $k^{A^\sharp}$ (car c'est comme $k[A]$ un foncteur de type fini à valeurs de dimensions finies), de sorte que la propriété de Krull-Schmidt (cf. par exemple \cite[§\,5.1]{Pop}) donne la conclusion.
\end{proof}

\begin{rema}
La fin de la démonstration repose sur une variante dans les catégories de foncteurs considérées de \cite[§\,2, th.~3]{Bki}.
\end{rema}

\begin{rema}
Lorsque $k$ contient assez de racines de l'unité, on peut s'affranchir de l'hypothèse que $A$ est de type fini dans l'énoncé précédent, et l'isomorphisme obtenu est de plus fonctoriel en le foncteur additif $A$ (à valeurs finies et inversibles dans $k$).
\end{rema}

\begin{coro}\label{cor-dual_lin} Si la catégorie additive $\A$ est $k$-triviale, alors pour tout foncteur de type fini $A$ de $\rep(\A)$, il existe un isomorphisme naturel $k[A]^\vee\simeq k[A^\sharp]$ dans $\F(\A^\op;k)$.
\end{coro}

\begin{coro}\label{cor-lin_tcf} Supposons que la catégorie additive $\A$ est $k$-triviale. Si $A$ est un foncteur de type cofini de $\rep(\A)$, alors $k[A]$ est un foncteur de type cofini de $\F(\A;k)$.
\end{coro}

\begin{proof}
En raison de la condition (FH), les foncteurs de dualité sur $\rep(\A)$ et sur $\F(\A;k)$ échangent les foncteurs de type fini et les foncteurs de type cofini. La conclusion résulte donc de la proposition~\ref{pr-lin_tf} et du corollaire~\ref{cor-dual_lin}.
\end{proof}

Nous utiliserons abondamment la conséquence suivante de la proposition~\ref{pr-dual_lin}.

\begin{coro}\label{cor-tftcf} Supposons que la catégorie $\A$ est $k$-triviale. Étant donné un foncteur $F$ de $\F(\A;k)$, les assertions suivantes sont équivalentes :
\begin{enumerate}
\item[(a)] $F$ est de type fini et de type cofini ;
\item[(b)] il existe une famille finie $(A_i)$ de foncteurs de type fini de $\rep(\A)$, une famille finie $(B_j)$ de foncteur de type cofini de $\rep(\A)$ et des familles $(\lambda_{i,j}(f))_{f\in\mathrm{Hom}(A_i,B_j)}$ d'éléments de $k$ telles que
$$F=\mathrm{Im}\;\bigoplus_i k[A_i]\xrightarrow{\big(\Theta_{A_i,B_j}(\sum_{f\in\mathrm{Hom}(A_i,B_j)}\lambda_{i,j}(f)[f])\big)_{i,j}}\bigoplus_j k[B_j]\;;$$
\item[(c)]  il existe des familles finies $(A_i)$ et $(B_j)$ de foncteurs de type fini et de type cofini de $\rep(\A)$ et des familles $(\lambda_{i,j}(f))_{f\in\mathrm{Hom}(A_i,B_j)}$ d'éléments de $k$ telles que
$$F=\mathrm{Im}\;\bigoplus_i k[A_i]\xrightarrow{\big(\Theta_{A_i,B_j}(\sum_{f\in\mathrm{Hom}(A_i,B_j)}\lambda_{i,j}(f)[f])\big)_{i,j}}\bigoplus_j k[B_j].$$
\end{enumerate}
\end{coro}

(Comme $\A$ vérifie l'hypothèse (FH), $\mathrm{Hom}(A,B)$ est fini pour $A$ de type fini et $B$ de type cofini.)

\begin{proof}
$(a)\Rightarrow (b)$ : si $F$ est de type fini, c'est un quotient d'une somme directe finie de foncteurs de la forme $k[\A(x_i,-)]$, et les $\A(x_i,-)$ sont des foncteurs de type fini de $\rep(\A)$. Si $F$ est de type cofini, c'est un sous-foncteur d'une somme directe finie de foncteurs de la forme $k^{\A(-,y_j)}$, qui sont isomorphes à $k[\A(-,y_j)^\sharp]$ par la proposition~\ref{pr-dual_lin}. Comme $\A(-,y_j)^\sharp$ est un foncteur de type cofini de $\rep(\A)$ (car $\A$ vérifie (FH)), on obtient (b) en appliquant la proposition~\ref{pr-morlin}.

$(b)\Rightarrow (c)$ : si les $(A_i)$, $(B_j)$ et $(\lambda_{i,j}(f))$ sont comme dans (b), on peut remplacer chaque $B_j$ par son sous-foncteur $B'_j$ somme des images des $f\in\mathrm{Hom}(A_i,B_j)$ tels que $\lambda_{i,j}(f)\ne 0$. Il n'y a qu'un nombre fini de tels $f$, et $\mathrm{Im}\,f$ est de type fini, comme quotient de $A_i$. Il s'ensuit que $B'_j$ est de type fini ; ce foncteur est par ailleurs de type cofini comme sous-foncteur de $B_j$. Un raisonnement dual montre qu'on peut remplacer chaque $A_i$ par un quotient $A'_i$ de type fini et de type cofini, d'où (c).

$(c)\Rightarrow (a)$ : cela provient de ce que la classe des foncteurs de type fini (resp. cofini) est stable par somme directe finie et par quotient (resp. sous-objet) et de ce que la linéarisation $k[-] : \rep(\A)\to\F(\A;k)$ préserve les objets de type fini (proposition~\ref{pr-lin_tf}), ainsi que les objets de type cofini (corollaire~\ref{cor-lin_tcf}).
\end{proof}

\subsection{Groupes d'extensions}

Le résultat fondamental suivant, qui repose sur la correspondance de Dold-Kan, est bien connu des experts ; il s'appuie sur des rappels figurant dans \cite{DT-ext}, dont il constitue une variation.

\begin{prop}\label{pr-ExtLin} Supposons que la catégorie $\A$ est $k$-triviale. Soient $A$ et $B$ des foncteurs $\rep(\A)$ et $n>0$ un entier. On suppose que $B$ est à valeurs finies. Alors $\mathrm{Ext}^n(k[A],k[B])=0$.
\end{prop}

\begin{proof} Comme $B$ est à valeurs finies, on a $k[B]\simeq (k^{B})^\vee$, donc le $k$-espace vectoriel $\mathrm{Ext}^n(k[A],k[B])$ est le dual de $\mathrm{Tor}_n^{k[\A]}(k^B,k[A])$ (cf. par exemple \cite[prop.~2.12]{DT-ext}). Écrivons $B^\sharp\simeq\underset{i\in\I}{\col}T_i$ dans $\rep(\A^\op)$, où $\I$ est une petite catégorie filtrante, les $T_i$ sont de type fini et les flèches structurales $T_i\to T_j$ des monomorphismes. Cette colimite est \emph{ponctuellement triviale} au sens où, pour tout objet $a$ de $\A$, la colimite $\underset{i\in\I}{\col}T_i(a)$ est triviale, c'est-à-dire qu'il existe $i\in\mathrm{Ob}\,\I$ tel que toute flèche $i\to j$ de $\I$ induise un isomorphisme $T_i(a)\to T_j(a)$, car $B(a)$ est fini. Il s'ensuit que $B\simeq\underset{i\in\I}{\lim}\,T_i^\sharp$, puis que $k^B\simeq \underset{i\in\I}{\col}k^{T_i^\sharp}$. Maintenant, on a $k^{T_i^\sharp}\simeq k[T_i]$ par le corollaire~\ref{cor-dual_lin}. Comme les foncteurs Tor commutent aux colimites filtrantes, on a finalement $\mathrm{Tor}_n^{k[\A]}(k^B,k[A])\simeq\underset{i\in\I}{\col}\mathrm{Tor}_n^{k[\A]}(k[B_i^\sharp],k[A])$.

Autrement dit, on peut supposer $B$ de type fini, et il s'agit de montrer la nullité de $\mathrm{Tor}_n^{k[\A]}(k[B^\sharp],k[A])$. Soit $P_\bullet$ une résolution projective \emph{simpliciale} de $B^\sharp$ dans $\rep(\A^\op)$. Alors $k[P_\bullet]$ est une résolution projective simpliciale de $k[B^\sharp]$ dans $\F(\A^\op;k)$ (cf. \cite[lemme~9.8]{DT-ext}). Il s'ensuit, en utilisant \cite[lemme~5.3]{DT-ext}, que $\mathrm{Tor}_n^{k[\A]}(k[B^\sharp],k[A])$ est isomorphe à l'homologie à coefficients dans $k$ du groupe abélien simplicial $P_\bullet\underset{\A}{\otimes}A$, qui est nulle en raison de l'hypothèse $k$-triviale sur $\A$ et de \cite[cor.~3.6]{DT-ext}.
\end{proof}

\begin{rema}
Si $A$ possède une résolution projective de type fini (par exemple, si $A$ est de type fini et que $\rep(\A)$ est localement noethérienne), on peut s'affranchir de toute hypothèse sur $B$ et éviter le recours à la dualité qui occupe la première partie de la démonstration précédente. On peut en effet utiliser la proposition~\ref{pr-dual_lin} et raisonner directement à partir de $k$-espaces vectoriels (co)simpliciaux, comme dans la deuxième partie de la démonstration.
\end{rema}

\section{Les foncteurs $\Q_A$ et $\Q^A$}\label{sQA}

\subsection{Généralités}

\begin{nota}
Soit $A$ un foncteur de $\rep(\A)$. On note $\Q_A$ le conoyau du morphisme canonique
$$\bigoplus_{B\subsetneq A} k[B]\to k[A]$$
de $\F(\A;k)$ dont les composantes sont les inclusions, et $\Q^A$ le noyau du morphisme canonique
$$k[A]\to\prod_{0\ne B\subset A} k[A/B]$$
dont les composantes sont les projections.
\end{nota}

\begin{exem}\label{ex-Qsom}
Les foncteurs $\Q_0$ et $\Q^0$ sont constants en $k$. Si $A$ est un foncteur simple de $\rep(\A)$, alors $\Q_A\simeq\Q^A\simeq k[A]^\red$.

Si $A\in\mathrm{Ob}\,\rep(\A)$ possède un sous-foncteur strict maximal $B$, alors on dispose d'une suite exacte courte
\begin{equation}\label{eq-secQ}
    0\to k[B]\to k[A]\to\Q_A\to 0.
\end{equation}
Si $A$ n'est pas simple, comme $A$ est de type fini et que $B$ n'est pas facteur direct de $A$, la proposition~\ref{pr-morlin} montre que cette suite exacte n'est pas scindée.
\end{exem}

La propriété de dualité suivante, qui résulte du corollaire~\ref{cor-dual_lin}, explique pourquoi nous nous concentrerons, dans la suite du présent travail, sur les foncteurs $\Q_A$, laissant au lecteur le soin d'écrire les énoncés qu'on peut en déduire pour les $\Q^A$.

\begin{prop}\label{pr-dualQ} Supposons que la catégorie $\A$ est $k$-triviale et que $A$ est un foncteur de type fini de $\rep(\A)$. Alors il existe un isomorphisme $(\Q_A)^\vee\simeq Q^{A^\sharp}$ dans $\F(\A^\op;k)$.
\end{prop}

\begin{nota}
Si $x$ et $y$ sont des objets d'une catégorie $\C$, on note $\surj_\C(x,y)$, ou $\surj(x,y)$ s'il n'y a pas d'ambiguïté possible, l'ensemble des épimorphismes de $x$ vers $y$ dans $\C$. On fait de $x\mapsto k[\surj(x,y)]$ un foncteur $\C^\op\to k\Md$ en associant à un morphisme $\varphi : x'\to x$ de $\C$ l'application linéaire $k[\surj(x,y)]\to k[\surj(x',y)]$ associant à $[f]$ (pour $f\in\surj(x,y)$) $[f\circ\varphi]$ si $f\circ\varphi$ est un épimorphisme, et $0$ sinon. Cette définition fait sens car, si $\psi$, $\varphi$ et $f$ sont des flèches composables de $\C$, $f\circ\varphi\circ\psi$ ne peut être un épimorphisme que si $f\circ\varphi$ en est un. On note que le foncteur $k[\surj(-,y)] : \C^\op\to k\Md$ est muni d'une action tautologique du groupe $\mathrm{Aut}_\C(y)$.
\end{nota}

Les trois énoncés simples suivants sont laissés en exercice.

\begin{prop}\label{pr-descrQ} Soient $A$ un foncteur de $\rep(\A)$ et $a$ un objet de $\A$. Il existe un diagramme commutatif
$$\xymatrix{k[A(a)]\ar@{->>}[d]\ar@{->>}[dr] & & \\
\Q_A(a)\ar[r]^-\simeq & k[\{\xi\in A(a)\,|\,A_\xi=A\}]\ar[r]^-\simeq & k[\surj(\A(a,-),A)]
}$$
où la flèche verticale de gauche est l'évaluation en $a$ de la projection canonique $k[A]\twoheadrightarrow\Q_A$ et la flèche oblique l'application linéaire envoyant $[\xi]\in k[A(a)]$, pour $\xi\in A(a)$, sur $[\xi]$ si $A_\xi=A$ et sur $0$ sinon.
\end{prop}

\begin{exem}\label{ex-Q-rep} Pour tous objets $x$ et $t$ de $\A$, $\Q_{\A(x,-)}(t)$ s'identifie au quotient de $k[\A(x,t)]$ par le sous-espace vectoriel engendré par $[f]$, où $f : x\to t$ n'est pas un monomorphisme scindé.
\end{exem}

\begin{coro}\label{cor-valQ} Soient $A$ un foncteur de $\rep(A)$ et $a$ un objet de $\A$.
\begin{enumerate}
    \item L'espace vectoriel $\Q_A(a)$ est non nul si et seulement si $A$ est engendré par un élément de $A(a)$, i.e. est isomorphe à un quotient de $\A(a,-)$.
    \item Le foncteur $\Q_A$ de $\F(\A;k)$ est de type fini. Il est non nul si et seulement si $A$ est un foncteur de type fini de $\rep(\A)$.
\end{enumerate}
\end{coro}

\begin{prop}[fonctoralité de $A\mapsto\Q_A$]\label{pr-foncQ} Soit $\alpha : A\to A'$ un morphisme de $\rep(\A)$, alors il existe un (unique) morphisme $\Q_\alpha$ faisant commuter le diagramme suivant
$$\xymatrix{k[A]\ar[r]^-{k[\alpha]}\ar@{->>}[d] & k[A']\ar@{->>}[d] \\
\Q_A\ar[r]^-{\Q_\alpha} & \Q_{A'}
}$$
où les flèches verticales sont les projections canoniques si et seulement si l'une des deux conditions suivantes est vérifiée :
\begin{enumerate}
    \item $\alpha$ n'est pas un épimorphisme --- on a alors $\Q_\alpha=0$ ;
    \item $\alpha$ est un épimorphisme essentiel --- $\Q_\alpha$ est alors un épimorphisme.
\end{enumerate}
En particulier, $\Q_A$ est muni d'une action canonique du groupe $\operatorname{Aut}(A)$.
\end{prop}

\subsection{Résolution fondamentale}

Pour un foncteur $A$ de $\rep(\A)$ et un entier $d\in\mathbb{N}$, on note $\N_d(A)$ l'ensemble des $(d+1)$-uplets $T=(T_0,\dots,T_d)$ de sous-objets \emph{stricts} de $A$ tels que $T_0\subsetneq\dots\subsetneq T_d$. Ainsi\,\footnote{On rappelle que $\lgr(A)$ désigne la longueur de $A$ --- cf. appendice~\ref{ap-catab}.}, $\lgr(A)\le d$ si et seulement si $\N_d(A)=\varnothing$. Si $A$ est de longueur finie et que $\A$ vérifie (FH), alors chaque ensemble $\N_d(A)$ est fini en vertu de la proposition~\ref{sousobjets-nbfini}.

\begin{prop}\label{pr-ressimpl}
Soit $A$ un foncteur de $\rep(\A)$. Il existe une suite exacte
$$\cdots\to\bigoplus_{T\in\N_d(A)}k[T_0]\to\dots\to\bigoplus_{T\in\N_0(A)}k[T_0]\to k[A]\to\Q_A\to 0.$$
\end{prop}

\begin{proof}
Étant donné un objet $a$ de $A$ et un élément $\xi$ de $A(a)$, considérons l'ensemble $E_\xi$ des sous-foncteurs \emph{stricts} $T$ de $A$ tels que $\xi\in T(a)$, ordonné par inclusion. Si $\xi$ engendre $A$, $E_\xi$ est vide ; sinon, cet ensemble ordonné possède un plus petit élément, à savoir $A_\xi$, il est donc \emph{contractile}. Pour tout $d\in\mathbb{N}$, l'ensemble des $d$-simplexes non dégénérés du nerf de $E_\xi$ est $\{T\in\N_d(A)\,|\,\xi\in T_0\}$. Par conséquent, on en déduit un complexe de la forme
$$\cdots\to\underset{\xi\in T_0}{\bigoplus_{T\in\N_d(A)}}k\to\dots\to\underset{\xi\in T_0}{\bigoplus_{T\in\N_0(A)}}k$$
dont l'homologie est concentrée en degré nul, où elle est isomorphe à $k$ (via l'augmentation) si $A_\xi\ne A$, et est identiquement nulle si $A_\xi=A$.

En considérant la somme directe sur tous les $\xi\in A(a)$ de ces complexes, on en déduit une suite exacte
$$\cdots\to\bigoplus_{T\in\N_d(A)}k[T_0(a)]\to\dots\to\bigoplus_{T\in\N_0(A)}k[T_0(a)]\to k[\{\xi\in A(a)\,|\,A_\xi\ne A\}]$$
puis en utilisant la proposition~\ref{pr-descrQ} une suite exacte
$$\cdots\to\bigoplus_{T\in\N_d(A)}k[T_0(a)]\to\dots\to\bigoplus_{T\in\N_0(A)}k[T_0(a)]\to k[A(a)]\to\Q_A(a)\to 0$$
qui est manifestement l'évaluation en $a$ d'une suite exacte de $\F(\A;k)$ de la forme souhaitée : la flèche $k[A(a)]\to\Q_A(a)$ est induite par la projection canonique $k[A]\twoheadrightarrow\Q_A$, et les composantes des autres morphismes sont soit nulles pour tout $a$, soit induites par des inclusions canoniques $k[T_0]\hookrightarrow k[U]$ pour des sous-foncteurs $T_0\subset U$ de $A$. Cela achève la démonstration.
\end{proof}

On note $\Qse_d(X)$, pour un foncteur $X$ de $\rep(\A)$ et $d\in\mathbb{N}$, l'ensemble des $(d+1)$-uplets $N=(N_0,\dots,N_d)$ de sous-foncteurs de $X$ tels que $N_0\supsetneq N_1\supsetneq\dots\supsetneq N_d\ne 0$. Le résultat suivant découle des propositions~\ref{pr-ressimpl} et~\ref{pr-dualQ}.

\begin{coro}\label{cor-res_cosimpl}
Supposons que la catégorie $\A$ est $k$-triviale. Si $X$ est un foncteur de type cofini de $\rep(\A)$, on dispose d'une suite exacte
$$0\to\Q^X\to k[X]\to\prod_{N\in\Qse_0(X)}k[X/N_0]\to\cdots\to\prod_{N\in\Qse_d(X)}k[X/N_0]\to\cdots\;.$$
\end{coro}

\subsection{Groupes d'extensions}

\begin{theo}\label{th-ExtQ} Supposons que la catégorie $\A$ est $k$-triviale. Soient $A$ un foncteur fini et $X$ un foncteur de type fini de $\rep(\A)$.
\begin{enumerate}
    \item Il existe un isomorphisme $\mathrm{Hom}(k[X],\Q_A)\simeq k[\surj(X,A)]$ naturel en $X$ et équivariant par rapport à l'action de $\mathrm{Aut}(A)$. Plus précisément, il existe un diagramme commutatif
    $$\xymatrix{\mathrm{Hom}(k[X],k[A])\ar[r] & \mathrm{Hom}(k[X],\Q_A)\ar[d]^\simeq \\
    k[\mathrm{Hom}(X,A)]\ar@{->>}[r]\ar[u]^{\Theta_{X,A}}_\simeq & k[\surj(X,A)]
    }$$
    où la flèche horizontale supérieure est induite par la projection canonique $k[A]\twoheadrightarrow\Q_A$ et la flèche horizontale inférieure envoie $[f]\in k[\mathrm{Hom}(X,A)]$ sur $[f]\in k[\surj(X,A)]$ si $f$ est un épimorphisme et sur $0$ sinon.
    \item On a $\mathrm{Ext}^n(k[X],\Q_A)=0$ pour tout $n>0$.
    \item On a $\mathrm{Ext}^n(\Q_A,k[X])=0$ pour tout $n>\lgr(A)$.
\end{enumerate}
\end{theo}

\begin{proof} La proposition~\ref{pr-ressimpl} fournit un complexe \emph{exact} $C^\bullet$ tel que $C^{\lgr(A)+1}=\Q_A$, $C^{\lgr(A)}=k[A]$, $C^{\lgr(A)-1}=\bigoplus_{B\subsetneq A}k[B]$, $C^i$ est une somme directe finie de linéarisations de foncteurs de $\rep(\A)$ pour $i\in\{0,\dots,\lgr(A)\}$, et $C^i=0$ sinon. On en déduit une suite spectrale d'hypercohomologie d'aboutissement nul dont la première page est donnée par
$$E_1^{i,j}=\mathrm{Ext}^j(k[X],C^i).$$
Par la proposition~\ref{pr-ExtLin}, on a $E_1^{i,j}=0$ pour $j>0$ et $i\ne\lgr(A)+1$. Il s'ensuit que $E_1^{\lgr(A)+1,j}=0$ pour $j>0$ (aucune différentielle ne peut entrer ni sortir de ce terme, et l'aboutissement est nul), d'où la deuxième assertion du théorème, et que la suite
$$\bigoplus_{B\subsetneq A}\mathrm{Hom}(k[X],k[B])=E_1^{\lgr(A)-1,0}\to\mathrm{Hom}(k[X],k[A])=E_1^{\lgr(A),0}\to E_1^{\lgr(A)+1,0}\to 0$$
est exacte, d'où en utilisant la proposition~\ref{pr-morlin} une suite exacte
$$\bigoplus_{B\subsetneq A}k[\mathrm{Hom}(X,B)]\to k[\mathrm{Hom}(X,A)]\to\mathrm{Hom}(k[X],\Q_A)\to 0$$
où la première flèche est induite par les inclusions $B\hookrightarrow A$, et a donc pour image le sous-espace vectoriel de $k[\mathrm{Hom}(X,A)]$ engendré par les $[f]$ où $f : X\to A$ est un morphisme dont l'image est un sous-objet strict de $A$, c'est-à-dire un morphisme qui n'est pas un épimorphisme. Cela établit la première assertion.

Pour le dernier point, la proposition~\ref{pr-ExtLin} montrant que la linéarisation d'un foncteur de $\rep(\A)$ est $\mathrm{Ext}^*(-,k[X])$-acyclique, la proposition~\ref{pr-ressimpl} implique que $\mathrm{Ext}^*(\Q_A,k[X])$ est isomorphe à l'homologie d'un complexe $\mathrm{Hom}(C_\bullet,k[X])$, où $C_0=k[A]$ et $C_i=\bigoplus_{T\in\N_{i-1}(A)} k[T_0]$ pour $i>0$. Comme $\N_i(A)$ est vide pour $\lgr(A)\le i$, cela conclut la démonstration du théorème.
\end{proof}

\begin{theo}\label{thm-Ext-QA}
Supposons que la catégorie $\A$ est $k$-triviale. Soient $A$ et $B$ des foncteurs finis de $\rep(\A)$ tels que $\lgr(B)\le\lgr(A)$ et $n\in\mathbb{N}$. Alors $\mathrm{Ext}^n(\Q_B,\Q_A)$ est nul sauf si $n=0$ et que $A$ et $B$ sont isomorphes. De plus, le morphisme d'anneaux canonique $k[\mathrm{Aut}(A)]\to\mathrm{End}(\Q_A)$ est un isomorphisme.
\end{theo}

\begin{proof}
La proposition~\ref{pr-ressimpl} fournit une suite spectrale d'hypercohomologie
\begin{equation*}
  E_2^{i,j}= \bigoplus_{B_j\subsetneq \dots \subsetneq B_0=B}\mathrm{Ext}^i( k[B_j],\Q_A) \Rightarrow \mathrm{Ext}^{i+j}(\Q_B,\Q_A)
\end{equation*}

Le théorème~\ref{th-ExtQ} montre que $E_2^{i,j}$ est nul pour $i>0$ ainsi que dans les cas où $\surj(B_j,A)$ est vide pour tout $B_j\subsetneq\dots\subsetneq B_0=B$. Comme $\lgr(B_j)\le\lgr(B)-j\le\lgr(A)-j$, c'est toujours le cas, sauf peut-être si $j=0$ et que $B$ est isomorphe à $A$, et l'on obtient alors que le morphisme canonique $\mathrm{Hom}(\Q_B,\Q_A)\to E_2^{0,0}=\mathrm{Hom}(k[B],\Q_A)\simeq k[\surj(B,A)]$ est un isomorphisme, d'où le théorème.
\end{proof}

\begin{rema}\label{rq-ExtQnz}
L'annulation de $\mathrm{Ext}^1(\Q_B,\Q_A)$ (la catégorie $\A$ étant toujours supposée $k$-triviale) ne vaut plus en général si $\lgr(B)>\lgr(A)$, dès lors que la catégorie $\rep(\A)$ n'est pas semi-simple. Soit en effet $0\to S\to T\to S'\to 0$ une suite exacte non scindée de $\rep(\A)$ avec $S$ et $S'$ simples. Alors $\mathrm{Ext}^1(\Q_T,k[S])\ne 0$ par l'exemple~\ref{ex-Qsom} ; comme $\Q_S\simeq k[S]^\red$ par le même exemple, on a également $\mathrm{Ext}^1(\Q_T,\Q_S)\ne 0$.

On notera qu'il peut aussi exister un morphisme non nul $\Q_A\to\Q_B$ avec $A$ et $B$ non isomorphes ; c'est par exemple le cas si $A$ est de type fini et qu'il existe un épimorphisme essentiel $A\twoheadrightarrow B$, car on dispose alors d'un épimorphisme $\Q_A\twoheadrightarrow\Q_B$ par la proposition~\ref{pr-foncQ}.
\end{rema}

Nous aurons également de la variante suivante du résultat d'annulation du théorème~\ref{thm-Ext-QA} :

\begin{prop}\label{pr-ExtQQ}
Supposons que la catégorie $\A$ est $k$-triviale. Soient $A$ et $B$ des foncteurs finis de $\rep(\A)$ tels que $\lgr(B)<\lgr(A)$. Alors $\mathrm{Ext}^*(\Q^B,\Q_A)=0$.
\end{prop}

\begin{proof}
Ce résultat se démontre de façon entièrement analogue au théorème~\ref{thm-Ext-QA}, en utilisant la corésolution de $\Q^B$ donnée par le corollaire~\ref{cor-res_cosimpl}, corésolution qui est finie et fait intervenir des produits finis puisque $B$ est fini.
\end{proof}

\section{Les foncteurs $\Q_{A,M}$ et $\Q^{A,M}$}\label{sQAM}

Nous introduisons maintenant des protagonistes fondamentaux de ce travail, qui constituent des \guillemotleft~briques élémentaires~\guillemotright\ de la catégorie $\F(\A;k)$, au moins lorsque $\A$ est $k$-triviale, en un sens que nos résultats ultérieurs (notamment le théorème~\ref{th-princ} et la proposition~\ref{pr-G0Fd}) préciseront.

\begin{defi}
Soient $A$ un foncteur de $\rep(\A)$ et $M$ un $k[\mathrm{Aut}(A)]$-module. On note $\Q_{A,M}$ (resp. $\Q^{A,M}$) le foncteur $M\underset{k[\mathrm{Aut}(A)]}{\otimes}\Q_A\simeq (M\otimes\Q_A)_{\mathrm{Aut}(A)}$ (resp. $(M\otimes\Q_A)^{\mathrm{Aut}(A)}$) de $\F(\A;k)$, où l'indice (resp. l'exposant) indique les co-invariants (resp. les invariants) sous l'action diagonale du groupe $\mathrm{Aut}(A)$.
\end{defi}

\begin{rema}
Ces constructions sont directement inspirées des travaux de Powell \cite{GP-cow}. Précisément, dans la catégorie $\F(k,k)$, lorsque $k$ est un corps fini de caractéristique $p$, si $\lambda$ est une partition $p$-régulière de longueur $n$ et $M_\lambda$ la représentation irréductible correspondante de $\GL_n(k)$ à coefficients dans $k$, alors $Q_{\mathrm{Hom}_k(k^n,-),M_\lambda}$ est exactement le foncteur $DJ_\lambda$ de \cite[déf.~3.0.1]{GP-cow}.

De plus, Powell fournit une présentation conceptuelle de ces foncteurs en termes d'adjonctions, qui s'étend essentiellement sans changement à tous les cas où la catégorie abélienne $\rep(\A)$ est semi-simple. La généralisation de ce point de vue au cas général semble toutefois poser des difficultés substantielles. 
\end{rema}

La propriété de dualité suivante, qui découle de la proposition~\ref{pr-dualQ} (il n'y a pas besoin d'hypothèse de finitude sur $A$, car $\Q_A=0$ si $A$ n'est pas de type fini --- cf. corollaire~\ref{cor-valQ}), justifie que nous 
nous concentrions sur l'étude des foncteurs $\Q_{A,M}$.

\begin{prop}\label{pr-dQAM} Supposons que la catégorie $\A$ est $k$-triviale, que $A$ est un foncteur de type fini de $\rep(\A)$ et $M$ un $k[\mathrm{Aut}(A)]$-module. On dispose dans $\F(\A^\op;k)$ d'un isomorphisme $(\Q_{A,M})^\vee\simeq\Q^{A^\sharp,\mathrm{Hom}_k(M,k)}$ naturel en $A$ et en $M$.
\end{prop}

\begin{prop}\label{pr-Qexact} Le foncteur
$$k[\mathrm{Aut}(A)]\Md\to\F(\A;k)\qquad M\mapsto\Q_{A,M}$$
est exact. Il est bicontinu si $A$ est à valeurs finies. Il est fidèle si $A$ est de type fini.
\end{prop}

\begin{proof} Soit $a$ un objet de $\A$. Le $\mathrm{Aut}(A)$-ensemble
$$\Gamma(a):=\{\xi\in A(a)\,|\,A_\xi=A\}$$
est \emph{libre}. En effet, si $\varphi\in\mathrm{Aut}(A)$ et $\xi\in A(a)$ sont tels que $\varphi_*\xi=\xi$, alors $\varphi$ induit l'identité sur $A_\xi$, donc $\varphi=\mathrm{Id}_A$ si $A_\xi=A$. Par ailleurs, $\Gamma(a)$ est fini si $A$ est à valeurs finies, et si $A$ est de type fini, il existe $a\in\mathrm{Ob}\,\A$ tel que $\Gamma(a)$ soit non vide.
La conclusion découle donc de la proposition~\ref{pr-descrQ}.
\end{proof}

Le résultat suivant découle des propositions~\ref{pr-dQAM} et~\ref{pr-Qexact}.

\begin{coro}\label{cor-QAMdual}
Supposons que la catégorie $\A$ est $k$-triviale. Soient $A$ un foncteur de $\rep(\A)$ et $M$ une représentation du groupe $\mathrm{Aut}(A)$.
\begin{enumerate}
    \item Le foncteur
$$k[\mathrm{Aut}(A)]\Md\to\F(\A;k)\qquad M\mapsto\Q^{A,M}$$
est bicontinu.
    \item Le morphisme naturel en $M$
    $$M\underset{k[\mathrm{Aut}(A)]}{\otimes}\Q^A\simeq (M\otimes\Q^A)_{\mathrm{Aut}(A)}\to (M\otimes\Q^A)^{\mathrm{Aut}(A)}=\Q^{A,M}$$
    donné par la norme est un isomorphisme.
    \item En appliquant aux morphismes canoniques $\Q^A\hookrightarrow k[A]\twoheadrightarrow\Q_A$ le foncteur $M\underset{k[\mathrm{Aut}(A)]}{\otimes} -$, on obtient, à isomorphisme naturel en $M$ près, des flèches $\Q^{A,M}\hookrightarrow M\underset{k[\mathrm{Aut}(A)]}{\otimes}k[A]\twoheadrightarrow\Q_{A,M}$ dont la première est un monomorphisme et la seconde un épimorphisme.
\end{enumerate}
\end{coro}

La définition suivante n'interviendra pas dans la suite de la présente section, mais elle s'avérera importante pour classifier les foncteurs simples de $\F(\A;k)$ (cf. théorème~\ref{th-simp_Fd} ci-après).

\begin{nota}\label{not-prl_interm}
Supposons que la catégorie $\A$ est $k$-triviale. Soient $A$ un foncteur de $\rep(\A)$ et $M$ une représentation du groupe $\mathrm{Aut}(A)$. On note $\Q(A,M)$ l'image du morphisme $\Q^{A,M}\hookrightarrow M\underset{k[\mathrm{Aut}(A)]}{\otimes}k[A]\twoheadrightarrow\Q_{A,M}$ du corollaire précédent.

On note simplement $\Q(A)$ pour $\Q(A,k[\mathrm{Aut}(A)])$.
\end{nota}

Avant d'étudier, dans la section~\ref{sect-Qprolint}, les foncteurs $\Q(A,M)$, nous donnons quelques propriétés générales des foncteurs $\Q_{A,M}$.

\begin{prop}\label{pr-tauQAM} Soit $A$ un foncteur de $\rep(\A)$.
Soient $A$ un foncteur de type fini de $\rep(\A)$ et $M$ un $k[\mathrm{Aut}(A)]$-module non nul. Il existe un objet $x$ de $\A$ tel que $k[A]$ soit facteur direct de $\tau_x(\Q_{A,M})$.
\end{prop}

\begin{proof} On choisit $x$ tel que $A$ soit quotient de $\A(x,-)$. Comme le foncteur $\tau_x$ est exact, $\tau_x(\Q_A)$ est le conoyau du morphisme
$$\underset{\zeta\in B(x)}{\bigoplus_{B\subsetneq A}} k[B]\to\bigoplus_{\xi\in A(x)}k[A]$$
dont la composante $k[B]\to k[A]$ étiquetée par $\zeta$ à la source et $\xi$ au but est la linéarisation de l'inclusion $B\hookrightarrow A$ si $\zeta=\xi\in A(x)\supset B(x)$ et $0$ sinon. En particulier, cette composante est toujours nulle si $A_\xi=A$. Il s'ensuit que $k[\{\xi\in A(x)\,|\,A_\xi=A\}]\otimes k[A]$ est facteur direct de $\tau_x(\Q_A)$.

Comme $\tau_x$ est cocontinu, $\tau_x(\Q_{A,M})\simeq M\underset{k[\mathrm{Aut}(A)]}{\otimes}\tau_x(\Q_A)$, donc
$$M\underset{k[\mathrm{Aut}(A)]}{\otimes}\big(k[\{\xi\in A(x)\,|\,A_\xi=A\}]\otimes k[A]\big)$$
est facteur direct de $\tau_x(\Q_{A,M})$. Or $k[\{\xi\in A(x)\,|\,A_\xi=A\}]$ est un $k[\mathrm{Aut}(A)]$-module \emph{libre} (cf. la démonstration de la proposition~\ref{pr-Qexact}), de rang non nul vu le choix de $x$, donc $M\underset{k[\mathrm{Aut}(A)]}{\otimes}\big(k[\{\xi\in A(x)\,|\,A_\xi=A\}]\otimes k[A]\big)$ est la somme directe d'un nombre strictement positif de copies de $k[A]$, d'où la proposition.
\end{proof}

Nous donnons maintenant des résultats homologiques fondamentaux spécifiques à l'inégale caractéristique.

\begin{theo}\label{thm-Ext-QAM}
Supposons que la catégorie $\A$ est $k$-triviale. Soient $A$ et $B$ des foncteurs finis de $\rep(\A)$ tels que $\lgr(A)\le\lgr(B)$, $M$ et $N$ des objets de $k[\mathrm{Aut}(A)]\Md$ et $k[\mathrm{Aut}(B)]\Md$ respectivement.
\begin{enumerate}
    \item Si $A$ et $B$ ne sont pas isomorphes, alors $\mathrm{Ext}^*_{\F(\A;k)}(\Q_{A,M},\Q_{B,N})$ est nul.
    \item Si $A=B$, alors le morphisme naturel
    $$\mathrm{Ext}^*_{k[\mathrm{Aut}(A)]}(M,N)\to\mathrm{Ext}^*_{\F(\A;k)}(\Q_{A,M},\Q_{A,N})$$ induit par le foncteur exact $\Q_{A,-}$ est un isomorphisme.
\end{enumerate}
\end{theo}

\begin{proof} Comme le foncteur $\Q_{A,-}$ est exact (proposition~\ref{pr-Qexact}),
$$(M,N)\mapsto\mathrm{Ext}^*_{\F(\A;k)}(\Q_{A,M},\Q_{B,N})$$
définit un bifoncteur sur $k[\mathrm{Aut}(A)]\Md\times k[\mathrm{Aut}(B)]\Md$ qui est cohomologique par rapport à chaque variable. Le théorème~\ref{thm-Ext-QA} montre qu'il est nul lorsque $M=k[\mathrm{Aut}(A)]$ et $N=k[\mathrm{Aut}(B)]$, sauf si $A$ et $B$ sont isomorphes, et que, si $A=B$, le morphisme naturel $\mathrm{Ext}^*_{k[\mathrm{Aut}(A)]}(M,N)\to\mathrm{Ext}^*_{\F(\A;k)}(\Q_{A,M},\Q_{A,N})$ est un isomorphisme si $M=N=k[\mathrm{Aut}(A)]$. Comme $Q_{A,-}$ est bicontinu (proposition~\ref{pr-Qexact}) et qu'un $k[\mathrm{Aut}(B)]$-module est injectif si et seulement s'il est projectif, on en déduit que, pour $M$ projectif et $N$ injectif, $\mathrm{Ext}^*_{\F(\A;k)}(\Q_{A,M},\Q_{B,N})=0$, sauf si $A\simeq B$, et que pour $A=B$ le morphisme naturel $\mathrm{Ext}^*_{k[\mathrm{Aut}(A)]}(M,N)\to\mathrm{Ext}^*_{\F(\A;k)}(\Q_{A,M},\Q_{A,N})$ est un isomorphisme (pour $M$ projectif et $N$ injectif). Par comparaison de foncteurs cohomologiques, l'annulation ou l'isomorphisme s'étendent à tous les modules $M$ et $N$, d'où le théorème.
\end{proof}

La proposition suivante nous sera utile pour établir un autre résultat d'annulation cohomologique, l'importante proposition~\ref{pr-Ext_HR}.
\begin{prop}\label{pr-ExtQAMQ}
Supposons que la catégorie $\A$ est $k$-triviale. Soient $A$ et $B$ des foncteurs finis de $\rep(\A)$ tels que $\lgr(B)<\lgr(A)$, $M$ et $N$ des objets de $k[\mathrm{Aut}(A)]\Md$ et $k[\mathrm{Aut}(B)]\Md$ respectivement. Alors $\mathrm{Ext}^*(\Q^{B,N},\Q_{A,M})=0$.
\end{prop}

\begin{proof}
Ce résultat s'établit par des arguments formels analogues à ceux de la démonstration du théorème~\ref{thm-Ext-QAM}, à partir des propositions~\ref{pr-ExtQQ}, \ref{pr-Qexact} et de la première assertion du corollaire~\ref{cor-QAMdual}.
\end{proof}

\section{Les foncteurs $\Q(A,M)$}\label{sect-Qprolint}

Cette section vise à préciser la structure des foncteurs $\Q(A,M)$ à partir de celle des foncteurs de la forme $\Q_{A,M}$ ou $\Q_{A,M}$. Ses résultats --- plus précisément, la proposition~\ref{pr-QA_copres} et le corollaire~\ref{cor-QAMequiv} (qui est immédiat lorsque $k$ est de caractéristique nulle) --- joueront un rôle crucial dans la démonstration de l'importante proposition~\ref{pr-lin_fini}. Ils interviendront aussi pour montrer la première assertion de la proposition~\ref{pr-bicont}, mais cette assertion n'est pas essentielle (en particulier on peut démontrer sans cela le théorème~\ref{th-princ} qui constitue l'un des principaux résultats de l'article).

Dans ce qui suit, on note $\rad(A)$ le radical d'un foncteur $A$ de $\rep(A)$, c'est-à-dire l'intersection de ses sous-foncteurs stricts maximaux. Sous l'hypothèse (FH), si $A$ est de type fini, alors $A/\rad(A)$ est semi-simple fini (c'est le plus grand quotient semi-simple, ou cosocle, de $A$) ; si $N$ est un sous-foncteur de $A$, alors l'épimorphisme $A\twoheadrightarrow A/N$ est essentiel si et seulement si $N\subset\rad(A)$.

\begin{prop}\label{pr-QA_copres} Supposons que $\A$ est $k$-triviale et que $A$ est un foncteur de type fini et de type cofini de $\rep(\A)$. Alors la suite
$$0\to\Q(A)\to\Q_A\to\prod_{0\ne N\subset\rad(A)}\Q_{A/N}$$
dont le morphisme $\Q(A)\to\Q_A$ est l'inclusion canonique et le dernier morphisme a pour composantes les morphismes $\Q_A\twoheadrightarrow\Q_{A/N}$ induits par les projections $A\twoheadrightarrow B$ (cf. proposition~\ref{pr-foncQ}) est exacte.
\end{prop}

\begin{proof}
Pour $0\ne N\subset\rad(A)$, considérons le diagramme commutatif
$$\xymatrix{\Q^A\ar[r]^{\subset}\ar@{->>}[d] & k[A]\ar@{->>}[r]\ar@{->>}[d] & k[A/N]\ar@{->>}[d] \\
\Q(A)\ar[r]^{\subset} & \Q_A\ar@{->>}[r] & \Q_{A/N}
}$$
dont les flèches verticales sont les projections canoniques, les flèches horizontales de gauche  les inclusions et les flèches de droite sont induites par $A\twoheadrightarrow A/N$. La composée horizontale supérieure est nulle par définition de $\Q^A$, ce qui entraîne que la composée $\Q(A)\to\Q_A\to\Q_{A/N}$ est nulle.

Réciproquement, utilisons les identifications
$$\Q_A(x)\simeq k[\{\xi\in A(x)\,|\,A_\xi=A\}]\simeq k[\surj(\A(x,-),A)]$$
de la proposition~\ref{pr-descrQ} et considérons une décomposition $A\simeq S_1^{\oplus i_1}\oplus\dots\oplus S_n^{\oplus i_n}\oplus A'$ où les $S_j$ sont des simples deux à deux non isomorphes, les $i_j>0$ des entiers, et $A'$ est tel que, pour tout sous-objet simple $S$ de $A'$, l'épimorphisme $A'\twoheadrightarrow A'/S$ est essentiel. Considérons le morphisme composé
$$\theta : \Q_A(x)\simeq k[\{\xi\in A(x)\,|\,A_\xi=A\}]\hookrightarrow k[A(x)]\to k[A(x)]\,,$$
où la deuxième flèche est l'endomorphisme de
$$k[A(x)]\simeq k[S_1(x)^{\oplus i_1}]\otimes\dots\otimes k[S_n(x)^{\oplus i_n}]\otimes k[A'(x)]$$
donné par $\varepsilon_1\otimes\dots\otimes\varepsilon_n\otimes\mathrm{Id}$, où $\varepsilon_j$ est l'\emph{idempotent de Kov\'acs} \cite{Kov} de $k[\M_{i_j}(\mathrm{End}(S_j))]$, c'est-à-dire un idempotent central tel que $[1]-\varepsilon_j\in k[\M_{i_j}(\mathrm{End}(S_j))\setminus\GL_{i_j}(\mathrm{End}(S_j))]$ et $\varepsilon_j.k[\M_{i_j}(\mathrm{End}(S_j))\setminus\GL_{i_j}(\mathrm{End}(S_j))]=0$. Alors la composée
$$\Q_A(x)\xrightarrow{\theta}k[A(x)]\twoheadrightarrow\Q_A(x)$$
égale l'identité, grâce à la condition $[1]-\varepsilon_j\in k[\M_{i_j}(\mathrm{End}(S_j))\setminus\GL_{i_j}(\mathrm{End}(S_j))]$. Considérons la restriction $\theta'$ de $\theta$ au noyau de la flèche $\Q_A\to\prod_{0\ne N\subset\rad(A)}\Q_{A/N}$ évaluée en $x$ et vérifions qu'elle prend ses valeurs dans le sous-espace $\Q^A(x)$ de $k[A(x)]$. Pour cela, comme $A$ est de type cofini (ainsi, tout sous-foncteur non nul de $A$ contient un foncteur simple), il suffit de vérifier que, si $S$ est un sous-foncteur simple de $A$, alors la composée de $\theta'$ avec la projection $k[A(x)]\twoheadrightarrow k[A(x)/S(x)]$ est nulle. Si l'inclusion $S\hookrightarrow A$ est scindée, alors $S\subset S_j^{\oplus i_j}$ pour un entier $1\le j\le n$, et cette nullité résulte de ce que l'idempotent $\varepsilon_j$ de $k[\M_{i_j}(\mathrm{End}(S_j))]$ annule la projection $k[\mathrm{End}(S_j)^{\oplus i_j}]\twoheadrightarrow k[\mathrm{End}(S_j)^{\oplus i_j}/D]$ pour toute droite $D$ du $\mathrm{End}(S_j)$-espace vectoriel $\mathrm{End}(S_j)^{\oplus i_j}$. Sinon, $0\ne S\subset\rad(A)$, et l'annulation résulte de la définition, ce qui termine la démonstration.
\end{proof}

On rappelle que la notation $\Qse_d$ (resp. $\N_d$) qui intervient dans l'énoncé ci-dessous a été introduite juste avant le corollaire~\ref{cor-res_cosimpl} (resp. la proposition~\ref{pr-ressimpl}). On note également $\soc(A)$ le socle (i.e. la somme des sous-objets simples) d'un foncteur $A$ de $\rep(\A)$ et, si $T$ est un sous-foncteur de $A/\soc(A)$, on note $\tilde{T}$ le sous-foncteur de $A$ contenant $\soc(A)$ qui lui est canoniquement associé (qui est donc un sous-objet essentiel de $A$ si $A$ est de type cofini).

\begin{prop}\label{pr-resol_QA} Supposons que $\A$ est $k$-triviale et que $A$ est un foncteur fini de $\rep(\A)$. Alors il existe dans $\F(\A;k)$ des suites exactes
\begin{equation}\label{eq-seQ1}
    0\to\Q(A)\to\bigoplus_{N\in\Qse_0(\rad(A))}\Q_{A/N_0}\to\cdots\to\bigoplus_{N\in\Qse_d(\rad(A))}\Q_{A/N_0}\to\cdots
\end{equation}
et
\begin{equation}\label{eq-seQ2}
\cdots\to\bigoplus_{T\in\N_d(A/\soc(A))}\Q^{\tilde{T}_0}\to\cdots\to\bigoplus_{T\in\N_0(A/\soc(A))}\Q^{\tilde{T}_0}\to\Q^A\to\Q(A)\to 0
\end{equation}
équivariantes pour les actions canoniques $\mathrm{Aut}(A)$.
\end{prop}

\begin{proof} Comme $A$, et donc a fortiori $\rad(A)$ et $A/\soc(A)$, sont finis et que $\A$ vérifie (FH), les ensembles $\Qse_d(\rad(A))$ et $\N_d(A/\soc(A))$ sont finis, de sorte qu'on peut remplacer indifféremment toutes les sommes directes qui apparaissent dans l'énoncé et la présente démonstration par des produits. En particulier, au vu de la proposition~\ref{pr-dualQ}, il suffit de construire la première suite exacte, dont la deuxième se déduit par dualité.

Le début de la suite exacte~\eqref{eq-seQ1} est donné par la proposition~\ref{pr-QA_copres}.

Considérons l'ensemble, ordonné par inclusion, des sous-foncteurs non nuls de $\rad(A)$, vu comme une petite catégorie notée $\C$ : on dispose d'un foncteur $\C\to\F(\A;k)\quad N\mapsto\Q_{A/N}$, par la proposition~\ref{pr-foncQ}. Pour tout $d\in\mathbb{N}$, $\bigoplus_{N\in\Qse_d(\rad(A))}\Q_{A/N_0}$ n'est autre que les $d$-cosimplexes non dégénérés de l'objet cosimplicial de $\F(\A;k)$ associé. On obtient ainsi un complexe de cochaînes
\begin{equation}\label{eq-cossaux}
    \bigoplus_{N\in\Qse_0(\rad(A))}\Q_{A/N_0}\to\cdots\to\bigoplus_{N\in\Qse_d(\rad(A))}\Q_{A/N_0}\to\cdots
\end{equation}
équivariant pour les actions canoniques $\mathrm{Aut}(A)$ dont l'homologie en degré nul est $\Q(A)$ et dont il s'agit de montrer que l'homologie est nulle en degrés strictement positifs.

Soit $x$ un objet de $\A$. Fixons une section \emph{ensembliste} du morphisme surjectif de groupes abéliens $A(x)\twoheadrightarrow (A/\rad(A))(x)$, notée $\alpha\mapsto\tilde{\alpha}$. Si $N$ est un sous-foncteur de $\rad(A)$, un morphisme $\A(x,-)\to A/N$ de $\rep(A)$ est un épimorphisme si et seulement si sa composée avec la projection $A/N\twoheadrightarrow A/\rad(A)$ est un épimorphisme. La proposition~\ref{pr-descrQ} fournit donc un isomorphisme d'espaces vectoriels
\begin{equation}\label{eq-isointerm}
    \Q_{A/N}(x)\simeq\bigoplus_{\alpha\in\surj(\A(x,-),A/\rad(A))}k[(\rad(A)/N)(x)]
\end{equation}
linéarisation de la bijection associant à un élément $\xi$ de $(A/N)(x)$ tel que le morphisme $\A(x,-)\to A/N$ correspondant soit surjectif le couple formé de sa composée $\alpha$ avec $A/N\twoheadrightarrow A/\rad(A)$ (vue indifféremment comme élément de $(A/\rad(A))(x)$) et de $\xi-\pi(\tilde{\alpha})$, où $\pi : A(x)\twoheadrightarrow (A/N)(x)$ désigne la projection.

Les isomorphismes~\eqref{eq-isointerm} montrent que l'évaluation en $x$ du complexe \eqref{eq-cossaux} est isomorphe à la somme directe sur $\alpha\in\surj(\A(x,-),A/\rad(A))$ de l'évaluation en $x$ des complexes de cochaînes
$$\bigoplus_{N\in\Qse_0(\rad(A))}k[\rad(A)/N_0]\to\cdots\to\bigoplus_{N\in\Qse_d(\rad(A))}k[\rad(A)/N_0]\to\cdots$$
dont l'homologie est nulle en degrés strictement positifs d'après le corollaire~\ref{cor-res_cosimpl}.  Cela établit la proposition.
\end{proof}

\begin{lemm}\label{lm-stab_ord} Supposons que la catégorie $\A$ est $k$-triviale. Soient $A$ un foncteur de $\rep(\A)$, $d\in\mathbb{N}$ et $N\in\Qse_d(\rad(A))$. Notons $G$ le sous-groupe de $\mathrm{Aut}(A)$ stabilisateur de $N$. Alors le foncteur
$$k[\mathrm{Aut}(A/N_0)]\Md\to k[\mathrm{Aut}(A)]\Md$$
composé de la restriction $k[\mathrm{Aut}(A/N_0)]\Md\to k[G]\Md$ le long du morphisme de groupes canonique $G\to\mathrm{Aut}(A/N_0)$ et de l'induction $k[G]\Md\to k[\mathrm{Aut}(A)]\Md$
 préserve les modules projectifs.
\end{lemm}

\begin{proof}
Tout foncteur d'induction préserve les modules projectifs, il suffit donc de vérifier que le foncteur de restriction $k[\mathrm{Aut}(A/N_0)]\Md\to k[G]\Md$ préserve les modules projectifs. Il suffit pour cela de vérifier que l'ordre du noyau du morphisme canonique $G\to\mathrm{Aut}(A/N_0)$ (tous les groupes en jeu sont finis) est inversible dans $k$. Soit $G'$ le sous-groupe de $\mathrm{Aut}(A)$ stabilisateur de $N_0$ : $G$ est un sous-groupe de $G'$, il suffit donc de montrer que l'ordre du noyau $H$ du morphisme canonique $G'\to\mathrm{Aut}(A/N_0)$ est inversible dans $k$. Or $H$ est isomorphe au groupe abélien $\mathrm{Hom}(A/N_0,N_0)$, l'isomorphisme s'obtenant en associant à $\varphi\in\mathrm{Hom}(A/N_0,N_0)$ la somme de l'identité de $A$ et de la composée $A\twoheadrightarrow A/N_0\xrightarrow{\varphi} N_0\hookrightarrow A$. Comme $A/N_0$ (resp. $N_0$) est un foncteur de type fini (resp. cofini) de $\rep(\A)$, c'est un quotient d'un foncteur représentable $\A(x,-)$ (resp. un sous-foncteur d'un $\A(-,y)^\sharp$), de sorte que $H$ est isomorphe à un sous-groupe de $\A(x,y)^\sharp$, dont l'ordre est inversible dans $k$ puisque $\A$ est $k$-triviale, d'où le lemme.
\end{proof}

\begin{coro}\label{cor-val_projQ} Supposons que $\A$ est $k$-triviale et que $A$ est un foncteur fini de $\rep(\A)$. Alors $\Q(A)$ est à valeurs projectives sur $k[\mathrm{Aut}(A)]$.
\end{coro}

\begin{proof}
Comme $A$, donc $\rad(A)$, est fini, on a $\Qse_d(\rad(A))=\varnothing$ pour $d$ assez grand : la suite exacte \eqref{eq-seQ1} est finie. La structure de $k[\mathrm{Aut}(A)]$-module sur $\bigoplus_{N\in\Qse_d(\rad(A))}\Q_{A/N_0}$ s'obtient en appliquant à $\bigoplus_N\Q_{A/N_0}$, où la somme est prise sur un ensemble complet de représentants des classes d'équivalence de $\Qse_d(\rad(A))$ modulo l'action de $\mathrm{Aut}(A)$, le foncteur $$k[\mathrm{Aut}(A/N_0)]\Md\to k[\mathrm{Aut}(A)]\Md$$
du lemme~\ref{lm-stab_ord}. Comme $Q_{A/N_0}$ est à valeurs dans les $k[\mathrm{Aut}(A/N_0)]$-modules projectifs (cf. proposition~\ref{pr-Qexact}), le lemme~\ref{lm-stab_ord}, la proposition~\ref{pr-resol_QA} et le fait qu'un $k[\mathrm{Aut}(A)]$-module est projectif si et seulement s'il est injectif ($k$ étant un corps et $\mathrm{Aut}(A)$ un groupe fini) montrent que les valeurs de $\Q(A)$ possèdent une résolution injective finie et sont donc projectives comme $k[\mathrm{Aut}(A)]$-modules, comme souhaité.
\end{proof}

Le résultat suivant (qui est trivial lorsque $k$ est de caractéristique nulle) découle de la définition de $\Q(A,M)$ et du corollaire~\ref{cor-val_projQ}.

\begin{coro}\label{cor-QAMequiv}
Supposons que $\A$ est $k$-triviale et que $A$ est un foncteur fini de $\rep(\A)$. Alors on dispose d'isomorphismes
$$\Q(A,M)\simeq (\Q(A)\otimes M)_{\mathrm{Aut}(A)}\simeq (\Q(A)\otimes M)^{\mathrm{Aut}(A)}$$
naturels en la représentation $M$ du groupe $\mathrm{Aut}(A)$.

En particulier, le foncteur $\Q(A,-) : k[\mathrm{Aut}(A)]\Md\to\F(\A;k)$ est bicontinu.
\end{coro}

\begin{rema}
On vérifie facilement à partir des résultats de cette section que, si $\A$ est $k$-triviale, que $A$ est un foncteur fini de $\rep(\A)$ et $M$ un $k[\mathrm{Aut}(A)]$-module, alors l'inclusion $\Q(A,M)\hookrightarrow\Q_{A,M}$ est un isomorphisme si et seulement si $M_{\mathrm{Hom}(A/S,S)}=0$ pour tout sous-foncteur simple $S$ de $\rad(A)$, l'indice $\mathrm{Hom}(A/S,S)$ désignant les co-invariants sous l'action de ce groupe, plongé dans $\mathrm{Aut}(A)$ comme dans la démonstration du lemme~\ref{lm-stab_ord}.
\end{rema}

\section{Foncteurs de décalage et de différence paraboliques}\label{sdecp}

\subsection{Les foncteurs $\tb_x$}

Soient $x$ et $a$ des objets de $\A$ et $F$ un foncteur de $\F(\A;k)$. Le groupe abélien $\A(x,a)$ opère sur $F(x\oplus a)$ par l'intermédiaire du monomorphisme de groupes
$$\A(x,a)\hookrightarrow\mathrm{Aut}_\A(x\oplus a)\qquad u\mapsto\left(\begin{array}{cc}
 \mathrm{Id}_x    &  0\\
    u & \mathrm{Id}_a
\end{array}\right).$$

De plus, si $f : a\to b$ est une flèche de $\A$, alors il existe une unique flèche $f_*$ faisant commuter le diagramme de $k$-modules suivant :
$$\xymatrix{F(x\oplus a)\ar[r]^-{F(x\oplus f)}\ar@{->>}[d] & F(x\oplus b)\ar@{->>}[d]\\
F(x\oplus a)_{\A(x,a)}\ar[r]^-{f_*} & F(x\oplus b)_{\A(x,b)}
}$$
car le diagramme
$$\xymatrix{x\oplus a\ar[rr]^-{x\oplus f}\ar[dd]_-{\left(\begin{array}{cc}
 \mathrm{Id}_x    &  0\\
    u & \mathrm{Id}_a
\end{array}\right)} & & x\oplus b\ar[dd]^-{\left(\begin{array}{cc}
 \mathrm{Id}_x    &  0\\
    fu & \mathrm{Id}_b
\end{array}\right)}\\
& & \\
x\oplus a\ar[rr]^-{x\oplus f} & & x\oplus b
}$$
de $\A$ commute pour tout $u\in\A(x,a)$.

Par conséquent, la projection canonique $\tau_x(F)(a)=F(x\oplus a)\twoheadrightarrow F(x\oplus a)_{\A(x,a)}$ fait de $(F(x\oplus a)_{\A(x,a)})_{a\in\A}$ un foncteur quotient de $\tau_x(F)$. On note $\tb_x(F)$ ce foncteur. Il est clair que $\tb_x$ définit un quotient de l'endofoncteur $\tau_x$ de $\F(\A;k)$. On l'appelle foncteur de \emph{décalage parabolique} associé à $x$. La terminologie provient de l'analogie avec la théorie des représentations des groupes algébriques (comme représentation de $\mathrm{Aut}(a)$, $\tb_x(F)(a)$ s'obtient en restreignant la représentation $F(x\oplus a)$ de $\mathrm{Aut}(x\oplus a)$ au sous-groupe parabolique $\big(\mathrm{Aut}(a)\times\mathrm{Aut}(x)\big)\ltimes\A(x,a)$ puis en prenant les co-invariants sous l'action de $\A(x,a)$). Cet outil a été introduit (avec la notation $\bar{\Sigma}$ plutôt que $\tb$) dans un cadre de catégories de foncteurs similaire au nôtre par Nagpal \cite[§\,4.2]{Nag1} (la source étant la sous-catégorie des \emph{monomorphismes} des espaces vectoriels de dimension finie sur un corps fini chez Nagpal). Nous utiliserons toutefois les foncteurs de décalage parabolique d'une façon assez différente de celle employée par Nagpal. 

\begin{prop}\label{pr-taubarexact} Pour tout objet $x$ de $\A$, l'endofoncteur $\tb_x$ de $\F(\A;k)$ est cocontinu. Si $\A$ est $k$-triviale, alors $\tb_x$ est bicontinu, et en particulier exact.
\end{prop}

\begin{proof}
Le foncteur des co-invariants sous l'action d'un groupe est cocontinu, et lorsque le groupe est fini d'ordre inversible dans la catégorie abélienne sur laquelle il agit, ce foncteur est bicontinu, d'où le résultat.
\end{proof}

Nous verrons un peu plus loin (remarque~\ref{rq-eqcar}) que les foncteurs $\tb_x$ ne sont généralement pas exacts à gauche, lorsque $\A$ n'est pas supposée $k$-triviale.

On prendra garde par ailleurs que $x\mapsto\tb_x$ ne définit \emph{pas} un quotient du \emph{foncteur} $x\mapsto\tau_x$ --- sinon, la naturalité en $x\in\A$ entraînerait que l'identité de $\F(\A;k)$ est facteur direct de $\tb_x$, ce qui n'est pas le cas (cf. remarque~\ref{rq-nonscinde} ci-après). Toutefois, une forme faible de fonctorialité en $x$ de la construction $\tb_x$ subsiste : si $\varphi : x\to y$ est un monomorphisme scindé de $\A$, alors tout scindement $\psi : y\to x$ de $\varphi$ procure, pour tout $a\in\mathrm{Ob}\,\A$ et tout $u\in\A(x,a)$, un diagramme commutatif 
$$\xymatrix{x\oplus a\ar[rr]^-{\varphi\oplus a}\ar[dd]_-{\left(\begin{array}{cc}
 \mathrm{Id}_x    &  0\\
    u & \mathrm{Id}_a
\end{array}\right)} & & y\oplus a\ar[dd]^-{\left(\begin{array}{cc}
 \mathrm{Id}_y    &  0\\
    u\psi & \mathrm{Id}_a
\end{array}\right)}\\
& & \\
x\oplus a\ar[rr]^-{\varphi\oplus a} & & y\oplus a
}$$
dans $\A$, d'où l'existence d'une (unique) flèche en pointillé faisant commuter le diagramme ci-dessous.
$$\xymatrix{F(x\oplus a)\ar[r]^-{F(\varphi\oplus a)}\ar@{->>}[d] & F(y\oplus a)\ar@{->>}[d]\\
F(x\oplus a)_{\A(x,a)}\ar@{-->}[r] & F(y\oplus a)_{\A(y,a)}
}$$
Ainsi, $x\mapsto\tb_x$ définit un sous-foncteur de la restriction de $x\mapsto\tau_x$ à la sous-catégorie des morphismes scindés de $\A$.

\subsection{Les foncteurs $\db_x$}

En particulier, l'unique morphisme $0\to x$ induit une transformation naturelle $\mathrm{Id}_{\F(\A;k)}\simeq\tb_0\to\tb_x$, dont le conoyau sera noté $\db_x$. On le nomme foncteur de \emph{différence parabolique} associé à $x$. Nos foncteurs $\db$ sont analogues aux foncteurs $\bar{\Delta}$ de Nagpal \cite[§\,4.2]{Nag1}.

\begin{prop}\label{pr-delbarexact} Pour tout objet $x$ de $\A$, l'endofoncteur $\db_x$ de $\F(\A;k)$ est cocontinu. Si $\A$ est $k$-triviale, alors $\db_x$ est bicontinu, et en particulier exact. De plus, la transformation naturelle $\mathrm{Id}_{\F(\A;k)}\to\tb_x$ est un monomorphisme, et $\db_x(F)(a)$ s'identifie aux co-invariants de $\delta_x(F)(a)$ sous l'action de $\A(x,a)$ déduite du monomorphisme canonique $\delta_x(F)(a)\hookrightarrow\tau_x(F)(a)=F(x\oplus a)$.
\end{prop}

\begin{proof} La première assertion est immédiate. Supposons maintenant que $\A$ est $k$-triviale.

Le monomorphisme canonique $F(a)\hookrightarrow F(x\oplus a)$ est équivariant pour l'action de $\A(x,a)$, qu'on fait opérer \emph{trivialement} sur  $F(a)$, car le diagramme
$$\xymatrix{a\ar[r]\ar[rdd] & x\oplus a\ar[dd]^-{\left(\begin{array}{cc}
 \mathrm{Id}_x    &  0\\
    u & \mathrm{Id}_a
\end{array}\right)}\\
& \\
 & x\oplus a
}$$
de $\A$ (dont les flèches horizontale et oblique sont le monomorphisme canonique) commute pour tout $u\in\A(x,a)$. Par conséquent, on dispose d'une suite exacte $\A(x,a)$-équivariante
$$0\to F(a)\to\tau_x(F)(a)\to\delta_x(F)(a)\to 0\,,$$
de sorte que la conclusion résulte de l'exactitude (et de la continuité) des co-invariants sous l'action du groupe $\A(x,a)$ lorsque $\A$ est $k$-triviale.
\end{proof}

\begin{rema}\label{rq-nonscinde}
On prendra garde que, contrairement au monomorphisme canonique $\mathrm{Id}\hookrightarrow\tau_x$, le monomorphisme canonique $\mathrm{Id}\hookrightarrow\tb_x$ n'est généralement pas scindé --- cf. remarque~\ref{rq-tbQd} ci-après, qui montrera même qu'il n'existe en général aucun épimorphisme $\tb_x\twoheadrightarrow\mathrm{Id}$. En particulier, $\tb_x$ n'est pas auto-dual : il n'existe (généralement) pas d'isomorphisme $\tb_x(F^\vee)\simeq\tb_x(F)^\vee$ dans $\F(\A^\op;k)$ naturel en le foncteur $F$ de $\F(\A;k)$. 
\end{rema}

\subsection{(Non-)commutation}

Les foncteurs $\tau_x$ et $\delta_y$ usuels possèdent de nombreuses propriétés de commutation : on a $\tau_x\circ\tau_y\simeq\tau_{x\oplus y}\simeq\tau_y\circ\tau_x$ et $\tau_x\circ\delta_y\simeq\delta_y\circ\tau_x$. On perd ces propriétés avec les foncteurs $\tb_x$ et $\db_y$. Nous donnons toutefois ci-dessous une propriété relative à des compositions de foncteurs qui interviendront dans la section~\ref{sFd}.

\begin{prop}\label{pr-comtbdb} Soient $x$ et $y$ des objets de $\A$. Il existe un diagramme commutatif
$$\xymatrix{\tb_y\circ\tau_x\ar@{->>}[r]\ar@{->>}[d] & \tau_x\circ\tb_y\ar@{->>}[d] \\
\db_y\circ\tau_x\ar@{->>}[r] & \tau_x\circ\db_y
}$$
d'endofoncteurs de $\F(\A;k)$ dont toutes les flèches sont des épimorphismes.
\end{prop}

\begin{proof} Soit $t$ un objet de $\A$. On a $(\tau_y\circ\tau_x)(F)(t)=F(x\oplus y\oplus t)$ et $(\tb_y\circ\tau_x)(F)(t)=(\tau_x F)(y\oplus t)_{\A(y,t)}=F(x\oplus y\oplus t)_{\A(y,t)}$, tandis que $(\tau_x\circ\tau_y)(F)(t)=F(y\oplus x\oplus t)$ et $(\tau_x\circ\tb_y)(F)(t)=\tb_y(F)(x\oplus t)=F(y\oplus x\oplus t)_{\A(y,x\oplus t)}$. Comme la conjugaison par l'isomorphisme canonique $x\oplus y\oplus t\xrightarrow{\simeq}y\oplus x\oplus t$ envoie le sous-groupe $\A(y,t)$ de $\mathrm{Aut}(x\oplus y\oplus t)$ dans le sous-groupe $\A(y,x\oplus t)$ de $\mathrm{Aut}(y\oplus x\oplus t)$, on peut donc trouver une flèche en pointillé faisant commuter le diagramme
$$\xymatrix{\tau_y\circ\tau_x\ar[r]^-\simeq\ar@{->>}[d] & \tau_x\circ\tau_y\ar@{->>}[d] \\
\tb_y\circ\tau_x\ar@{-->}[r] & \tau_x\circ\tb_y
}$$
dans lequel la flèche supérieure est l'isomorphisme canonique et les flèches verticales sont les épimorphismes canoniques. On obtient ainsi un épimorphisme $\tb_y\circ\tau_x\twoheadrightarrow\tau_x\circ\tb_y$ tel que le diagramme
$$\xymatrix{\tau_x\ar@{=}[r]\ar[d]_-{(\mathrm{Id}\hookrightarrow\tb_y)\circ\tau_x} & \tau_x\ar[d]^-{\tau_x\circ(\mathrm{Id}\hookrightarrow\tb_y)} \\
\tb_y\circ\tau_x\ar@{->>}[r] & \tau_x\circ\tb_y
}$$
commute, puisque le diagramme
$$\xymatrix{\tau_x\ar@{=}[r]\ar[d]_-{(\mathrm{Id}\hookrightarrow\tau_y)\circ\tau_x} & \tau_x\ar[d]^-{\tau_x\circ(\mathrm{Id}\hookrightarrow\tau_y)} \\
\tau_y\circ\tau_x\ar[r]^-\simeq & \tau_x\circ\tau_y
}$$
commute lui-même. En considérant le conoyau des flèches verticales, on en déduit la proposition.
\end{proof}

Nous verrons dans l'exemple~\ref{ex-commdt} que les épimorphismes de la proposition précédente ne sont pas des isomorphismes (et même qu'il n'existe pas d'isomorphisme entre $\tb_y\circ\tau_x$ et $\tau_x\circ\tb_y$, par exemple).

\subsection{Calcul sur les $k[A]$}

\begin{prop}\label{pr-taubarlin}
Soient $A$ un foncteur de $\rep(\A)$ et $x$ un objet de $\A$. Il existe dans $\F(\A;k)$ un isomorphisme
$$\tb_x (k[A])\simeq\bigoplus_{\xi\in A(x)}k[A/A_\xi]$$
naturel en $A$ au sens où, pour tout morphisme $f : A\to B$ de $\rep(\A)$, le morphisme
$$\bigoplus_{\xi\in A(x)}k[A/A_\xi]\simeq\tb_x (k[A])\xrightarrow{\tb_x(k[f])}\tb_x (k[B])\simeq\bigoplus_{\zeta\in B(x)}k[B/B_\zeta]$$
a pour composante $k[A/A_\xi]\to k[B/B_\zeta]$ le morphisme $k[\bar{f}]$, où $\bar{f} : A/A_\xi\to B/B_\zeta$ est induit par $f$, lorsque $f_*\xi=\zeta$ et $0$ sinon,
et faiblement naturel en $x$ au sens où pour tout monomorphisme scindé $\varphi : x\to y$ de $\A$, le morphisme
$$\bigoplus_{\xi\in A(x)}k[A/A_\xi]\simeq\tb_x (k[A])\xrightarrow{\tb_\varphi(k[A])}\tb_y (k[A])\simeq\bigoplus_{\zeta\in A(y)}k[A/A_\zeta]$$
a pour composante $k[A/A_\xi]\to k[A/A_\zeta]$ l'identité si $A(\varphi)(\xi)=\zeta$ (ce qui implique $A_\xi=A_\zeta$) et $0$ sinon. 
\end{prop}

\begin{proof} Soit $a\in\mathrm{Ob}\,\A$. L'action de $\A(x,a)$ sur
\begin{equation}\label{eq-tlin}
    \tau_x(k[A])(a)=k[A(x\oplus a)]\simeq\bigoplus_{\xi\in A(x)}k[A(a)]
\end{equation}
est donnée par $f.(\xi,[\alpha])=(f,[f_*\xi+\alpha])$, où $(\xi,[\alpha])$ désigne l'image de $[\alpha]\in k[A(a)]$ par l'injection du facteur étiqueté par $\xi$. On en tire un isomorphisme
$$\tb_x(k[A])(a)\simeq\bigoplus_{\xi\in A(x)}k[A(a)]/<[f_*\xi+\alpha]-[\alpha]\,|\,(f,\alpha)\in\A(x,a)\times A(a)>$$
$$\simeq\bigoplus_{\xi\in A(x)}k[A(a)]/<[\beta]-[\alpha]\,|\,\beta-\alpha\in A_\xi(a)>\;\simeq\bigoplus_{\xi\in A(x)}k[A(a)/A_\xi(a)].$$
Les fonctorialités en $A$, $a$ et $x$ de ces isomorphismes se déduisent directement de celles de la formule~\eqref{eq-tlin}.
\end{proof}

\begin{coro}\label{cor-delbarlin} Soient $A$ un foncteur de $\rep(\A)$ et $x$ un objet de $\A$. Il existe dans $\F(\A;k)$ un isomorphisme
$$\db_x (k[A])\simeq\bigoplus_{\xi\in A(x)\setminus\{0\}}k[A/A_\xi].$$
\end{coro}

\begin{exem}\label{ex-commdt}
Évalué sur $k[A]$, l'épimorphisme $\tb_y\circ\tau_x\twoheadrightarrow\tau_x\circ\tb_y$ (resp. $\db_y\circ\tau_x\twoheadrightarrow\tau_x\circ\db_y$) de la proposition~\ref{pr-comtbdb} s'identifie à la somme directe sur les $\xi\in A(x)$ (resp. $\xi\in A(x)\setminus\{0\}$) des épimorphismes canoniques $k[A/A_\xi]\otimes k[A(x)]\twoheadrightarrow k[A/A_\xi]\otimes k[(A/A_\xi)(x)]$, qui ne sont généralement pas des isomorphismes (et il n'existe même généralement aucun isomorphisme entre leur but et leur source).
\end{exem}

\subsection{Calcul sur les $\Q_A$}

\begin{prop}\label{pr-tbQ} Soient $A$ un foncteur de $\rep(\A)$ et $x$ un objet de $\A$. Il existe dans $\F(\A;k)$ des isomorphismes $\mathrm{Aut}(A)$-équivariants
$$\tb_x(\Q_A)\simeq\bigoplus_{\xi\in A(x)}\Q_{A/A_\xi}\qquad\text{et}\qquad\db_x(\Q_A)\simeq\bigoplus_{\xi\in A(x)\setminus\{0\}}\Q_{A/A_\xi}\,,$$
où l'action de $\varphi\in\mathrm{Aut}(A)$ sur les membres de droite a pour composante $\Q_{A/A_\xi}\to\Q_{A/A_\zeta}$ l'isomorphisme induit par $\varphi$ si $\varphi_*\xi=\zeta$ et $0$ sinon.
\end{prop}

\begin{proof} La cocontinuité du foncteur $\tb_x$ (proposition~\ref{pr-taubarexact}) et la proposition~\ref{pr-taubarlin} fournissent un isomorphisme
\begin{equation}\label{eq-tbQA}
    \tb_x(\Q_A)\simeq\mathrm{Coker}\,\Big(\underset{\zeta\in B(x)}{\bigoplus_{B\subsetneq A}} k[B/B_\zeta]\to\bigoplus_{\xi\in A(x)}k[A/A_\xi]\Big)
\end{equation}
dont la composante $k[B/B_\zeta]\to k[A/A_\xi]$ est la linéarisation du morphisme $B/B_\zeta\to A/A_\xi$ qu'induit l'inclusion $B\hookrightarrow A$ si $\zeta=\xi$ dans $A(x)\supset B(x)$ et $0$ sinon. Or pour $B\subset A$ et $\xi\in A(x)$, la relation $\xi\in B(x)$ équivaut à $A_\xi\subset B$, et l'on a alors $B_\xi=A_\xi$, de sorte qu'on déduit de~\eqref{eq-tbQA} des isomorphismes
$$\tb_x(\Q_A)\simeq\bigoplus_{\xi\in A(x)}\mathrm{Coker}\,\Big(\bigoplus_{A_\xi\subset B\subsetneq A}k[B/A_\xi]\to k[A/A_\xi]\Big)$$
$$\simeq\bigoplus_{\xi\in A(x)}\mathrm{Coker}\,\Big(\bigoplus_{T\subsetneq A/A_\xi}k[T]\to k[A/A_\xi]\Big)=\bigoplus_{\xi\in A(x)}\Q_{A/A_\xi}.$$

L'équivariance résulte de la naturalité en $A$ des isomorphismes de la proposition~\ref{pr-taubarlin} ; la démonstration du deuxième isomorphisme est entièrement analogue à partir du corollaire~\ref{cor-delbarlin}.
\end{proof}

La remarque ci-dessous illustre que la proposition~\ref{pr-tbQ} n'a pas d'analogue direct pour les foncteurs $\Q^A$, même avec une catégorie source $k$-triviale, et en particulier que les foncteurs $\tb_x$ et $\db_x$ ne sont pas auto-duaux (cf. remarque~\ref{rq-nonscinde}).

\begin{rema}\label{rq-tbQd} Supposons que $\A$ est $k$-triviale et que $0\to S\to A\to T\to 0$ est une suite exacte \emph{non scindée} de $\rep(\A)$, avec $S$ et $T$ simples non isomorphes.

À partir de la suite exacte
\begin{equation}\label{sec-nsc}
0\to\Q^A\to k[A]\to k[T]\to 0\;,    
\end{equation}
les propositions~\ref{pr-taubarexact} et~\ref{pr-taubarlin} fournissent une suite exacte
$$0\to\tb_x(\Q^A)\to\bigoplus_{\xi\in A(x)}k[A/A_\xi]\to\bigoplus_{\zeta\in T(x)}k[T/T_\zeta]\to 0\;;$$
pour $\xi\in A(x)\setminus\{0\}$, si $A_\xi\ne A$, alors $A_\xi=S$, et la restriction au facteur étiqueté par $\xi$ du morphisme de droite est l'injection du facteur direct correspondant à $\zeta=0$. On en déduit une suite exacte
$$0\to\tb_x(\Q^A)^\red\to k[A]^\red\oplus (k[T]^\red)^{\oplus n}\to k[T]^\red\to 0$$
où $n$ est le cardinal de $S(x)\setminus\{0\}$ et la flèche de droite a pour composante les parties réduites de la linéarisation $k[A]\to k[T]$ de la projection $A\twoheadrightarrow T$ et l'identité de $k[T]$ sur chacun des $n$ facteurs. Si $n>0$ (i.e. si $S(x)\ne 0$), cette suite exacte est scindée procure un isomorphisme $\tb_x(\Q^A)^\red\simeq k[A]^\red\oplus (k[T]^\red)^{\oplus n-1}$. Or il n'y a pas de morphisme non nul de $k[T]$ vers $\Q^A$ grâce à la proposition~\ref{pr-morlin}, qui montre que $\mathrm{Hom}(k[T],\Q^A)$ s'identifie au noyau de l'application linéaire $k[\mathrm{Hom}(T,A)]\to k[\mathrm{Hom}(T,T)]$ linéarisation du morphisme $\mathrm{Hom}(T,A)\to\mathrm{Hom}(T,T)$ induit par la projection $A\twoheadrightarrow T$. En effet, ce morphisme est injectif vu la suite exacte \eqref{sec-nsc} ($S$ et $T$ sont simples et non isomorphes). On voit de façon analogue (en utilisant que \eqref{sec-nsc} est non scindée) que $\mathrm{Hom}(k[A],\Q^A)=0$.

Il s'ensuit qu'il n'y a pas de morphisme non nul $\tb_x(\Q^A)\to\Q^A$ ; en particulier, la suite exacte $0\to\Q^A\to\tb_x(\Q^A)\to\db_x(\Q^A)\to 0$ n'est pas scindée.
\end{rema}

\subsection{Calcul sur les $\Q_{A,M}$}

Précisons tout d'abord quelques notations utilisées dans la suite de ce paragraphe. Si $\xi$ est un élément de l'ensemble quotient $A(x)/\mathrm{Aut}(A)$ (où $x$ et $A$ sont des objets de $\A$ et $\rep(\A)$ respectivement), le sous-foncteur $A_{\tilde{\xi}}$ de $A$ ne dépend pas du choix du relevé $\tilde{\xi}$ de $\xi$ dans $A(x)$, on le notera donc simplement $A_\xi$. Par ailleurs, on note $\Pi(\xi)$ le sous-groupe de $\mathrm{Aut}(A)$ stabilisateur de $\tilde{\xi}$, qui ne dépend que de $\xi$ \emph{à conjugaison près}. Tout élément de $\Pi(\xi)$ induit un automorphisme de $A/A_\xi$, d'où un morphisme de groupes $\Pi(\xi)\to\mathrm{Aut}(A/A_\xi)$ dont le noyau s'identifie au groupe additif $\mathrm{Hom}(A/A_\xi,A_\xi)$, plongé dans $\mathrm{Aut}(A)$ par le morphisme envoyant $g$ sur la somme de l'identité de $A$ et du morphisme $A\twoheadrightarrow A/A_\xi\xrightarrow{g}A_\xi\hookrightarrow A$. On note $\R_\xi : k[\mathrm{Aut}(A)]\Md\to k[\mathrm{Aut}(A/A_\xi)]\Md$ le foncteur de restriction parabolique donné par la composée de la restriction $k[\mathrm{Aut}(A)]\Md\to k[\Pi(\xi)]\Md$ à $\Pi(\xi)$ et du foncteur $k[\mathrm{Aut}(A/A_\xi)]\underset{k[\Pi(\xi)]}{\otimes} -$, qui s'identifie d'après l'observation précédente à la composée du foncteur des co-invariants sous l'action de $\mathrm{Hom}(A/A_\xi,A_\xi)$ et de l'induction de l'image de $\Pi(\xi)$ dans $\mathrm{Aut}(A/A_\xi)$ à $\mathrm{Aut}(A/A_\xi)$.

\begin{prop}\label{pr-tbQAM} Soient $A$ un foncteur de $\rep(\A)$, $M$ un $k[\mathrm{Aut}(A)]$-module et $x$ un objet de $\A$. Il existe dans $\F(\A;k)$ des isomorphismes
$$\tb_x(\Q_{A,M})\simeq\bigoplus_{\xi\in A(x)/\mathrm{Aut}(A)}\Q_{A/A_\xi,\R_\xi(M)}\qquad\text{et}$$
$$\db_x(\Q_{A,M})\simeq\bigoplus_{\xi\in (A(x)\setminus\{0\})/\mathrm{Aut}(A)}\Q_{A/A_\xi,\R_\xi(M)}$$
naturels en $M$.
\end{prop}

\begin{proof}
Nous nous contenterons d'établir le premier isomorphisme, le second se montrant de façon similaire (ou s'en déduisant).

La proposition~\ref{pr-tbQ} fournit un isomorphisme $\mathrm{Aut}(A)$-équivariant
$$\tb_x(\Q_A)\simeq\bigoplus_{\xi\in A(x)/\mathrm{Aut}(A)}\Q_{A/A_\xi}\uparrow_{\Pi(\xi)}^{\mathrm{Aut}(A)}\;,$$
d'où, comme $\tb_x$ est cocontinu (proposition~\ref{pr-taubarexact}), des isomorphismes
$$\tb_x(\Q_{A,M})\simeq M\underset{k[\mathrm{Aut}(A)]}{\otimes}\tb_x(\Q_A)\simeq\bigoplus_{\xi\in A(x)/\mathrm{Aut}(A)}(M\downarrow^{\mathrm{Aut}(A)}_{\Pi(\xi)})\underset{k[\Pi(\xi)]}{\otimes}\Q_{A/A_\xi}$$
$$\simeq\bigoplus_{\xi\in A(x)/\mathrm{Aut}(A)}\R_\xi(M)\underset{k[\mathrm{Aut}(A/A_\xi)]}{\otimes}\Q_{A/A_\xi}=\bigoplus_{\xi\in A(x)/\mathrm{Aut}(A)}\Q_{A/A_\xi,\R_\xi(M)}$$
naturels en le $k[\mathrm{Aut}(A)]$-module $M$.
\end{proof}

\begin{rema}\label{rq-annuldb} Si la catégorie $\A$ est $k$-triviale et le foncteur $A$ de type fini, alors $\mathrm{Hom}(A/A_\xi,A_\xi)$ est un groupe fini d'ordre inversible dans $k$, de sorte que le foncteur $\R_\xi$ est \emph{exact}. Il n'est en revanche généralement pas fidèle.

La proposition~\ref{pr-tbQAM} montre toutefois que, si $A$ est un foncteur de type fini non nul de $\rep(\A)$ et $M$ une représentation non nulle de $\mathrm{Aut}(A)$, alors il existe un objet $x$ de $\A$ tel que $\db_x(\Q_{A,M})$ soit non nul, car on peut trouver $\xi$ tel que $A_\xi=A$ si $x$ est bien choisi, de sorte que le facteur $\Q_{A/A_\xi,\R_\xi(M)}=\Q_{0,M}$ est un foncteur constant non nul. Néanmoins, on peut trouver des foncteurs $A$ de $\rep(\A)$ de longueur finie et \emph{arbitrairement grande} (si $\A$ est $k$-triviale et non nulle) tels qu'il existe une représentation irréductible $M$ de $\mathrm{Aut}(A)$ telle que $\db_x(\Q_{A,M})$ soit \emph{constant} pour tout $x\in\mathrm{Ob}\,\A$. Ainsi, si $\A=\mathbf{P}(\FF)$ pour un corps fini $\FF$ de caractéristique distincte de celle de $k$, identifiant $\rep(\A)$ aux $\FF$-espaces vectoriels, pour $A:=\FF^d$ et $M$ une représentation $k$-linéaire irréductible \emph{cuspidale} de $\GL_d(\FF)$, on a $\R_\xi(M)=0$ pour tout $\xi\in A(x)\setminus\{0\}$ tel que $A_\xi\ne 0$, de sorte que $\db_x(\Q_{A,M})$ se réduit à son terme constant (correspondant aux $\xi$ tels que $A_\xi=A$ dans l'isomorphisme de la proposition~\ref{pr-tbQAM}). Cet exemple se propage facilement à toute catégorie $k$-triviale non nulle $\A$ par changement additif de base à la source.
\end{rema}

\begin{rema}\label{rq-eqcar}
Supposons que $\tb_x$ ou $\db_x$ est un foncteur exact. Alors la proposition~\ref{pr-tbQAM} implique que, pour tout foncteur $A$ de $\rep(\A)$ et tout $\xi\in (A(x)\setminus\{0\})/\mathrm{Aut}(A)$, le foncteur
$$k[\mathrm{Aut}(A)]\Md\to\F(\A;k)\qquad M\mapsto \Q_{A/A_\xi,\R_\xi(M)}$$
est exact, et donc, si $A$ est de type fini, que le foncteur $\R_\xi : k[\mathrm{Aut}(A)]\Md\to k[\mathrm{Aut}(A/A_\xi)]\Md$ est également exact, grâce à la proposition~\ref{pr-Qexact}. Ce n'est généralement pas le cas lorsque $\A$ n'est pas $k$-triviale. Par exemple, si $R$ est un anneau fini dont la caractéristique est une puissance de celle de $k$, alors le foncteur $k[\GL_2(R)]\Md\to k\Md$ des co-invariants sous l'action du groupe additif sous-jacent à $R$ (plongé dans $\GL_2(R)$ via les matrices triangulaires supérieures strictes) n'est pas exact à gauche (par exemple, l'injection diagonale $k\hookrightarrow k^{\GL_2(R)}$ est équivariante, où $\GL_2(R)$ opère trivialement à la source, mais induit le morphisme nul entre les co-invariants, qui sont non nuls). On en déduit facilement que si $\A$ vérifie la condition (FH) mais n'est pas $k$-triviale, alors il existe un objet $x$ de $\A$ tel que ni $\tb_x$ ni $\db_x$ ne soient exacts à gauche dans $\F(\A;k)$.
\end{rema}

\section{Les sous-catégories $\F_d(\A;k)$}\label{sFd}

\begin{defi}\label{df-Fd} Soit $d\ge -1$ un entier. On note $\F_d(\A;k)$ la sous-catégorie pleine de $\F(\A;k)$ constituée des foncteurs $F$ tels que $\db_{a_0}\dots\db_{a_d}\tau_x(F)=0$ pour tout $(a_0,\dots,a_d,x)\in\mathrm{Ob}\,\A^{d+2}$.
\end{defi}

On remarquera que le dernier endofoncteur de $\F(\A;k)$ qui intervient dans la composition ci-dessus est $\tau_x$ et non $\tb_x$ ; on voit facilement (en s'appuyant sur la remarque ci-après, par exemple) que l'omission du terme $\tau_x$ (ou son remplacement par $\tb_x$) donnerait lieu à une notion non équivalente, et moins intéressante.

\begin{rema}\label{rq-Fd-def}
La définition précédente est très semblable à la définition de foncteur polynomial de degré au plus~$d$ rappelée au début de la section~\ref{sfpap}, \emph{à l'exception de l'ajout du terme $\tau_x$} dans la définition~\ref{df-Fd}. Il n'est pas difficile de voir qu'un foncteur $F$ de $\F(\A;k)$ est polynomial de degré au plus~$d$ si et seulement si $\delta_{a_0}\dots\delta_{a_d}\tau_x(F)=0$ pour tout $(a_0,\dots,a_d,x)\in\mathrm{Ob}\,\A^{d+2}$ : l'ajout du terme supplémentaire $\tau_x$ dans la définition d'un foncteur polynomial ne change rien, car les foncteurs $\tau_x$ et $\delta_a$ commutent. Ce n'est en revanche pas le cas dans la définition précédente.

Supposons ainsi que $\A$ est une catégorie $k$-triviale non nulle. La remarque~\ref{rq-annuldb} montre qu'il existe des foncteurs finis $A$ de $\rep(\A)$ de longueur $d$ et des représentations irréductibles $M_A$ de $\mathrm{Aut}(A)$ tels que $\db_{a_0}\db_{a_1}(Q_{A,M_A})=0$ pour tout $(a_0,a_1)\in\mathrm{Ob}\,\A^2$ pour $d$ arbitrairement grand, alors que $\Q_{A,M_A}$ n'appartient à $\F_1(\A;k)$ que si $d\le 1$, par le corollaire~\ref{cor-QAM-Fd} ci-après. Ainsi, l'annulation d'itérées en nombre fixé de foncteurs de la forme $\db_x$ ne fournit pas une notion de \guillemotleft~taille~\guillemotright\ appropriée pour les foncteurs de $\F(\A;k)$, contrairement à la définition~\ref{df-Fd}, comme l'illustre la proposition~\ref{pr-Fdlin} ci-dessous.
\end{rema} 

\begin{prop}\label{pr-Fd-tau} Soient $d\in\mathbb{N}$, $F$ un foncteur de $\F_d(\A;k)$ et $y$ un objet de $\A$. Alors $\tau_y(F)$ appartient à $\F_d(\A;k)$ et $\db_y(F)$ appartient à $\F_{d-1}(\A;k)$.
\end{prop}

\begin{proof}
L'assertion relative à $\tau_y(F)$ résulte de l'isomorphisme canonique $\tau_x\circ\tau_y\simeq\tau_{x\oplus y}$. Celle relative à $\db_y(F)$ provient de l'épimorphisme $\db_y\circ\tau_x\twoheadrightarrow\tau_x\circ\db_y$ donné par la proposition~\ref{pr-comtbdb}.
\end{proof}

\begin{prop}\label{pr-Fdlin} Soient $A$ un foncteur de $\rep(\A)$ et $d\ge -1$ un entier. Le foncteur $k[A]$ appartient à $\F_d(\A;k)$ si et seulement si $\lgr(A)\le d$.
\end{prop}

\begin{proof}
Comme $\tau_x(k[A])$ est somme directe d'un nombre non nul de copies de $k[A]$ et que les $\db_a$ sont cocontinus, $k[A]$ appartient à $\F_d(\A;k)$ si et seulement si $\db_{a_0}\dots\db_{a_d}(k[A])$ est nul pour tout $(a_0,\dots,a_d)\in\mathrm{Ob}\,\A^{d+1}$. La conclusion s'obtient alors par récurrence sur $d$ à partir du corollaire~\ref{cor-delbarlin}. En effet, tout sous-foncteur de type fini non nul de $A$ est de la forme $A_\xi$ pour un $a\in\mathrm{Ob}\,\A$ et un $\xi\in A(a)\setminus\{0\}$, et si $A$ n'est pas de longueur finie, pour tout $n\in\mathbb{N}$, il possède un sous-foncteur de type fini non nul $B_n$ tel que $A/B_n$ n'est pas de longueur finie au plus $n$.
\end{proof}

\begin{coro}\label{cor-QAM-Fd} Soient $A$ un foncteur de type fini de $\rep(\A)$, $M$ un $k[\mathrm{Aut}(A)]$-module non nul et $d\ge -1$ un entier. Le foncteur $\Q_{A,M}$ appartient à $\F_d(\A;k)$ si et seulement si $\lgr(A)\le d$.
\end{coro}

\begin{proof}
Le foncteur $\Q_{A,M}$ est quotient d'une somme directe (éventuellement infinie) de copies de $k[A]$. Comme $\F_d(\A;k)$ est stable par somme directe et par quotients (car les foncteurs $\tau_x$ et $\db_a$ sont cocontinus), la proposition~\ref{pr-Fdlin} montre que $\Q_{A,M}$ appartient à $\F_d(\A;k)$ si $\lgr(A)\le d$.

La proposition~\ref{pr-tauQAM} montre par ailleurs que, si $\Q_{A,M}$ appartient à $\F_d(\A;k)$ (avec $A$ de type fini et $M$ non nul), alors $k[A]$ est aussi dans $\F_d(\A;k)$, ce qui implique $\lgr(A)\le d$ par la proposition~\ref{pr-Fdlin}.
\end{proof}

\begin{coro}\label{cor-tttfFd}
Les assertions suivantes sont équivalentes :
\begin{enumerate}
    \item pour tout foncteur de type fini $F$ de $\F(\A;k)$, il existe $d\in\mathbb{N}$ tel que $F\in\F_d(\A;k)$ ;
    \item la catégorie abélienne $\rep(\A)$ est localement finie.
\end{enumerate}
\end{coro}

\begin{proof}
La première condition entraîne la seconde grâce à la proposition~\ref{pr-Fdlin}, appliquée aux foncteurs additifs représentables $\A(a,-)$. L'implication réciproque se déduit de la même proposition, puisque $\F_d(\A;k)$ est stable par somme directe arbitraire et par quotient.
\end{proof}

\begin{hyp}
\textbf{\emph{Dans toute la suite de cet article, on suppose, sauf mention expresse du contraire\,\footnote{Cette exception concerne exclusivement les remarques~\ref{rq-pol1} et~\ref{rq-pol12}.}, que la catégorie additive $\A$ est $k$-triviale.}}
\end{hyp}

\begin{prop}\label{pr-Fdbiloc} La sous-catégorie $\F_d(\A;k)$ de $\F(\A;k)$ est bilocalisante pour tout entier $d\ge -1$.
\end{prop}

\begin{proof}
Les foncteurs de décalage $\tau_x$ sont toujours bicontinus. C'est également le cas des foncteurs de différence parabolique $\db_a$ lorsque $\A$ est $k$-triviale (proposition~\ref{pr-delbarexact}), d'où le résultat.
\end{proof}

Dans ce qui suit, si $f$ est un morphisme d'une catégorie abélienne, on note $\rg(f):=\lgr(\mathrm{Im}\,f)$.

\begin{prop}\label{pr-Fd-tftcf} Soient $d\ge -1$ un entier et $F$ un foncteur de type fini et de type cofini de $\F(\A;k)$. Les assertions suivantes sont équivalentes :
\begin{enumerate}
    \item[(1)] $F$ appartient à $\F_d(\A;k)$ ;
    \item[(2)] il existe des familles finies $(A_i)$ et $(B_j)$ de foncteurs finis de $\rep(\A)$ et des familles $(\lambda_{i,j}(f))_{f\in\mathrm{Hom}(A_i,B_j)}$ d'éléments de $k$ telles que
\begin{equation}\label{eq-imFd}
    F\simeq\mathrm{Im}\;\bigoplus_i k[A_i]\xrightarrow{\big(\Theta_{A_i,B_j}(\sum_{f\in\mathrm{Hom}(A_i,B_j)}\lambda_{i,j}(f)[f])\big)_{i,j}}\bigoplus_j k[B_j]
\end{equation}
et que $\rg(f)\le d$ pour tout $f\in\mathrm{Hom}(A_i,B_j)$ tel que $\lambda_{i,j}(f)\ne 0$ ;
    \item[(3)] $F$ est un sous-quotient d'une somme directe finie de foncteurs de la forme $k[A]$ pour $A$ dans $\rep(\A)$ tel que $\lgr(A)\le d$.
\end{enumerate}
\end{prop}

\begin{proof} Si $F$ est de type fini et de type cofini, comme $\A$ est $k$-triviale, on peut écrire $F$ comme image d'un morphisme entre sommes directes finies de linéarisations de foncteurs finis de $\rep(\A)$, par le corollaire~\ref{cor-tftcf}, comme dans son point \emph{(b)}. Soit $f\in\mathrm{Hom}(A_i,B_j)$ tel que $\lambda_{i,j}(f)\ne 0$ : par la proposition~\ref{pr-sqtau}, il existe un objet $x$ de $\A$ tel que $k[\mathrm{Im}\,(f)]$ soit sous-quotient de $\tau_x(F)$. Si $F$ appartient à $\F_d(\A;k)$, on en déduit que c'est aussi le cas de $k[\mathrm{Im}\,(f)]$, puisque $\F_d(\A;k)$ est stable par sous-quotient (cf. proposition~\ref{pr-Fdbiloc}), et manifestement par $\tau_x$. La proposition~\ref{pr-Fdlin} montre alors que $\rg(f)\le d$, d'où l'implication $(1)\Rightarrow (2)$.

L'implication $(2)\Rightarrow (3)$ est claire. Quant à $(3)\Rightarrow (1)$, elle découle des propositions~\ref{pr-Fdlin} et~\ref{pr-Fdbiloc}.
\end{proof}

La proposition précédente entraîne aussitôt la variation suivante autour de l'une des implications du corollaire~\ref{cor-tttfFd} :

\begin{coro}\label{cor-tous_tftcf_ds_Fd} Supposons que tout foncteur de $\rep(\A)$ qui est simultanément de type fini et de type cofini est fini. Alors tout foncteur de type fini et de type cofini (en particulier, tout foncteur fini) de $\F(\A;k)$ appartient à $\F_d(\A;k)$ pour un $d\in\mathbb{N}$.
\end{coro}

\begin{rema}\label{rq-simpathol} En s'appuyant sur des résultats classiques d'Auslander \cite{Aus74}, on peut voir (de façon élémentaire, mais avec un peu de travail) que la réciproque du corollaire précédent est vraie ; plus précisément, s'il existe dans $\rep(\A)$ un foncteur de type fini et de type cofini qui n'est pas fini, alors il existe un foncteur \emph{simple} de $\F(\A;k)$ qui n'appartient à aucune sous-catégorie $\F_d(\A;k)$. De plus, il existe des petites catégories additives $k$-triviales $\A$ vérifiant ces conditions, qui équivalent, lorsque les idempotents de $\A$ se scindent, à la suivante : si $a$ et $b$ sont des objets de $\A$, il n'existe qu'un nombre fini de classes d'équivalence d'objets \emph{indécomposables} $t$ de $\A$ tels qu'existe une composée $a\to t\to b$ non nulle.
\end{rema}

\begin{coro}\label{cor-Fd-eq}
Pour tout entier $d\ge -1$, $\F_d(\A;k)$ est la plus petite sous-catégorie bilocalisante de $\F(\A;k)$ contenant $k[A]$ pour $A$ dans $\rep(\A)$ tel que $\lgr(A)\le d$.
\end{coro}

\begin{proof}
Les propositions~\ref{pr-Fdlin} et~\ref{pr-Fdbiloc} montrent que $\F_d(\A;k)$ est une sous-catégorie bilocalisante de $\F(\A;k)$ contenant $k[A]$ pour $A$ dans $\rep(\A)$ tel que $\lgr(A)\le d$. Réciproquement, une sous-catégorie bilocalisante $\E$ de $\F(\A;k)$ contenant ces foncteurs contient tous les objets de type fini et de type cofini de $\F_d(\A;k)$, grâce à l'implication $(1)\Rightarrow (3)$ de la proposition~\ref{pr-Fd-tftcf}. Comme tout foncteur de type cofini est colimite de ses sous-foncteurs qui sont à la fois de type fini et de type cofini, $\E$ contient tous les foncteurs de type cofini de $\F_d(\A;k)$. Cela donne la conclusion recherchée car tout foncteur de $\F_d(\A;k)$ est  sous-objet du produit de ses quotients de type cofini.
\end{proof}

Le résultat ci-dessous n'est pas a priori évident puisque l'endofoncteur $\db_x$ de $\F(\A;k)$ n'est pas auto-dual.

\begin{coro}\label{cor-FdAutoDual} Soient $d\ge -1$ un entier et $F$ un foncteur de $\F(\A;k)$. Alors $F$ appartient à $\F_d(\A;k)$ si et seulement si $F^\vee$ appartient à $\F_d(\A^\op;k)$.
\end{coro}

\begin{proof}
Si $F$ est de type fini et de type cofini, l'équivalence découle des propositions~\ref{pr-Fd-tftcf} et~\ref{pr-dual_lin}. Supposons maintenant $F$ de type fini : comme $F$ est à valeurs de dimensions finies (grâce à la condition (FH)), $F$ est limite filtrante \emph{ponctuellement triviale} (cf. le début de la démonstration de la proposition~\ref{pr-ExtLin}) de ses quotients de type fini et de type cofini $G_i$, ce qui implique que $F^\vee$ est colimite de ses sous-objets $G_i^\vee$, qui sont de type fini et de type cofini. Ainsi, l'équivalence vaut pour $F$ de type fini par le cas précédent. Le cas général s'en déduit en écrivant un foncteur comme colimite de ses sous-foncteurs de type fini et en utilisant la continuité du foncteur de dualité $(-)^\vee : \F(\A;k)^\op\to\F(\A^\op;k)$.
\end{proof}

La propriété ci-dessous de la filtration $(\F_d(\A;k))_{d\in\mathbb{N}}$ est analogue au comportement bien connu de la filtration polynomiale par rapport au produit tensoriel.

\begin{coro}\label{cor-Fd-ptens}
Soient $n, m\in\mathbb{N}$, $F$ et $G$ des foncteurs de $\F_n(\A;k)$ et $\F_m(\A;k)$ respectivement. Alors le foncteur $F\otimes G$ appartient à $\F_{n+m}(\A;k)$.
\end{coro}

\begin{proof}
Comme le produit tensoriel commute aux colimites par rapport à chaque variable et que les sous-catégories $\F_d(\A;k)$ sont localisantes, il suffit de montrer le résultat lorsque $F$ et $G$ sont de type fini (le cas général s'obtenant en écrivant un foncteur comme la colimite de ses sous-foncteurs de type fini). En particulier, $F$ et $G$ sont à valeurs de dimension finie sur $k$, grâce à l'hypothèse (FH). Cela entraîne que l'endofoncteur $F\otimes -$ de $\F(\A;k)$ est bicontinu. Il s'ensuit que la sous-catégorie pleine des foncteurs $T$ de $\F(\A;k)$ tels que $F\otimes T\in\F_{n+m}(\A;k)$ est, comme $\F_{n+m}(\A;k)$, bilocalisante. Ainsi, il suffit de montrer que $F\otimes k[B]$ appartient à $\F_{n+m}(\A;k)$ pour $B$ dans $\rep(\A)$ fini tel que $\lgr(B)\le m$, par le corollaire~\ref{cor-Fd-eq}. En raisonnant de la même façon sur l'autre argument du produit tensoriel, on voit qu'il suffit de montrer que $k[A]\otimes k[B]$ appartient à $\F_{n+m}(\A;k)$ pour $A$ et $B$ finis dans $\rep(\A)$ tels que $\lgr(A)\le n$ et $\lgr(B)\leq m$. Cela résulte de la proposition~\ref{pr-Fdlin} et de l'isomorphisme canonique $k[A]\otimes k[B]\simeq k[A\oplus B]$.
\end{proof}

\begin{rema}
Nous ne connaissons pas de démonstration plus directe de ce résultat, car les foncteurs $\db_a$ ne semblent pas posséder de propriété de compatibilité simple au produit tensoriel.
\end{rema}

\section{Théorème de structure}\label{sec-ThStruc}

\subsection{Les foncteurs $S_d$, $C_d$ et $T_d$}

\begin{nota}
Pour $d\in\mathbb{N}$, on choisit un système complet $\LL_d(\A)$ de représentants des classes d'isomorphisme de foncteurs finis de $\rep(\A)$ de longueur égale à $d$.
\end{nota}

Ce choix effectué, nous introduisons les protagonistes de notre théorème de structure~\ref{th-princ} ci-après : 

\begin{nota}
Pour $d\in\mathbb{N}$, on note :
\begin{enumerate}
    \item $\pi_d : \F_d(\A;k)\to\F_d(\A;k)/\F_{d-1}(\A;k)$ le foncteur canonique \cite[chap.~III, §\,1]{Gabriel} ;
    \item $S_d, C_d, T_d : \prod_{A\in\LL_d(\A)}k[\mathrm{Aut}(A)]\Md\to\F_d(\A;k)$ les foncteurs donnés par
$$S_d : (M_A)_{A\in\LL_d(\A)}\mapsto\prod_{A\in\LL_d(\A)}\Q_{A,M_A}\;,$$
$$C_d : (M_A)_{A\in\LL_d(\A)}\mapsto\bigoplus_{A\in\LL_d(\A)}\Q^{A,M_A}\;;$$
et
$$T_d : (M_A)_{A\in\LL_d(\A)}\mapsto\bigoplus_{A\in\LL_d(\A)}\Q(A,M_A)\;;$$
    \item $\Phi_d:=\pi_d\circ S_d :  \prod_{A\in\LL_d(\A)}k[\mathrm{Aut}(A)]\Md\to\F_d(\A;k)/\F_{d-1}(\A;k)$.
\end{enumerate}
\end{nota}

(Le corollaire~\ref{cor-QAM-Fd} et la proposition~\ref{pr-Fdbiloc} garantissent que $S_d$, $C_d$ et $T_d$ prennent bien leurs valeurs dans $\F_d(\A;k)$.)

\begin{nota}\label{not-FE} Soit $X$ un objet de $\rep(\A)$. On note $(\mathrm{FE}(X))$ la propriété suivante :
$$(\mathrm{FE}(X))\qquad\qquad\text{le groupe abélien }\bigoplus_{S\in\LL_1(\A)}\mathrm{Ext}^1(X,S)\text{ est fini.}$$

On note $(\mathrm{FE}(\A))$ la propriété suivante :
$$(\mathrm{FE}(\A))\qquad\qquad(\mathrm{FE}(X))\text{ est vérifiée pour tout }X\in\LL_1(\A).$$
\end{nota}

\begin{prop}\label{pr-ExtFE} Soit $X$ un objet de $\rep(\A)$. 
\begin{enumerate}
    \item Si la propriété $(\mathrm{FE}(X))$ est vérifiée, alors il n'existe qu'un nombre fini de classes d'équivalence d'objets $Y$ de $\rep(\A)$ tels qu'il existe une suite exacte courte non scindée $0\to S\to Y\to X\to 0$ avec $S$ simple ;
    \item si $X$ est simple et qu'il n'existe qu'un nombre fini de classes d'équivalence d'objets $Y$ de $\rep(\A)$ tels qu'il existe une suite exacte courte non scindée $0\to S\to Y\to X\to 0$ avec $S$ simple, alors $(\mathrm{FE}(X))$ est vérifiée ;
    \item la classe des objets $X$ de $\rep(\A)$ tels que $(\mathrm{FE}(X))$ soit vérifiée est stable par extension ;
    \item si $(\mathrm{FE}(\A))$ est vérifiée, alors $(\mathrm{FE}(X))$ est satisfaite pour tout objet fini $X$ de $\rep(\A)$ ;
    \item si $X$ est de présentation finie, alors $(\mathrm{FE}(X))$ est satisfaite.
\end{enumerate}
\end{prop}

\begin{proof}
La première assertion est immédiate.

Si $X$ et $S$ sont simples et que $0\to S\to Y\to X\to 0$ est une suite exacte courte non scindée, alors $S$ est le socle de $Y$ et $X$ son cosocle. Il s'ensuit que tout automorphisme de $Y$ préserve $S$ et $X$. Par conséquent, s'il n'existe qu'un nombre fini de classes d'équivalence d'objets $Y$ de $\rep(\A)$ tels qu'il existe une suite exacte courte non scindée $0\to S\to Y\to X\to 0$ avec $S$ simple, alors l'ensemble
$$\bigsqcup_{S\in\LL_1(\A)}{}_{\mathrm{Aut}(S)\backslash}\mathrm{Ext}^1(X,S)_{/\mathrm{Aut}(X)}$$
est fini. Comme le groupe $\mathrm{Aut}(S)$ est fini pour $S$ simple (en raison de la condition (FH)), on en déduit la deuxième assertion.

La troisième assertion résulte de la suite exacte longue en Ext associée à une suite exacte courte ; la quatrième assertion s'en déduit par récurrence sur la longueur de $X$.

Enfin, si $X$ est de présentation finie, il existe des objets $a$ et $b$ de $\A$ et une suite exacte $\A(a,-)\to\A(b,-)\to X\to 0$ dans $\rep(\A)$, ainsi $\mathrm{Ext}^1(X,S)$ est un sous-quotient de $S(a)$. Comme $\A$ vérifie (FH), $S(a)$ est un groupe fini pour $S$ simple, et il n'existe qu'un nombre fini de classes d'isomorphisme d'objets simples $S$ tels que $S(a)\ne 0$, ce qui achève la démonstration.
\end{proof}

\begin{exem}\label{ex-FE}
La dernière assertion de la proposition précédente montre que la condition $(\mathrm{FE}(\A))$ est vérifiée lorsque $\rep(\A)$ est localement noethérienne. Les travaux d'Auslander \cite{Aus74} donnent des exemples de petites catégories \emph{abéliennes} $k$-triviales $\A$ telles que tous les simples de $\rep(\A)$ soient de présentation finie, mais que cette catégorie ne soit pas localement noethérienne.
\end{exem}

\begin{rema}\label{rq-pasFE}
Il existe des catégories additives $k$-triviales $\A$ telles que $(\mathrm{FE}(\A))$ ne soit pas vérifiée. Considérons par exemple un corps fini $\FF$ de caractéristique différente de celle de $k$ et $\E$ la catégorie des foncteurs depuis la petite catégorie associée à l'ensemble $E:=\mathbb{N}$ ordonné de telle sorte que $0$ soit le plus petit élément et que tous les éléments de $\mathbb{N}^*$ soient deux à deux incomparables. Prenons pour $\A$ la catégorie opposée de la sous-catégorie pleine des objets projectifs de type fini de $\E$ : alors $\A$ est $k$-triviale, et $\rep(\A)\simeq\E$. Si l'on définit $S_i$ comme l'objet de $\E$ valant $\FF$ évalué sur $i\in E$ et nul ailleurs, alors $(S_i)_{i\in E}$ est un ensemble complet de représentants des classes d'isomorphisme d'objets simples de $\E$. On vérifie aussitôt que $\mathrm{Ext}^1(S_0,S_n)\simeq\FF$ pour tout $n\in\mathbb{N}^*$, ainsi la condition $(\mathrm{FE}(S_0))$ n'est-elle pas satisfaite. 
\end{rema}

Nous omettrons la démonstration de l'énoncé facile suivant.

\begin{lemm}\label{lm-bicont} Soient $E$ un ensemble et, pour chaque $i\in E$, $R_i$ un anneau et $\Xi_i : R_i\Md\to\F(\A;k)$ un foncteur bicontinu. Les assertions suivantes sont équivalentes :
\begin{enumerate}
    \item le foncteur $\prod_{i\in E}R_i\Md\to\F(\A;k)\quad (X_i)\mapsto\bigoplus_{i\in E}\Xi_i(X_i)$ est bicontinu ;
    \item le foncteur $\prod_{i\in E}R_i\Md\to\F(\A;k)\quad (X_i)\mapsto\prod_{i\in E}\Xi_i(X_i)$ est bicontinu ;
    \item le monomorphisme canonique $\bigoplus_{i\in E}\Xi_i(X_i)\to\prod_{i\in E}\Xi_i(X_i)$ est un isomorphisme pour toute famille $(X_i)$ de $R_i$-modules ;
    \item pour tout objet $a$ de $\A$, l'ensemble des $i\in E$ tels que $\Xi_i(R_i)(a)\ne 0$ est fini.
\end{enumerate}
\end{lemm}

\begin{prop}\label{pr-bicont} Soit $d\in\mathbb{N}$.
\begin{enumerate}
    \item Le foncteur $T_d$ est bicontinu.
    \item Le foncteur $S_d$ est exact et continu. Il est bicontinu si et seulement si la condition $(\mathrm{FE}(\A))$ est vérifiée ou que $d\le 1$.
    \item Le foncteur $C_d$ est exact et cocontinu. Il est bicontinu si et seulement si la condition $(\mathrm{FE}(\A^\op))$ est vérifiée ou que $d\le 1$.
    \item On a $\Phi_d\simeq\pi_d\circ C_d\simeq\pi_d\circ T_d$.
    \item Les foncteurs $\pi_d$ et $\Phi_d$ sont bicontinus.
\end{enumerate}
\end{prop}

\begin{proof} Soit $a$ un objet de $\A$. L'espace vectoriel $\Q_A(a)$ (resp. $\Q^A(a)$) est non nul si et seulement s'il existe un épimorphisme $\A(a,-)\twoheadrightarrow A$ (resp. un monomorphisme $A\hookrightarrow\A(-,a)^\sharp$), par les propositions~\ref{pr-descrQ} et~\ref{pr-dualQ}.

La proposition~\ref{pr-Qexact} et le lemme~\ref{lm-bicont} montrent par conséquent que $S_d$ est exact et continu, et qu'il est bicontinu si et seulement si pour tout objet $a$ de $\A$, l'ensemble des $A\in\LL_d(\A)$ qui sont quotients de $\A(a,-)$ est fini. Pour $d=1$, c'est toujours vrai car $\A(a,-)$, qui est de type fini, a un cosocle fini grâce à l'hypothèse (FH). Pour $d=2$, la condition nécessaire et suffisante de la deuxième assertion découle de la finitude de ce cosocle, de la proposition~\ref{pr-ExtFE} et de l'observation que si $0\to Y\to X\to 0$ est une suite exacte non scindée de $\rep(\A)$ avec $S$ simple qu'il existe un épimorphisme $\A(a,-)\twoheadrightarrow X$, alors il existe un épimorphisme $\A(a,-)\twoheadrightarrow Y$. Le cas $d\ge 2$ s'obtient en général par récurrence en utilisant encore la proposition~\ref{pr-ExtFE}, les observations précédentes et le fait que s'il existe une infinité d'éléments de $\LL_d(\A)$ (avec $d\ge 1$) sur lesquels $\A(a,-)$ se surjecte, alors il existe une infinité d'éléments de $\LL_{d+1}(\A)$ sur lesquels $\A(a^{\oplus 2},-)$ se surjecte. Cela termine la démonstration de la deuxième assertion.

La troisième assertion est duale de la deuxième.

La première se déduit elle aussi du lemme~\ref{lm-bicont} et de l'observation du début de cette démonstration, en utilisant le corollaire~\ref{cor-QAMequiv} : la non-nullité de $\Q(A)(a)$ entraîne celle de $\Q_A(a)$ et de $\Q^A(a)$ (puisque $\Q(A)$ est un sous-objet de $\Q_A$ et un quotient de $\Q^A$), donc que $A$ est l'image d'un morphisme $\A(a,-)\to\A(-,a)^\sharp$. Comme $\mathrm{Hom}(\A(a,-),A(-,a)^\sharp)\simeq\A(a,a)^\sharp$ est fini, cela montre la première assertion.

La quatrième assertion découle de ce que, pour $A\in\LL_d(\A)$ et $M\in\mathrm{Ob}\,k[\mathrm{Aut}(A)]\Md$, les morphismes canoniques $\Q^{A,M}\twoheadrightarrow\Q(A,M)\hookrightarrow\Q_{A,M}$ deviennent des isomorphismes dans $\F_d(\A;k)/\F_{d-1}(\A;k)$.

La cinquième assertion résulte des première et quatrième.
\end{proof}

\subsection{Les catégories $\F_d(\A;k)/\F_{d-1}(\A;k)$}

\begin{theo}\label{th-princ}
Pour tout $d\in\mathbb{N}$, le foncteur
$$\Phi_d :  \prod_{A\in\LL_d(\A)}k[\mathrm{Aut}(A)]\Md\xrightarrow{\simeq}\F_d(\A;k)/\F_{d-1}(\A;k)$$
est une équivalence de catégories. De plus, la composée de l'équivalence inverse avec le foncteur $S_d$ (resp. $C_d$) $\prod_{A\in\LL_d(\A)}k[\mathrm{Aut}(A)]\Md\to\F_d(\A;k)$ est le fonteur section (resp. co-section) $\F_d(\A;k)/\F_{d-1}(\A;k)\to\F_d(\A;k)$, c'est-à-dire l'adjoint à droite (resp. à gauche) du foncteur canonique.
\end{theo}

Ce théorème, qui constitue le résultat principal du présent article, sera démontré en plusieurs étapes dans la suite de cette section. Pour ce faire, nous introduisons quelques notations.

Nous noterons $HR(d)$ la conclusion du théorème. Celle-ci (notamment la pleine fidélité de $\Phi_d$) sera partiellement démontrée par récurrence sur $d$, via les deux propositions suivantes :

\begin{prop}\label{pr-Ext_HR}
Supposons que $HR(i)$ est vérifié pour tout $i<d$. Alors $\mathrm{Ext}^*_{\F(\A;k)}(T,F)=0$ pour tout foncteur $T$ de $\F_{d-1}(\A;k)$ et tout foncteur $F$ appartenant à l'image essentielle de $S_d$.
\end{prop}

\begin{proof}
On montre par récurrence sur $i\in\{0,\dots,d-1\}$ que $\mathrm{Ext}^*_{\F(\A;k)}(T,\Q_{A,M})=0$ pour tous $T$ dans $\F_i(\A;k)$, tout $A\in\LL_d(\A)$ et toute représentation $M$ de $\mathrm{Aut}(A)$. L'hypothèse $HR(i)$ entraîne qu'il existe une suite exacte
$$0\to X\to\bigoplus_{B\in\LL_i(\A)}\Q^{B,N_B}\to T\to Y\to 0$$
où les $N_B$ sont des représentations des $\mathrm{Aut}(B)$ et $X$ et $Y$ sont des foncteurs de $\F_{i-1}(\A;k)$. Les groupes d'extensions $\mathrm{Ext}^*_{\F(\A;k)}(X,\Q_{A,M})$ et $\mathrm{Ext}^*_{\F(\A;k)}(Y,\Q_{A,M})$ s'annulent grâce à l'hypothèse de récurrence sur $i$, tandis que les $\mathrm{Ext}^*_{\F(\A;k)}(\Q^{B,N_B},\Q_{A,M})$ s'annulent pour $B\in\LL_i(\A)$ grâce à la proposition~\ref{pr-ExtQAMQ}. Il s'ensuit que $\mathrm{Ext}^*_{\F(\A;k)}(T,\Q_{A,M})$ est également nul, d'où la conclusion.
\end{proof}

La notion d'objet $\C$-fermé dans une catégorie abélienne $\E$, où $\C$ est une sous-catégorie épaisse de $\E$, est définie par Gabriel \cite[chap.~III, §\,2, p.~371, avant le lemme~2]{Gabriel}. Le corollaire ci-dessous résulte de la proposition~\ref{pr-Ext_HR} et de \cite[chap.~III, §\,2, lemme~1]{Gabriel}.

\begin{coro}\label{cor-Fd_ferme} Si $HR(i)$ est vérifié pour tout $i<d$, tout objet de $\F_d(\A;k)$ appartenant à l'image essentielle de $S_d$ est $\F_{d-1}(\A;k)$-fermé.
\end{coro}

 Avant d'énoncer notre prochain résultat, rappelons \cite[déf.~3.6]{Psa-hom} qu'on dit qu'un foncteur $\Phi : \E'\to\E$ entre deux catégories de Grothendieck est un \emph{plongement homologique} si $\Phi$ est exact et que le morphisme de groupes abéliens gradués $\mathrm{Ext}^*_{\E'}(X,Y)\to\mathrm{Ext}^*_\E(\Phi(X),\Phi(Y))$ qu'il induit est un isomorphisme pour tous objets $X$ et $Y$ de $\E'$.

\begin{prop}\label{pr-Spf} Supposons que $HR(i)$ est vérifié pour tout $i<d$.
\begin{enumerate}
    \item\label{it-pf} Le foncteur $\Phi_d$ est pleinement fidèle.
    \item\label{it-ph} Le foncteur composé
    $$\prod_{A\in\LL_d(\A)}k[\mathrm{Aut}(A)]\Md\xrightarrow{S_d}\F_d(\A;k)\hookrightarrow\F(\A;k)$$
    est un plongement homologique.
\end{enumerate}
\end{prop}

\begin{proof} Soit $F$ un foncteur de $\F_d(\A;k)$ appartenant à l'image essentielle de $S_d$.
Le corollaire~\ref{cor-Fd_ferme} et \cite[chap.~III, §\,2, lemme~1]{Gabriel} impliquent que, pour tout foncteur $Y$ de $\F_d(\A;k)$, le morphisme
$$\mathrm{Hom}_{\F_d(\A;k)}(Y,X)\to\mathrm{Hom}_{\F_d(\A;k)/\F_{d-1}(\A;k)}(\pi_d(Y),\pi_d(X))$$
qu'induit $\pi_d$ est un isomorphisme. Il suffit donc d'établir l'assertion~\ref{it-ph}.

Soient $(M_A)_{A\in\LL_d(\A)}$ et $(N_A)_{A\in\LL_d(\A)}$ des familles de réprésentations $k$-linéaires des groupes $\mathrm{Aut}(A)$. On dispose d'isomorphismes canoniques
$$\mathrm{Ext}^*_{\F(\A;k)}\Big(\bigoplus_{A\in\LL_d(\A)}\Q_{A,M_A},\prod_{A\in\LL_d(\A)}\Q_{A,N_A}\Big)\simeq$$
$$\prod_{(A,B)\in\LL_d(\A)^2}\mathrm{Ext}^*_{\F(\A;k)}(\Q_{A,M_A},\Q_{B,N_B})\simeq\prod_{A\in\LL_d(\A)}\mathrm{Ext}^*_{k[\mathrm{Aut}(A)]}(M_A,N_A)$$
grâce au théorème~\ref{thm-Ext-QAM}.

Par ailleurs, les assertions 1. et 4. de la proposition~\ref{pr-bicont}, le lemme~\ref{lm-bicont} et la proposition~\ref{pr-Fdbiloc} montrent que le conoyau du monomorphisme canonique
$$\bigoplus_{A\in\LL_d(\A)}\Q_{A,M_A}\hookrightarrow\prod_{A\in\LL_d(\A)}\Q_{A,M_A}$$
appartient à $\F_{d-1}(\A;k)$. En utilisant la proposition~\ref{pr-Ext_HR}, on en déduit
$$\mathrm{Ext}^*_{\F(\A;k)}\Big(\prod_{A\in\LL_d(\A)}\Q_{A,M_A},\prod_{A\in\LL_d(\A)}\Q_{A,N_A}\Big)\simeq$$
$$\mathrm{Ext}^*_{\F(\A;k)}\Big(\bigoplus_{A\in\LL_d(\A)}\Q_{A,M_A},\prod_{A\in\LL_d(\A)}\Q_{A,N_A}\Big)\simeq$$
$$\simeq\prod_{A\in\LL_d(\A)}\mathrm{Ext}^*_{k[\mathrm{Aut}(A)]}(M_A,N_A)\;,$$
l'isomorphisme étant induit par la composée de $S_d$ et de l'inclusion $\F_d(\A;k)\hookrightarrow\F(\A;k)$, ce qui achève la démonstration.
\end{proof}

\subsection{Démonstration du théorème~\ref{th-princ}}

Nous introduisons au préalable quelques notations qui nous serviront à établir, en plusieurs étapes, l'essentielle surjectivité du foncteur $\Phi_d$.

Nous appellerons \emph{décomposition de type $D_d$} d'un foncteur $F$ de $\F(\A;k)$ tout isomorphisme de la forme~\eqref{eq-imFd} (page~\pageref{eq-imFd})
avec $\rg(f)=d$ pour tout $f\in\mathrm{Hom}(A_i,B_j)$ tel que $\lambda_{i,j}(f)\ne 0$. Nous dirons qu'une telle décomposition est de type $I_d$ (resp. $M_d$, $E_d$) si de plus tout $f\in\mathrm{Hom}(A_i,B_j)$ tel que $\lambda_{i,j}(f)\ne 0$ est un isomorphisme (resp. un monomorphisme, un épimorphisme). 

Tout foncteur de type $D_d$ est de type fini et de type cofini et appartient à $\F_d(\A;k)$ (cf. proposition~\ref{pr-Fd-tftcf}) ; réciproquement :

\begin{lemm}\label{lm-typeD} Soit $F$ un foncteur de type fini et de type cofini de $\F_d(\A;k)$. Alors il existe un foncteur $G$ admettant une décomposition de type $D_d$ tel que $\pi_d(F)\simeq\pi_d(G)$.
\end{lemm}

\begin{proof}
Cela résulte de la proposition~\ref{pr-Fd-tftcf} et du fait que si $f : A\to B$ est un morphisme de $\rep(\A)$ tel que $\rg(f)<d$, alors $k[f] : k[A]\to k[B]$ est nul dans  $\F(\A;k)/\F_{d-1}(\A;k)$.
\end{proof}

\begin{lemm}\label{lm-Id}
Si $F$ est un foncteur admettant une décomposition de type $I_d$, alors $\pi_d(F)$ appartient à l'image essentielle du foncteur $\Phi_d$.
\end{lemm}

\begin{proof}
Pour $A$ dans $\rep(\A)$ de longueur $d$, l'épimorphisme canonique $k[A]\twoheadrightarrow\Q_A$ induit un isomorphisme dans $\F_d(\A;k)/\F_{d-1}(\A;k)$. Il s'ensuit que $F$ est isomorphe dans $\F_d(\A;k)/\F_{d-1}(\A;k)$ à l'image d'un morphisme de la forme
$$\bigoplus_{A\in\LL_d(\A)}\Q_A^{\oplus n(A)}\xrightarrow{\bigoplus_{A\in\LL_d(\A)}(\sum_{f\in\mathrm{Aut}(A)}\lambda_{i,j}(f).f_*)_{1\le i\le n(A),1\le j\le m(A)}}\bigoplus_{A\in\LL_d(\A)}\Q_A^{\oplus m(A)}$$
où les $n(A)$ et les $m(A)$ sont des entiers naturels, nuls sauf pour un nombre fini de $A\in\LL_d(\A)$, et les $\lambda_{i,j}(A)$ sont des éléments de $k$. Si l'on note $M_A$ l'image du morphisme
$$k[\mathrm{Aut}(A)]^{\oplus n(A)}\xrightarrow{(\sum_{f\in\mathrm{Aut}(A)}\lambda_{i,j}(f)[f])_{1\le i\le n(A),1\le j\le m(A)}}k[\mathrm{Aut}(A)]^{\oplus m(A)}$$
de $k[\mathrm{Aut}(A)]\Md$, la proposition~\ref{pr-Qexact} montre que $\pi_d(F)\simeq\Phi_d\big((M_A)_{A\in\LL_d(\A)}\big)$.
\end{proof}

\begin{nota} Soient $A$, $B$ des foncteurs de $\rep(\A)$, $A'$ et $B'$ des sous-foncteurs de $A$ et $B$ respectivement. On note
$$\Upsilon_{A',B'} : k[\mathrm{Hom}(A,B)]\to k[\mathrm{Hom}(A/A',B/B')]$$
l'application linéaire envoyant $[f]$ sur $[\bar{f}]$ si $f(A')=B'$ et sur $0$ sinon, pour $f\in\mathrm{Hom}(A,B)$, où $\bar{f} : A/A'\to B/B'$ désigne le morphisme induit par $f$.

Pour $x\in\mathrm{Ob}\,\A$, $\xi\in A(x)$ et $\zeta\in B(x)$, on note $\Upsilon_{\xi,\zeta}$ pour $\Upsilon_{A_\xi,B_\zeta}$.

Pour $A'\subset A''\subset A$, on note $r_{A',A''} : k[\mathrm{Hom}(A/A'',B)]\to k[\mathrm{Hom}(A/A',B)]$ le morphisme induit par la projection $A/A'\twoheadrightarrow A/A''$.
\end{nota}

Dans la suite, on note $\mathrm{Hom}_d(A,B)$ le sous-ensemble de $\mathrm{Hom}(A,B)$ constitué des $f : A\to B$ tels que $\rg(f)=d$.

\begin{lemm}\label{lm-decD-db}
Soient $A$ et $B$ des foncteurs de $\rep(\A)$, avec $A$ fini, et $A'$ un sous-foncteur de $A$. Alors les applications linéaires
$$\Upsilon_{A',0}\quad\text{et}\quad\underset{A''\supset A'\,,\,\lgr(A/A'')=d}{\sum_{A''\subset A}}r_{A',A''}\circ\Upsilon_{A'',0} : k[\mathrm{Hom}(A,B)]\to k[\mathrm{Hom}(A/A',B)]$$ coïncident sur $k[\mathrm{Hom}_d(A,B)]$.
\end{lemm}

\begin{proof}
Soient $f\in\mathrm{Hom}_d(A,B)$ et $N$ le noyau de $f$ ; on a donc $\lgr(A/N)=d$. On a $\Upsilon_{A',0}([f])=[f]$ si $N\supset A'$, $0$ sinon, tandis que
$$\underset{\lgr(A/A'')=d}{\sum_{A''\supset A'}}r_{A',A''}\circ\Upsilon_{A'',0}([f])=\underset{\lgr(A/A'')=d}{\sum_{N\supset A''\supset A'}}[\bar{f}]\;;$$
comme $N\supset A''\supset A'$ implique $A''=N$ si $\lgr(A/A'')=d=\lgr(A/N)$, cela démontre le lemme.
\end{proof}

L'énoncé suivant est formel et immédiat.

\begin{lemm}\label{lm-rg}
Soient $A$, $B$ des foncteurs finis de $\rep(\A)$, $x$ un objet de $\A$, $\xi\in A(x)$ et $\zeta\in B(x)$ et $f : A\to B$ un morphisme tel que $f_*(\xi)=\zeta$. Alors $f$ induit des morphismes $\tilde{f} : A_\xi\to B_\zeta$ et $\bar{f} : A/A_\xi\to B/B_\zeta$ s'insérant dans un diagramme commutatif aux lignes exactes
$$\xymatrix{0 \ar[r] & A_\xi\ar[r]\ar@{->>}[d]_-{\tilde{f}} & A\ar[d]^-f\ar[r] & A/A_\xi\ar[r]\ar[d]^-{\bar{f}} & 0 \\
0 \ar[r] & B_\zeta\ar[r] & B\ar[r] & B/B_\zeta\ar[r] & 0
}\;;$$
de plus, on a $\rg(\bar{f})\le\rg(f)$, avec égalité si et seulement si $\zeta=0$.
\end{lemm}

\begin{lemm}\label{lm-regularis} Soit $F$ un foncteur de $\F(\A;k)$ admettant une décomposition de type $D_d$ : 
$$F\simeq\mathrm{Im}\,\bigoplus_{i\in I}k[A_i]\xrightarrow{(\Theta_{A_i,B_j}(f_{i,j}))_{(i,j)\in I\times J}}\bigoplus_{j\in J}k[B_j]$$
où $I$ et $J$ sont des ensembles finis, les $A_i$ et $B_j$ sont des foncteurs finis de $\rep(\A)$ et  $f_{i,j}\in\mathrm{Hom}_d(A_i,B_j)$.

Alors il existe un objet $x$ de $\A$ tel que $F$ soit isomorphe dans $\F_d(\A;k)/\F_{d-1}(\A;k)$ à l'image de
$$\underset{\lgr(A_i/T)=d}{\bigoplus_{i\in I\,,\,T\subset A_i}}k[A_i/T]\xrightarrow{(\Theta_{A_i/T,B_j}(\Upsilon_{T,0}(f_{i,j})))_{i,T,j}}\bigoplus_{j\in J}k[B_j]$$
\end{lemm}

\begin{proof} On choisit $x\in\mathrm{Ob}\,\A$ tel que tout sous-foncteur d'un $A_i$ soit quotient de $\A(x,-)$ (un tel $x$ existe car les $A_i$ sont finis).

Par la proposition~\ref{pr-Fd-tau}, $F$ est isomorphe à $\tb_x(F)$ dans $\F_d(\A;k)/\F_{d-1}(\A;k)$. Par ailleurs, la proposition~\ref{pr-taubarlin} (et l'exactitude de $\tb_x$ --- cf. proposition~\ref{pr-taubarexact}) montrent que $\tb_x(F)$ est isomorphe à l'image de 
$$\underset{\xi\in A_i(x)}{\bigoplus_{i\in I}}k[A_i/(A_i)_\xi]\xrightarrow{(\Theta_{A_i/(A_i)_\xi,B_j/(B_j)_\zeta}(\Upsilon_{\xi,\zeta}(f_{i,j})))_{(i,\xi),(j,\zeta)}}\underset{\zeta\in B_j(x)}{\bigoplus_{j\in J}}k[B_j/(B_j)_\zeta].$$
Le lemme~\ref{lm-rg} et notre choix de $x$ montrent que $F$ est isomorphe dans $\F_d(\A;k)/\F_{d-1}(\A;k)$ à l'image de
$$\underset{T\subset A_i}{\bigoplus_{i\in I}}k[A_i/T]\xrightarrow{(\Theta_{A_i/T,B_j}(\Upsilon_{T,0}(f_{i,j})))_{i,T,j}}\bigoplus_{j\in J}k[B_j].$$
Il suffit donc de montrer que, pour tous $i\in I $ et $A'\subset A_i$, l'image de
$$\alpha_{A'} : k[A_i/A']\xrightarrow{(\Theta_{A_i/A',B_j}(\Upsilon_{T,0}(f_{i,j})))_j}\bigoplus_{j\in J}k[B_j]$$
est incluse dans la somme des images des morphismes $\alpha_T$ pour $T\subset A_i$ tel que $\lgr(A_i/t)=d$. Cela provient du lemme~\ref{lm-decD-db}, qui montre que
$$\alpha_{A'}=\underset{T\supset A'\,,\,\lgr(A/T)=d}{\sum_{T\subset A_i}}\Big(k[A_i/A']\twoheadrightarrow k[A_i/T]\xrightarrow{\alpha_T}\bigoplus_{j\in J}k[B_j]\Big)\;,$$
d'où la conclusion.
\end{proof}

\begin{lemm}\label{lm-ReducType} Soit $F$ un foncteur de $\F_d(\A;k)$ possédant une décomposition de type $D_d$ (resp. $E_d$). Alors $F$ est isomorphe dans $\F_d(\A;k)/\F_{d-1}(\A;k)$ à un foncteur possédant une décomposition de type $M_d$ (resp. $I_d$).
\end{lemm}

\begin{proof} Pour $f\in\mathrm{Hom}_d(A,B)$ et $T\subset A$ tel que $\lgr(A/T)=d$, on a $\Upsilon_{T,0}([f])=[\bar{f}]$ si $\mathrm{Ker}\,f=T$, $0$ sinon, où $\bar{f} : A/T\to B$ désigne le \emph{monomorphisme} induit par $f$, monomorphisme qui est un isomorphisme si $f$ est un épimorphisme. La conclusion découle donc du lemme~\ref{lm-regularis}. 
\end{proof}

\begin{proof}[Démonstration du théorème~\ref{th-princ}]
On montre $HR(d)$ par récurrence sur $d$, on peut donc supposer $HR(i)$ établi pour $i<d$ (condition qui est vide si $d=0$). Cela entraîne que $\Phi_d$ est pleinement fidèle par la proposition~\ref{pr-Spf}.\ref{it-pf}.

Montrons que $\Phi_d$ est essentiellement surjectif : il s'agit de montrer que tout foncteur $F$ de $\F_d(\A;k)$ devient isomorphe dans $\F_d(\A;k)/\F_{d-1}(\A;k)$ à un objet de l'image essentielle de $\Phi_d$. Si $F$ possède une décomposition de type $I_d$, la conclusion est donnée par le lemme~\ref{lm-Id}. Le cas où $F$ possède une décomposition de type $E_d$ se ramène au précédent grâce au lemme~\ref{lm-ReducType}. Le cas où $F$ possède une décomposition de type $M_d$ se ramène au type $E_d$ (cf. corollaire~\ref{cor-FdAutoDual} --- $F$ possède une décomposition de type $M_d$ dans $\F(\A;k)$ si et seulement si $F^\vee$ possède une décomposition de type $E_d$ dans $\F(\A^\op;k)$). Le cas où $F$ possède une décomposition de type $D_d$ se ramène au cas où il possède une décomposition de type $M_d$ par le lemme~\ref{lm-ReducType}. Maintenant, le lemme~\ref{lm-typeD} montre que $F$ appartient à l'image essentielle de $\Phi_d$ si $F$ est de type fini et de type cofini. Dans le cas général, on écrit $F$ comme colimite d'une limite de foncteurs qui sont à la fois de type fini et de type cofini et l'on utilise la bicontinuité de $\Phi_d$ (proposition~\ref{pr-bicont}) et sa pleine fidélité, établie ci-avant, pour conclure.

Finalement, le fait que le foncteur $S_d$ soit isomorphe au foncteur section  $\F_d(\A;k)/\F_{d-1}(\A;k)\to\F_d(\A;k)$ résulte du corollaire~\ref{cor-Fd_ferme} et de \cite[chap.~III, §\,2, corollaire de la prop.~3]{Gabriel} ; le fait que $C_d$ soit isomorphe au foncteur co-section se vérifie de façon analogue.
\end{proof}

\section{Foncteurs simples et propriétés de finitude}

\subsection{Description des foncteurs simples de $\F_d(\A;k)$}

\begin{nota}
Si $\E$ est une catégorie de Grothendieck, on note $\irr(\E)$ l'ensemble des classes d'isomorphisme d'objets simples de $\E$. La classe d'un tel objet $S$ dans $\irr(\E)$ sera notée $[S]$. 

On note $G_0^f(\E)$ le groupe de Grothendieck des objets finis de $\E$, c'est-à-dire le quotient du groupe abélien libre sur les classes d'isomorphisme d'objets finis de $\E$ par le sous-groupe engendré par $[Y]+[Z]-[X]$ pour chaque suite exacte courte $0\to Y\to X\to Z\to 0$ de $\E$ avec $X$ fini ; la classe d'un objet fini $X$ de $\E$ dans $G_0^f(\E)$ sera notée $\{X\}$. Ainsi, $\mathbb{Z}[\irr(\E)]\to G_0^f(\E)\quad [S]\mapsto\{S\}$ définit un isomorphisme de groupes.

Si $R$ est un anneau, on note simplement $\irr(R)$ (resp. $G_0^f(R)$) pour $\irr(R\Md)$ (resp. $G_0^f(R\Md)$).
\end{nota}

Le résultat suivant constitue l'une des conséquences les plus importantes du théorème~\ref{th-princ}.

\begin{theo}\label{th-simp_Fd}
\begin{enumerate}
    \item Si $A$ est un foncteur fini de $\rep(\A)$ et $M$ un $k[\mathrm{Aut}(A)]$-module simple, alors $\Q(A,M)$ est un foncteur simple de $\F(\A;k)$.
    \item Pour tout $d\in\mathbb{N}$, on dispose d'une bijection
    $$\underset{A\in\LL_n(\A)}{\bigsqcup_{n\le d}}\irr(k[\mathrm{Aut}(A)])\xrightarrow{\simeq}\irr(\F_d(\A;k))\qquad (n,A,[M])\mapsto [\Q(A,M)].$$
    \item Supposons que tout fonteur de type fini et de type cofini de $\rep(\A)$ est fini (par exemple, que $\rep(\A)$ est localement finie). Alors on dispose d'une bijection
    $$\underset{A\in\LL_n(\A)}{\bigsqcup_{n\in\mathbb{N}}}\irr(k[\mathrm{Aut}(A)])\xrightarrow{\simeq}\irr(\F(\A;k))\qquad (n,A,[M])\mapsto [\Q(A,M)].$$
\end{enumerate}
\end{theo}

\begin{proof}
Pour tout $A\in\LL_d(\A)$, le foncteur
$$\Q(A,-) : k[\mathrm{Aut}(A)]\Md\to\F_d(\A;k)$$
est isomorphe à la composée du foncteur d'inclusion $$k[\mathrm{Aut}(A)]\Md\hookrightarrow\prod_{B\in\LL_d(\A)}k[\mathrm{Aut}(B)]\Md\;,$$
de l'équivalence $\prod_{B\in\LL_d(\A)}k[\mathrm{Aut}(B)]\Md\simeq\F_d(\A;k)/\F_{d-1}(\A;k)$ du théorème~\ref{th-princ} et du \emph{prolongement intermédiaire} \cite[§\,4]{Ku2}
$$\F_d(\A;k)/\F_{d-1}(\A;k)\to\F_d(\A;k)$$
(cela résulte de la description des foncteurs section et co-section dans le théorème~\ref{th-princ} ainsi que de la définition de $\Q(A,M)$). Il s'ensuit qu'on dispose d'une bijection $\irr(\F_{d-1}(\A;k))\sqcup\underset{A\in\LL_d(\A)}{\bigsqcup}\irr(k[\mathrm{Aut}(A)])\xrightarrow{\simeq}\irr(\F_d(\A;k))$ dont la restriction à $\irr(\F_{d-1}(\A;k))$ est induite par l'inclusion $\F_{d-1}(\A;k)\to\F_d(\A;k)$ et dont la restriction à $\irr(k[\mathrm{Aut}(A)])$ est donnée par $[M]\mapsto [\Q(A,M)]$. On en déduit aussitôt la deuxième assertion par récurrence sur $d$. La première assertion en découle également puisque l'inclusion $\F_d(\A;k)\to\F(\A;k)$ préserve les objets simples.

La dernière assertion résulte de la deuxième et du corollaire~\ref{cor-tous_tftcf_ds_Fd}.
\end{proof}

\begin{rema}\label{rq-param_simples}
Si les idempotents se scindent dans $\A$, on dispose également, de façon élémentaire et classique (cf. par exemple \cite[prop.~1.15]{DTV}), d'une bijection explicite (également donnée par des prolongements intermédiaires) $$\bigsqcup\irr(k[\mathrm{Aut}_\A(a)])\xrightarrow{\simeq}\irr(\F(\A;k))\;,$$
où la réunion est prise sur un ensemble complet de représentants des classes d'isomorphisme d'objets $a$ de $\A$.

Lorsque tout foncteur de type fini et de type cofini de $\rep(\A)$ est fini et que les idempotents de $\A$ se scindent --- par exemple, lorsque $\A=\mathbf{P}(R)$ pour un anneau $k$-trivial $R$ --- on dispose donc de \emph{deux} paramétrisations différentes de $\irr(\F(\A;k))$ comme réunion disjointe d'ensembles de la forme $\irr(k[G])$ pour des groupes finis $G$. Ces paramétrisations ne sont pas équivalentes (sauf lorsque $\rep(\A)$ est semi-simple, comme on peut le déduire de \cite{Ku-adv}), et il semble difficile de comprendre la correspondance entre les deux (ce qui constitue un problème de représentations de groupes finis), bien qu'elles soient explicites.
\end{rema}

\begin{rema}\label{rq-simpathol2} S'il existe dans $\rep(\A)$ un foncteur de type fini et de type cofini mais pas fini, alors il existe dans $\F(\A;k)$ des foncteurs simples qui n'entrent pas dans la classification précédente (cf. remarque~\ref{rq-simpathol}). La classification de tels foncteurs simples semble hors d'atteinte.
\end{rema}

On dispose d'une autre base importante des groupes abéliens libres $G_0^f(\F_d(\A;k))$ que celle constituée des classes des $\Q(A,M)$ :

\begin{prop}\label{pr-G0Fd}
Pour tout $d\in\mathbb{N}$, le morphisme de groupes abéliens
$$\underset{A\in\LL_n(\A)}{\bigoplus_{n\le d}}G_0^f(k[\mathrm{Aut}(A)])\to G_0^f(\F_d(\A;k))$$
dont la composante $G_0^f(k[\mathrm{Aut}(A)])\to G_0^f(\F_d(\A;k))$ est $\{M\}\mapsto\{\Q_{A,M}\}$ est un isomorphisme.
\end{prop}

\begin{proof}
Cette proposition se déduit soit du théorème~\ref{th-simp_Fd} combiné à la proposition~\ref{pr-resol_QA} et ses corollaires, soit directement du théorème~\ref{th-princ} par récurrence sur $d$.
\end{proof}

\subsection{Finitude locale de $\F(\A;k)$ et de $\F_d(\A;k)$}\label{sfl}

\begin{prop}\label{pr-lin_fini} Soit $A$ un fonteur de $\rep(\A)$. Les assertions suivantes sont équivalentes :
\begin{enumerate}
    \item $A$ est un foncteur fini de $\rep(\A)$ ;
    \item $k[A]$ est un foncteur fini de $\F(\A;k)$.
\end{enumerate}

De plus, lorsqu'elles sont vérifiées, les facteurs de composition de $k[A]$ sont exactement les $\Q(B,M)$, où $B$ est un sous-quotient de $A$ et $M$ un $k[\mathrm{Aut}(B)]$-module simple.
\end{prop}

\begin{proof}
Comme le foncteur de linéarisation $\rep(\A)\to\F(\A;k)$ préserve les sous-objets stricts, il est clair que $A$ est nécessairement fini si $k[A]$ l'est. De plus, si $B$ est un sous-quotient de $A$ et $M$ un $k[\mathrm{Aut}(B)]$-module simple, alors $k[B]$, donc $\Q_B$, puis $\Q(B)$ et $\Q(B,M)$ sont des sous-quotients de $k[A]$.

On montre maintenant par récurrence sur l'entier $d$ que si $A$ est fini de longueur au plus $d$, alors $k[A]$ est fini et a ses facteurs de composition parmi les $\Q(B,M)$, où $B$ est un sous-quotient de $A$ et $M$ un $k[\mathrm{Aut}(B)]$-module simple.

Tout d'abord, le corollaire~\ref{cor-QAMequiv} montre que $\Q(A)$ est fini et a pour facteurs de composition les $\Q(A,M)$, où $M$ parcourt les $k[\mathrm{Aut}(A)]$-modules simples. La proposition~\ref{pr-QA_copres} et l'hypothèse de récurrence permettent d'en déduire que $\Q_A$ est fini et a ses facteurs de composition parmi les $\Q(B,M)$, où $B$ est un sous-quotient de $A$ et $M$ un $k[\mathrm{Aut}(B)]$-module simple.

On conclut alors en utilisant de nouveau l'hypothèse de récurrence et la suite exacte
$$\bigoplus_{B\subsetneq A} k[B]\to k[A]\to\Q_A\to 0$$
qui définit $\Q_A$.
\end{proof}

\begin{rema}
Ce résultat contraste grandement avec la situation d'égale caractéristique, c'est-à-dire celle où les groupes abéliens $\A(x,y)$ sont des $p$-groupes finis, où $p$ est la caractéristique de $k$. Il est alors facile de voir que, si $A$ est un foncteur fini de $\rep(\A)$, le foncteur $k[A]$ de $\F(\A;k)$ est de dimension de Krull au moins $\lgr(A)$, en particulier, il n'est fini que si $A$ est nul.
\end{rema}

\begin{coro}\label{cor-lf}
La catégorie $\F(\A;k)$ est localement finie si et seulement si $\rep(\A)$ est localement finie.
\end{coro}

\begin{proof}
La catégorie $\F(\A;k)$ (resp. $\rep(\A)$) est localement finie si et seulement si le foncteur $k[\A(x,-)]$ (resp. $\A(x,-)$) est fini pour tout $x\in\mathrm{Ob}\,\A$. La conclusion découle donc de la proposition~\ref{pr-lin_fini}.
\end{proof}

\begin{rema}\label{rq-SSn} En particulier, si $R$ est un anneau $k$-trivial, alors la catégorie $\F(R,k)$ est localement finie. Ce résultat est aussi obtenu dans \cite[pr.~11.7]{DTV} comme conséquence (facile) du corollaire~\ref{cor-dual_lin} et d'un théorème (difficile) de Putman-Sam-Snowden \cite{PSam,SamSn} affirmant que $\F(A,k)$ est localement noethérienne pour tout anneau fini $A$ (non nécessairement $k$-trivial).

On notera également qu'il est possible de déduire le corollaire~\ref{cor-lf} du cas particulier précédent à l'aide de résultats classiques d'Auslander sur les foncteurs finis de $\rep(\A)$.
\end{rema}

Pour donner d'autres conséquences de la proposition~\ref{pr-lin_fini}, nous aurons besoin du résultat suivant, qui constitue une amélioration du corollaire~\ref{cor-Fd-eq} :

\begin{prop}\label{pr-Fd-eng_loc} Pour tout $d\in\mathbb{N}$, $\F_d(\A;k)$ est la plus petite sous-catégorie localisante de $\F(\A;k)$ contenant $k[A]$ pour $A$ dans $\rep(\A)$ tel que $\lgr(A)\le d$.
\end{prop}

\begin{proof} Soit $\C_d$ la plus petite sous-catégorie localisante de $\F(\A;k)$ contenant $k[A]$ pour $\lgr(A)\le d$. On a clairement $\C_d\subset\F_d(\A;k)$. On montre l'inclusion inverse par récurrence sur $d$.

Le foncteur $C_d$ prend ses valeurs dans $\C_d$, puisque $\Q^A$ est un sous-foncteur de $k[A]$ et que le foncteur $M\mapsto\Q^{A,M}$ est cocontinu (par les propositions~\ref{pr-Qexact} et~\ref{pr-dQAM}). Le théorème~\ref{th-princ} fournit pour tout foncteur $F$ de $\F_d(\A;k)$ une suite exacte $C_d(\pi_d(F))\to F\to G\to 0$ avec $G$ dans $\F_{d-1}(\A;k)=\C_{d-1}\subset\C_d$ par l'hypothèse de récurrence, d'où $F\in\C_d$, ce qui termine la démonstration.
\end{proof}

\begin{coro}\label{cor-Kdim} Pour tout $d\in\mathbb{N}$, la catégorie de Grothendieck $\F_d(\A;k)$ est de dimension de Krull-Gabriel nulle, c'est-à-dire qu'elle est égale à sa plus petite sous-catégorie localisante contenant tous ses objets simples.
\end{coro}

\begin{proof}
Combiner les propositions~\ref{pr-Fd-eng_loc} et~\ref{pr-lin_fini}.
\end{proof}
\begin{prop}\label{pr-Fd_lf} Soit $d\in\mathbb{N}$.
\begin{enumerate}
\item Si $d\le 1$, alors la catégorie $\F_d(\A;k)$ est localement finie.
\item Si $d\ge 2$, les assertions suivantes sont équivalentes :
\begin{enumerate}
    \item[(a)] la catégorie $\F_d(\A;k)$ est localement finie ;
    \item[(b)] la catégorie $\F_d(\A;k)$ est localement noethérienne ;
    \item[(c)] la catégorie $\F_{d-1}(\A;k)$ est localement noethérienne et le foncteur $S_d$ est bicontinu ;
    \item[(d)] la catégorie $\F_{d-1}(\A;k)$ est localement noethérienne et le foncteur $\pi_d$ préserve les objets de type fini ;
    \item[(e)] la condition $(\mathrm{FE}(\A))$ (cf. notation~\ref{not-FE}) est vérifiée.
\end{enumerate}
\end{enumerate}
\end{prop}

\begin{proof}
La catégorie $\F_1(\A;k)/\F_0(\A;k)$ est localement finie grâce au théorème~\ref{th-princ} ; comme $\F_0(\A;k)$ est réduite aux foncteurs constants et que tout foncteur de $\F(\A;k)$ se scinde naturellement en la somme directe de son terme constant et de sa partie réduite, la catégorie
\begin{equation}\label{eq-F1}
\F_1(\A;k)\simeq (\F_1(\A;k)/\F_0(\A;k))\times\F_0(\A;k)\simeq (\F_1(\A;k)/\F_0(\A;k))\times (k\Md)
\end{equation}
est également localement finie, d'où la première assertion.

Nous allons maintenant établir la deuxième assertion par récurrence sur $d$.

L'implication $(a)\Rightarrow (b)$ est immédiate. Si $(b)$ est vérifié, alors la sous-catégorie localisante $\F_{d-1}(\A;k)$ de $\F_d(\A;k)$ est localement noethérienne, et le foncteur section, qui s'identifie à $S_d$ par le théorème~\ref{th-princ}, commute nécessairement aux colimites filtrantes \cite[chap.~III, §\,4, cor.~1 de la prop.~9]{Gabriel}. Comme $S_d$ est exact et continu (proposition~\ref{pr-bicont}), il est bicontinu. Ainsi, $(b)$ entraîne $(c)$.

L'implication $(c)\Rightarrow (d)$ résulte de ce que $\pi_d$ est adjoint à gauche à un foncteur isomorphe au foncteur $S_d$.

Si la catégorie $\F_{d-1}(\A;k)$ est localement noethérienne, alors elle est nécessairement finie grâce à l'hypothèse de récurrence, ou à la première assertion de la proposition si $d=2$.

Par conséquent, l'équivalence $(c)\Leftrightarrow (e)$ découle de la proposition~\ref{pr-bicont}.

Supposons $(d)$ vérifié. Comme $\F_d(\A;k)$ est localement de type fini (comme $\F(\A;k)$), pour établir $(a)$, il suffit de montrer que tout foncteur de type fini $F$ de $\F_d(\A;k)$ est fini. L'objet $\pi_d(F)$ de $\F_d(\A;k)/\F_{d-1}(\A;k)$ est de type fini, donc fini puisque cette catégorie est localement finie, par le théorème~\ref{th-princ}. Ce même théorème fournit une suite exacte $C_d(\pi_d(F))\to F\to G\to 0$ avec $G$ dans $\F_{d-1}(\A;k)$. Quotient de $F$, $G$ est de type fini, donc fini puisque $\F_{d-1}(\A;k)$ est localement finie grâce à l'observation ci-dessus. Par ailleurs, la finitude de $\pi_d(F)$ et la proposition~\ref{pr-lin_fini} montrent que $C_d(\pi_d(F))$ est fini. Ainsi $F$ est fini, d'où l'implication $(d)\Rightarrow (a)$ et la proposition.
\end{proof}

\begin{rema}
La proposition~\ref{pr-Fd_lf} et la remarque~\ref{rq-pasFE} montrent que la catégorie $\F_2(\A;k)$ n'est pas nécessairement localement finie.
\end{rema}

\subsection{Fonctions de dimensions des foncteurs finis}

La proposition~\ref{pr-G0Fd} permet de résoudre une conjecture de \cite{DTV}, dans un cadre un peu plus général. Nous introduisons en préalable quelques notations :

\begin{nota}
Soit $k_0(\A)$ le monoïde commutatif des classes d'isomorphisme d'objets de $\A$ (pour la somme directe). Si $F$ est un foncteur de $\F(\A;k)$ prenant des valeurs de dimension finie (par exemple, un foncteur de type fini), on note $\mathrm{d}_F : k_0(\A)\to\mathbb{Q}$ la fonction envoyant la classe d'isomorphisme de $a\in\mathrm{Ob}\,\A$ sur la dimension du $k$-espace vectoriel $F(a)$.

On note $\chi(\A)$ le sous-ensemble de $\mathbb{N}^*$ constitué des entiers naturels qui divisent le cardinal d'un groupe abélien $\A(x,y)$. C'est un sous-monoïde multiplicatif de $\mathbb{N}^*$ puisque $\A$ est additive. On note $\Xi(\A)$ le sous-$\mathbb{Q}$-espace vectoriel de l'espace vectoriel $\mathbb{Q}^{k_0(\A)}$ des fonctions de $k_0(\A)$ dans $\mathbb{Q}$ engendré par les fonctions composées d'un morphisme de monoïdes du groupe abélien $k_0(\A)$ vers le monoïde multiplicatif $\chi(\A)$ et de l'inclusion $\chi(\A)\hookrightarrow\mathbb{Q}$. C'est donc une sous-algèbre de $\mathbb{Q}^{k_0(\A)}$, munie de la multiplication au but.

Les éléments de $\Xi(\A)$ sont appelés {\em $\chi$-polynômes}.
\end{nota}

\begin{rema}
Comme les anneaux d'endomorphismes des objets de $\A$ sont finis, le théorème classique de Krull-Schmidt implique que, si l'on suppose de plus que les idempotents se scindent dans $\A$, alors $k_0(\A)$ est un monoïde commutatif libre. C'est le cas en particulier lorsque $\A=\mathbf{P}(R)$ (où $R$ est un anneau $k$-trivial) ; le groupe de $K$-théorie $K_0(R)$ s'identifie alors à la complétion en groupes de $k_0(\mathbf{P}(R))$. On dispose d'un morphisme de monoïdes canonique $\mathbb{N}\to k_0(\mathbf{P}(R))$ (associant à $n$ la classe du module libre $R^n$) qui est injectif si $R$ est non nul et bijectif si $R$ est local.
\end{rema}

\begin{rema}\label{rq-pol1}
Les définitions introduites ci-avant font sens pour une catégorie additive $\A$ quelconque (non nécessairement $k$-triviale). Dans ce cadre général, on notera qu'un foncteur $F$ de $\F(\A;k)$ à valeurs de dimensions finies est polynomial de degré $d$ si et seulement si la fonction $\mathrm{d}_F : k_0(\A)\to\mathbb{Q}$ est polynomiale de degré $d$ au sens d'Eilenberg-MacLane \cite[§\,8]{EML}, parce que $\mathrm{d}_{\delta_x(F)}([a])=\mathrm{d}_F([x]+[a])-\mathrm{d}_F([a])$ pour tous objets $x$ et $a$ de $\A$.

En revanche, la fonction de dimensions de $\db_x(F)$ ne s'exprime pas comme une fonction de celle de $F$, ce qui rend le résultat ci-dessous (dans lequel on revient au cadre où la catégorie $\A$ est $k$-triviale) non trivial.
\end{rema}

\begin{theo}\label{th-fct_dim}
Si $F$ est un foncteur fini appartenant à une catégorie $\F_d(\A;k)$, alors la fonction $\mathrm{d}_F$ est un $\chi$-polynôme.
\end{theo}

\begin{proof}
Si $A$ est un foncteur fini de $\rep(\A)$, alors $\mathrm{d}_{k[A]}$ est un morphisme de monoïdes du groupe abélien $k_0(\A)$ vers le monoïde multiplicatif sous-jacent à $\mathbb{Q}$ ; de plus, ses valeurs appartiennent à $\chi(\A)$, car $A$ est quotient d'un foncteur représentable $\A(a,-)$. Ainsi $\mathrm{d}_{k[A]}\in\Xi(\A)$. La proposition~\ref{pr-ressimpl} permet d'en déduire la relation $\mathrm{d}_{\Q_A}\in\Xi(\A)$. Comme le groupe $\mathrm{Aut}(A)$ opère librement sur $\Q_\A$ (cf. proposition~\ref{pr-Qexact}), si $M$ est une représentation de $\mathrm{Aut}(A)$ de dimension finie sur $k$, on a
\begin{equation}\label{eq-dimQ}
\mathrm{d}_{\Q_{A,M}}=\frac{\dim_k(M)}{\cd(\mathrm{Aut}(A))}\,\mathrm{d}_{\Q_A}\;,
\end{equation}
d'où $\mathrm{d}_{\Q_{A,M}}\in\Xi(\A)$. La conclusion résulte maintenant de la proposition~\ref{pr-G0Fd}, puisque $\{F\}\mapsto\mathrm{d}_F$ définit un morphisme de groupes abéliens $G_0^f(\F_d(\A;k))\to\mathbb{Q}^{k_0(\A)}$.
\end{proof}

\begin{rema}\label{rq-pol12}
Revenons dans cette remarque au cas où la catégorie additive $\A$ est arbitraire (non nécessairement $k$-triviale), et supposons que tout foncteur de type fini et de type cofini de $\rep(\A/\I)$ est fini lorsque $\I$ est un idéal $k$-cotrivial de $\A$ (par exemple, $\A=\mathbf{P}(R)$ pour un anneau arbitraire $R$). On déduit du théorème précédent, du corollaire~\ref{cor-tous_tftcf_ds_Fd}, de la remarque~\ref{rq-pol1} sur les fonctions de dimensions des foncteurs polynomiaux et de \cite[cor.~4.11]{DTV} que, si $F$ est un foncteur fini à valeurs de dimensions finies de $\F(\A;k)$, alors il existe une fonction $P : k_0(\A)\times k_0(\A)\to\mathbb{Q}$ qui est polynomiale par rapport à la première variable et $\chi$-polynomiale par rapport à la seconde telle que $\mathrm{d}_F(x)=P(x,x)$ pour tout $x\in k_0(\A)$.
\end{rema}

\begin{coro}\label{cor-conj_DTV} Soient $p$ un nombre premier différent de la caractéristique de $k$ et $R$ un anneau dont le cardinal est une puissance de $p$. Si $F$ est un foncteur de type fini de $\F(R,k)$, alors il existe un polynôme $P\in\mathbb{Q}[X]$ tel que $\dim_k(F(R^n))=P(p^n)$ pour tout $n\in\mathbb{N}$. 
\end{coro}

Ce résultat, qui découle du théorème~\ref{th-fct_dim} et des corollaires~\ref{cor-tttfFd} et~\ref{cor-lf}, démontre la conjecture~6.8 de \cite{DTV} --- ou plus précisément, une version légèrement affaiblie où l'on remplace une fonction polynomiale à valeurs entières par une fonction polynomiale à valeurs rationnelles (la version à valeurs entières de \cite[conj.~6.8]{DTV} n'est pas exacte à cause du dénominateur qui apparaît dans \eqref{eq-dimQ}).

\section{Propriétés homologiques des catégories $\F_d(\A;k)$}\label{shom}

Pour tout $d\in\mathbb{N}$, la catégorie $\F_d(\A;k)$ est une catégorie de Grothendieck ; en particulier, elle possède des enveloppes injectives. On a également :
\begin{prop}\label{pr-proj} Soit $d\in\mathbb{N}$. La catégorie $\F_d(\A;k)$ possède assez d'objets projectifs. De plus, tout objet de type fini de $\F_d(\A;k)$ possède une couverture projective.
\end{prop}

\begin{proof}
Soit $F$ un foncteur de type fini de $\F_d(\A;k)$. Alors $F^\vee$ est un objet de type cofini de $\F_d(\A^\op;k)$ (corollaire~\ref{cor-FdAutoDual}). Par conséquent, si $J$ est une enveloppe injective de $F^\vee$ dans $\F_d(\A^\op;k)$, alors $J$ est de type cofini, donc un foncteur à valeurs de dimensions finies. Il s'ensuit que $J^\vee$ est une couverture projective de $F$ dans $\F(\A;k)$. Comme la catégorie $\F_d(\A;k)$ est localement de type fini, l'existence de couvertures projectives pour ses objets de type fini entraîne qu'elle possède assez de projectifs.
\end{proof}

\begin{rema}
Un ensemble de générateurs projectifs de $\F_d(\A;k)$ s'obtient, pour des raisons formelles, en appliquant l'adjoint à gauche de l'inclusion $\F_d(\A;k)\hookrightarrow\F(\A;k)$ aux foncteurs $k[\A(a,-)]$ (c'est-à-dire en considérant leur plus grand quotient appartenant à $\F_d(\A;k)$), où $a$ parcourt un squelette de $\A$. Nous ne connaissons toutefois pas de description simple de cet adjoint.

Néanmoins, en utilisant le théorème~\ref{th-princ}, on voit que $\Q^A$ (resp. $\Q_A$) est projectif (resp. injectif) dans $\F_d(\A;k)$ lorsque $A$ est un foncteur de $\rep(\A)$ tel que $\lgr(A)=d$.
\end{rema}

Le prochain résultat de cette section sur les catégories $\F_d(\A;k)$ (corollaire~\ref{cor-plg_hom}) dépend de la propriété générale suivante des catégories abéliennes.

\begin{prop}\label{pr-locExt} Soient $\F$ une catégorie de Grothendieck, $\E$ une sous-catégorie localisante de $\F$ possédant assez d'objets projectifs et $\E'$ une sous-catégorie bilocalisante de $\E$. Supposons que :
\begin{enumerate}
    \item[$1$.] l'inclusion $\E'\to\F$ est un plongement homologique ;
    \item[$2$.] le foncteur section $s: \E/\E'\to\E$ est exact ;
    \item[$2^*$.] le foncteur co-section $c: \E/\E'\to\E$ est exact ;
    \item[$3$.] le foncteur $\E/\E'\to\F$ composé de $s$ et de  l'inclusion $\E\to\F$ est un plongement homologique ;
    \item[$3^*$.] le foncteur $\E/\E'\to\F$ composé de $c$ et de  l'inclusion $\E\to\F$ est un plongement homologique ;
    \item[$4$.] Si $X$ est un objet de $\E'$ et $T$ un objet de $\E/\E'$, alors $\mathrm{Ext}^*_\F(X,s(T))=0$ ;
    \item[$4^*$.] Si $X$ est un objet de $\E'$ et $T$ un objet de $\E/\E'$, alors $\mathrm{Ext}^*_\F(c(T),X)=0$.
\end{enumerate}

Alors l'inclusion $\E\to\F$ est un plongement homologique.
\end{prop}

\begin{proof}
On montre d'abord que l'inclusion $\E'\to\E$ est un plongement homologique. Soit $\alpha : \E\to\E'$ l'adjoint à droite de l'inclusion : les objets $\alpha(J)$, où $J$ est un objet injectif de $\E'$, forment une classe de cogénérateurs injectifs de $\E'$, il suffit donc de vérifier que $\mathrm{Ext}^n_\E(X,\alpha(J))=0$ pour tout $n\in\mathbb{N}^*$, tout objet $X$ de $\E'$ et tout objet injectif $J$ de $\E$.

De plus, si l'on note $\pi : \E\to\E/\E'$ le foncteur canonique, on dispose pour tout objet $T$ de $\E$ d'une suite exacte naturelle
\begin{equation}\label{eq-Ext2}
    0\to\alpha(T)\to T\to s\pi(T)\to\beta(T)\to 0
\end{equation}
où $\beta(T)$ appartient à $\E'$ (cf. par exemple \cite[prop.~2.6\,(ii)]{Psa-hom}). Notons $h(T)$ la classe dans $\mathrm{Ext}^2_\E(\beta(T),\alpha(T))$ de l'extension \eqref{eq-Ext2}. Si $n>0$ est un entier et $J$ un objet injectif de $\E$, on a $\mathrm{Ext}^n_\E(-,J)=0$, et la restriction à $\E'$ de $\mathrm{Ext}^n_\E(-,s\pi(J))$ est également nulle en raison de l'isomorphisme
\begin{equation}\label{eq-isadj}
    \mathrm{Ext}^n_\E(X,s\pi(J))\simeq\mathrm{Ext}^n_{\E/\E'}(\pi(X),\pi(J))
\end{equation}
naturel en l'objet $X$ de $\E$ déduit de l'adjonction entre les foncteurs exacts $\pi$ et $s$ (hypothèse $2$.). Il s'ensuit que le produit de composition par $h(J)$ définit un isomorphisme
\begin{equation}\label{eq-isoExt}
    \cdot h(J) : \mathrm{Ext}^n_\E(X,\beta(J))\xrightarrow{\simeq}\mathrm{Ext}^{n+2}_\E(X,\alpha(J))
\end{equation}
pour tout objet $X$ de $\E'$, tout objet injectif $J$ de $\E$ et tout entier $n>0$. Mais comme $\E$ est une sous-catégorie épaisse de $\F$, le morphisme canonique $\mathrm{Ext}^i_\E(X,\alpha(J))\to\mathrm{Ext}^i_\F(X,\alpha(J))$ est bijectif pour $i\le 1$ et injectif pour $i=2$. Or $\mathrm{Ext}^i_\F(X,\alpha(J))\simeq\mathrm{Ext}^i_{\E'}(X,\alpha(J))$ en vertu de la l'hypothèse 1., donc pour $J$ injectif dans $\E$ et $X$ dans $\E'$, on a $\mathrm{Ext}^i_\E(X,\alpha(J))=0$ pour $i\in\{1,2\}$. En particulier, $h(J)=0$, de sorte que \eqref{eq-isoExt} montre qu'on a aussi $\mathrm{Ext}^{n+2}_\E(X,\alpha(J))=0$ pour $n>0$. Finalement, on a bien $\mathrm{Ext}^i_\E(X,\alpha(J))=0$ pour tout entier $i>0$ lorsque $X$ est dans $\E'$ et que $J$ est un injectif de $\E$. Cela achève de montrer que l'inclusion $\E'\to\E$ est un plongement homologique.

En utilisant l'hypothèse 1., cela montre que le morphisme canonique $\mathrm{i}(X,Y) : \mathrm{Ext}^*_\E(X,Y)\to\mathrm{Ext}^*_\F(X,Y)$ est bijectif lorsque $X$ et $Y$ appartiennent à $\E'$. 

Par ailleurs, $\mathrm{i}(X,Y)$ est bijectif pour $Y$ dans l'image essentielle de $s$, disons $Y\simeq s(A)$, grâce à la succession d'isomorphismes
$$\mathrm{Ext}^*_\E(X,s(A))\simeq\mathrm{Ext}^*_{\E/\E'}(\pi(X),A)\simeq\mathrm{Ext}^*_\F(s\pi(X),s(A))\simeq\mathrm{Ext}^*_\F(X,s(A))$$
dont le premier se déduit de l'hypothèse $2$., le second de l'hypothèse $3$. et le troisième de l'hypothèse $4$. et de la suite exacte \eqref{eq-Ext2} (appliquée avec $T=X$).

Maintenant, $\mathrm{i}(X,Y)$ est bijectif lorsque $X$ appartient à $\E'$, comme il résulte du cas où $Y$ est aussi dans $\E'$, de celui où $Y$ appartient à l'image essentielle de $s$, tous deux précédemment traités, et de la suite exacte \eqref{eq-Ext2} (appliquée avec $T=Y$).

On montre de façon duale (en utilisant les hypothèses $2^*$, $3^*$ et $4^*$ au lieu de $2$, $3$ et $4$ respectivement) que $\mathrm{i}(X,Y)$ est bijectif lorsque $Y$ appartient à $\E'$.

Le cas général se déduit de celui où $Y$ appartient à $\E'$ et de celui où $Y$ appartient à l'image essentielle du foncteur $s$ à l'aide de la suite exacte \eqref{eq-Ext2} (appliquée avec $T=Y$), d'où la proposition.
\end{proof}

\begin{coro}\label{cor-plg_hom}
Pour tout $d\in\mathbb{N}$, l'inclusion $\F_d(\A;k)\to\F(\A;k)$ est un plongement homologique.
\end{coro}

\begin{proof} Raisonnant par récurrence sur $d$ (l'assertion est vide pour $d<0$ et triviale pour $d=0$), on peut peut supposer établi que l'inclusion $\F_{d-1}(\A;k)\to\F(\A;k)$ est un plongement homologique.

Montrons que les hypothèses de la proposition~\ref{pr-locExt} sont vérifiées avec $\F=\F(\A;k)$, $\E=\F_d(\A;k)$ et $\E'=\F_{d-1}(\A;k)$. Tout d'abord, $\F_d(\A;k)$ est une sous-catégorie bilocalisante de $\F(\A;k)$ avec assez de projectifs et $\F_{d-1}(\A;k)$ une sous-catégorie bilocalisante de $\F_d(\A;k)$ grâce aux propositions~\ref{pr-Fdbiloc} et~\ref{pr-proj}. La première hypothèse de la proposition~\ref{pr-locExt} n'est autre que notre hypothèse de récurrence ; nous nous bornerons dans la suite à vérifier les hypothèses 2, 3 et 4 de la proposition~\ref{pr-locExt}, les hypothèses $2^*$, $3^*$ et $4^*$ se prouvant de façon entièrement analogue.

L'hypothèse 2. est vérifiée grâce au théorème~\ref{th-princ} et à la proposition~\ref{pr-Qexact}. La proposition~\ref{pr-Spf}.\ref{it-ph}  et le théorème~\ref{th-princ} entraînent que l'hypothèse 3. est également vérifiée. Enfin,
la proposition~\ref{pr-Ext_HR} et le théorème~\ref{th-princ} montrent que l'hypothèse 4. est vérifiée, d'où la conclusion.
\end{proof}

\begin{rema}
La proposition précédente est à mettre en regard du comportement homologique de l'inclusion des foncteurs polynomiaux d'un certain degré $d$ dans la catégorie de tous les foncteurs depuis une petite catégorie additive (qu'on ne suppose pas $k$-triviale !) vers les $k$-espaces vectoriels : si $k$ est de caractéristique nulle, l'inclusion est toujours un plongement homologique, mais ce n'est pas le cas en général, cf. \cite{D-Ext_pol}. En revanche, l'inclusion des foncteurs polynomiaux de degré $d$ dans la catégorie de tous les foncteurs sur les groupes libres de type fini est toujours un plongement homologique, cf. \cite{DPV}.
\end{rema}

On rappelle que la \emph{dimension globale} d'une catégorie de Grothendieck $\E$ est l'élément $\operatorname{gldim}(\E)$ de $\mathbb{N}\cup\{\infty\}$ défini par
$$\forall n\in\mathbb{N}\qquad(n>\operatorname{gldim}(\E))\Leftrightarrow(\mathrm{Ext}^n_\E=0)\;.$$

\begin{prop}\label{pr-gldim}
Supposons le corps $k$ de caractéristique nulle. Alors pour tout entier $d\ge 1$, la catégorie $\F_d(\A;k)$ est de dimension globale finie, au plus égale à $2(d-1)$.
\end{prop}

\begin{proof} Le recollement de catégories abéliennes
$$(\F_{d-1}(\A;k),\F_d(\A;k),\F_d(\A;k)/\F_{d-1}(\A;k))$$
vérifie l'hypothèse (iv) de la proposition~4.15 de \cite{Psa-hom}. En effet, le foncteur d'inclusion $\F_{d-1}(\A;k)\hookrightarrow\F_d(\A;k)$ est un plongement homologique grâce au corollaire~\ref{cor-plg_hom}, et le foncteur section $\F_d(\A;k)/\F_{d-1}(\A;k))\to\F_d(\A;k)$ est exact par le théorème~\ref{th-princ} et la proposition~\ref{pr-Qexact}. Il suffit alors d'utiliser \cite[prop.~4.12]{Psa-hom} pour voir que \cite[prop.~4.15]{Psa-hom} s'applique à notre recollement, d'où
$$\operatorname{gldim}(\F_d(\A;k))\le\operatorname{gldim}(\F_{d-1}(\A;k))+\operatorname{gldim}(\F_d(\A;k)/\F_{d-1}(\A;k))+2\;,$$
or $\operatorname{gldim}(\F_d(\A;k)/\F_{d-1}(\A;k))=0$ lorsque $k$ est de caractéristique nulle, par le théorème~\ref{th-princ}. Celui-ci, combiné à l'équivalence \eqref{eq-F1}, fournit également l'égalité $\operatorname{gldim}(\F_1(\A;k))=0$, ce qui permet de conclure la démonstration par récurrence sur $d$.
\end{proof}

\begin{rema}
La proposition précédente est comparable à un résultat sur les foncteurs polynomiaux de degré au plus $d$ des groupes libres de type fini vers les $k$-espaces vectoriels : lorsque $k$ est de caractéristique nulle, ils forment une catégorie de dimension globale $d-1$, cf. \cite[prop.~4.6]{DPV}. Toutefois, contrairement à la situation de \cite{DPV}, il ne semble pas facile de préciser la proposition~\ref{pr-gldim} en déterminant la dimension globale \emph{exacte} de $\F_d(\A;k)$, qui dépend de toute façon de $\A$. Par exemple, pour $d>1$, cette dimension globale est nulle si et seulement si la catégorie $\rep(\A)$ est semi-simple, comme on le déduit de \cite[cor.~1.3]{Ku-adv} (et de l'exemple~\ref{ex-Qsom} pour le sens réciproque).
\end{rema}

\appendix

\section{Rappels sur les catégories abéliennes}\label{ap-catab}

Cet appendice rassemble quelques définitions et propriétés classiques dans les catégories abéliennes ; on pourra se reporter pour davantage de détails à \cite{Gabriel,Pop}, notamment.

\paragraph*{Sous-catégories (bi)localisantes et recollements}

 Une \emph{catégorie de Grothendieck} est une catégorie abélienne cocomplète où les colimites filtrantes sont exactes et possédant un générateur ; une telle catégorie est complète et admet des enveloppes injectives, on peut en particulier y faire de l'algèbre homologique. Toutes les catégories de foncteurs qui apparaissent dans cet article sont des catégories de Grothendieck.

Une sous-catégorie $\E'$ d'une catégorie abélienne $\E$ est dite \emph{épaisse} si elle est pleine et stable par sous-objet, quotient et extensions ; on peut alors définir la catégorie quotient $\E/\E'$, elle-même abélienne, qui est la localisation obtenue en inversant formellement les morphismes de $\E$ dont le noyau et le conoyau appartiennent à $\E'$. On dispose alors d'un foncteur exact canonique $\E\to\E/\E'$ ; $\E'$ est dite \emph{localisante} (resp. \emph{colocalisante}) si ce foncteur possède un adjoint à droite (resp. un adjoint à gauche), nommé foncteur section (resp. co-section). Une sous-catégorie épaisse est dite \emph{bilocalisante} si elle est à la fois localisante et colocalisante. Une sous-catégorie épaisse d'une catégorie de Grothendieck est localisante si et seulement si elle est stable par coproduits arbitraires ; c'est alors une catégorie de Grothendieck, de même que $\E/\E'$. Une sous-catégorie épaisse d'une catégorie de Grothendieck est colocalisante si et seulement si elle est bilocalisante, ce qui équivaut encore à dire qu'elle stable par produits arbitraires.

Un \emph{recollement} $(\E',\E,\E/\E')$ de catégories abéliennes est la donnée d'une catégorie abélienne $\E$, d'une sous-catégorie bilocalisante $\E'$ et des différents foncteurs canoniques reliant les catégories $\E'$, $\E$ et $\E/\E'$ ; on renvoie à \cite[§\,2]{Psa-hom} pour davantage de détails à ce propos. Les recollements de catégories abéliennes semblent être apparus explicitement pour la première fois dans l'article \cite{Ku2}, qui traite des catégories $\F(k,k)$ pour un corps fini $k$. L'une des utilités des recollements provient de ce que, si $\E'$ est une sous-catégorie bilocalisante d'une catégorie abélienne $\E$, on peut décrire les objets \emph{simples} (i.e. non nuls mais sans sous-objet non trivial) de $\E$ à partir de ceux de $\E'$ et de $\E/\E'$.

\paragraph*{Propriétés de finitude dans les catégories abéliennes (Cf. \cite[§\,5.7]{Pop})}

Un objet $X$ d'une catégorie de Grothendieck est dit \emph{de type fini} si le foncteur $\mathrm{Hom}(X,-)$ commute aux colimites filtrantes de monomorphismes ; cela équivaut à dire que toute famille filtrante croissante de sous-objets dont la borne supérieure est $X$ contient $X$ \cite[chap.~3, prop.~5.6]{Pop}.

On dit qu'un objet $X$ d'une catégorie abélienne est \emph{de type cofini} si toute famille filtrante décroissante de sous-objets de $X$ d'intersection $0$ contient $0$. Un objet est dit noethérien (resp. artinien) si toute suite croissante (resp. décroissante) de sous-objets stationne. Un objet est dit de longueur finie, ou simplement \emph{fini}, s'il est artinien et noethérien. Un objet est dit \emph{simple} s'il est non nul et ne contient aucun sous-objet strict non nul. Il est classique d'un objet $X$ est fini si et seulement s'il possède une filtration finie dont les sous-quotients sont simples ; la longueur d'une telle filtration est appelée simplement \emph{longueur} de $X$, nous la noterons $\lgr(X)$. Ainsi $\lgr(X)=0$ (resp. $1$) si et seulement si $X$ est nul (resp. simple). Si $X$ n'est pas fini, on pose par convention $\lgr(X)=\infty$.

Une catégorie abélienne est dite \emph{localement noethérienne} (resp. \emph{localement finie}) si elle est engendrée par ses objets noethériens (resp. de type fini).

Le résultat suivant se montre facilement par récurrence sur $\lgr(X)$.

\begin{prop}\label{sousobjets-nbfini} Soient $\E$ une catégorie abélienne et $X$ un objet de $\E$. On suppose que $X$ est fini et que tous les sous-quotients simples de $X$ ont des corps d'endomorphismes finis. Alors l'ensemble des sous-objets de $X$ est fini.
\end{prop}

\backmatter

\bibliographystyle{smfplain}
\bibliography{bib-car-cr.bib}
 
 \end{document}